%% file: ripple-arxiv-v2.tex
\DeclareSymbolFont{AMSb}{U}{msb}{m}{n}
\documentclass[reqno,centertags,11pt,a4paper,noamsfonts]{amsart}
\pdfoutput=1
\usepackage[english]{babel}
\usepackage[babel=true,expansion=alltext,protrusion=alltext-nott,final]{microtype}
\usepackage[margin={3cm},marginparwidth={2.5cm}]{geometry}
\usepackage[svgnames]{xcolor}
\usepackage[debug=true]{hyperref}
\hypersetup{
	linktoc=all,colorlinks=true,linkcolor=Green,citecolor=Orange,urlcolor=DarkBlue,
	pdftitle={Variational methods for a singular SPDE yielding the universality of the magnetization ripple},
	pdfauthor={R. Ignat, F. Otto, T. Ried, P. Tsatsoulis}, pdfsubject={60H17, 35J60, 78A30, 82D40},
	pdfkeywords={singular stochastic PDE, nonlocal elliptic PDE, regularity theory, renormalized energy, Gamma-convergence, micromagnetics, Burgers equation}} 
\usepackage[utf8]{inputenc}
\usepackage{libertine}
\usepackage[libertine]{newtxmath}
\useosf
\usepackage[varqu,varl]{inconsolata}
\usepackage[T1]{fontenc}
\usepackage{textcomp}
\usepackage{amsthm}
\usepackage{mhequ} 
\usepackage[inline]{enumitem}
\numberwithin{equation}{section}
\usepackage[normalem]{ulem}
\usepackage{tikzsymbols}
\providecommand{\mr}[1]{\href{http://www.ams.org/mathscinet-getitem?mr=#1}{MR~#1}}
\providecommand{\zbl}[1]{\href{https://zbmath.org/?q=an:#1}{Zbl~#1}}
\providecommand{\arxiv}[1]{\href{https://arxiv.org/abs/#1}{arXiv:#1}}
\providecommand{\doi}[2]{\href{https://doi.org/#1}{#2}}
\newcommand{\RR}{\mathbb{R}}

\newcommand{\NN}{\mathbb{N}}
\newcommand{\ZZ}{\mathbb{Z}}

\newcommand{\EE}{\mathcal{E}}
\newcommand{\HH}{\mathcal{H}}
\newcommand{\TT}{\mathbb{T}}
\renewcommand{\SS}{\mathbb{S}}
\newcommand{\LL}{\mathcal{L}}
\newcommand{\GG}{\mathcal{G}}
\newcommand{\C}{\mathcal{C}}

\newcommand{\ii}{\mathrm{i}}
\newcommand{\dd}{\mathrm{d}}
\newcommand{\dotB}{\dot{\mathcal{B}}}

\newcommand{\Hnorm}[2]{\|\kern-.2ex| #2 \|\kern-.2ex|_{#1}}

\newcommand{\eps}{\varepsilon}
\newcommand{\lng}{\langle}
\newcommand{\rng}{\rangle}

\newcommand{\triple}{\mathcal{T}}

\DeclareMathOperator{\sgn}{sgn}

\theoremstyle{plain}
\newtheorem{theorem}{Theorem}[section]
\newtheorem{proposition}[theorem]{Proposition}
\newtheorem{corollary}[theorem]{Corollary}
\newtheorem{lemma}[theorem]{Lemma}
\newtheorem{assumption}[theorem]{Assumption}
\newtheorem*{theorem*}{Theorem}

\theoremstyle{definition}
\newtheorem{definition}[theorem]{Definition}
\newtheorem{remark}[theorem]{Remark}

\begin{document}
\title[Variational methods for a singular SPDE]{Variational methods for a singular SPDE yielding the 
universality of the magnetization ripple}

\author[R.~Ignat]{Radu Ignat}
\address[R.~Ignat]{Institut de Math\'ematiques de Toulouse \& Institut Universitaire de France, UMR 5219, Universit\'e de Toulouse, CNRS, UPS
IMT, F-31062 Toulouse Cedex 9, France}
\email{Radu.Ignat@math.univ-toulouse.fr}

\author[F.~Otto]{Felix Otto}
\address[F.~Otto]{Max-Planck-Institut für Mathematik in den Naturwissenschaften, Inselstraße 22, 04103 Leipzig, Germany}
\email{Felix.Otto@mis.mpg.de}

\author[T.~Ried]{Tobias Ried}
\address[T.~Ried]{Max-Planck-Institut für Mathematik in den Naturwissenschaften, Inselstraße 22, 04103 Leipzig, Germany}
\email{Tobias.Ried@mis.mpg.de}

\author[P.~Tsatsoulis]{Pavlos Tsatsoulis}
\address[P.~Tsatsoulis]{Max-Planck-Institut für Mathematik in den Naturwissenschaften, Inselstraße 22, 04103 Leipzig, Germany}
\email{Pavlos.Tsatsoulis@mis.mpg.de}

\date{\today, version \jobname}
\thanks{\textcopyright 2022 by the authors. Faithful reproduction of this article, in its entirety, by any means is permitted for noncommercial purposes.}
\begin{abstract}
The magnetization ripple is a microstructure formed in thin ferromagnetic films. It can be described by 
minimizers of a nonconvex energy functional leading to a nonlocal and nonlinear elliptic 
SPDE in two dimensions driven by white noise, which is singular. We address the universal character of the magnetization
ripple using variational methods based on $\Gamma$-convergence. Due to the infinite energy of the system, the (random) energy functional has to be renormalized. 
Using the topology of $\Gamma$-convergence, we give a sense to the law of the renormalized functional that is independent of the way white noise is approximated.
More precisely, this universality holds in the class of (not necessarily Gaussian) approximations to white noise 
satisfying the spectral gap inequality, which allows us to obtain sharp stochastic estimates. As a corollary, we 
obtain the existence of minimizers with optimal regularity.

\end{abstract}

\maketitle
{
\setcounter{tocdepth}{1}
\tableofcontents}
\setcounter{tocdepth}{2} 
\input{introduction-6}
\input{burgers}
\input{gamma-convergence}
\input{regularity}
\input{spectralgap-6}
{
\addtocontents{toc}{\setcounter{tocdepth}{-10}}

\section*{Acknowledgements}

The authors would like to thank the Isaac Newton Institute for Mathematical Sciences, Cambridge, for support and hospitality during the programme ``The mathematical design of new materials'' where work on this paper was completed. This programme was supported by EPSRC grant no EP/K032208/1. 
F.O., T.R., and P.T. also thank the Centre International de Mathématiques et d’Informatique de Toulouse (CIMI) 
and R.I. thanks the Max Planck Institute for Mathematics in the Sciences for their kind hospitality.  

\addtocontents{toc}{\setcounter{tocdepth}{1}}
}

\appendix
\hypertarget{appendix}{}

\input{app-hoelder}
\input{app-besov}
\input{app-stochastic-2}
\input{app-linear}
\input{app-regularity}
\input{app-cylinder}

\input{bibliography-2}
\end{document}

%% file: introduction-6.tex
\section{Introduction}\label{sec:introduction}

We study minimizers of the energy functional 
\begin{align}\label{eq:total-energy}
	E_{tot}(u) &:= \int_{\TT^2} (\partial_1 u)^2\,\mathrm{d}x +\int_{\TT^2}(|\partial_1|^{-\frac{1}{2}} (\partial_2 u - \partial_1 \frac{1}{2} u^2))^2\,\mathrm{d}x -2\sigma\int_{\TT^2}\xi u\,\mathrm{d}x
\end{align} 
where $\xi$ is (periodic) white noise, $\sigma\in \RR$, $\TT^2=[0,1)^2$ is the two-dimensional torus and $u:\TT^2 \to \RR$ is a periodic function with vanishing average in $x_1$, i.e.,
\begin{align*}
	\int_0^{1} u\,\dd x_1=0 \quad \textrm{ for all } \quad x_2\in [0,1).
\end{align*}

The energy functional $E_{tot}$ was considered in \cite{IO19} as a reduced model describing the 
magnetization ripple, a microstructure formed by the magnetization in a thin ferromagnetic film, which is a result of the polycrystallinity of the material. In 
 thin films, the magnetization can approximately be described by a two-dimensional unit-length vector field (in 
the film plane); the ripple is a perturbation of small amplitude of the constant state, say $(1,0)$. 
In this context, the function $u$ corresponds to the transversal component of the magnetization, after a
suitable rescaling. The theoretical treatment in the physics literature \cite{Hof68,Har68}
takes it for granted that the ripple is universal, in the sense that 
it does not depend on the precise composition and geometry of the polycrystalline material. Our 
main result gives a rigorous justification of the universal behavior of the ripple (see Remark~\ref{rem:univers}).   

The first term in $E_{tot}$ can be interpreted as the exchange energy, an attractive short-range interaction of the spins. 
The second term is the energy of the stray field generated by the magnetization; the fractional structure is due to the 
scaling of the stray field in the thin film regime. On the scales that are 
relevant for the description of the magnetization ripple,
the noise acts like a random transversal field of white-noise character.
It comes via the crystalline anisotropy from the fact that the material is made up of randomly oriented grains that are 
smaller than the ripple scale, which is set to unity in the above model.
In view of its origin, it is reasonable to assume that this noise, which is quenched as opposed to thermal in character, is isotropic, nevertheless the nonlocal 
interaction given by the stray field energy leads to an \emph{anisotropic} response of the 
magnetization. For a more in-depth description and formal derivation of the energy $E_{tot}$ we 
refer to the discussion in \cite[Section 2]{IO19}, which in turn follows \cite{SSWMO12}.

Formally, critical points of $E_{tot}$ are solutions to the Euler-Lagrange equation 
\begin{equs}
	(-\partial_1^2-|\partial_1|^{-1}\partial_2^2) u 
	+ P\left(u R_1\partial_2 u - \frac{1}{2} u R_1 \partial_1 u^2\right) + \frac{1}{2}
	R_1 \partial_2  u^2  
	= \sigma P\xi \quad \text{in } \TT^2, \label{eq:ripple}
\end{equs}
where $R_1=|\partial_1|^{-1}\partial_1$ is the Hilbert transform acting on the $x_1$ variable, see \eqref{eq:hilbert}, and $P$ is the $L^2$-orthogonal projection on functions of zero average in $x_1$ (extended to periodic Schwartz distributions in the natural way). 
One of the main challenges of this equation is that the right-hand side of \eqref{eq:ripple} is too irregular to make sense of the
nonlinear terms, even though the nonlocal elliptic operator 
\begin{align*}
	\mathcal{L}:= -\partial_1^2-|\partial_1|^{-1}\partial_2^2 
\end{align*}
has the expected regularizing properties. 

If we endow our space with a Carnot--Carathéodory metric which respects the natural scaling induced 
by $\mathcal{L}$, that is,
one derivative in the $x_1$ direction costs as much as $\frac{2}{3}$ derivatives in the $x_2$ 
direction, the effective dimension in terms of scaling is given by $\mathrm{dim}=\frac{5}{2}$. It is well-known 
that in this case $\xi$ is a Schwartz distribution of regularity just below $-\frac{\mathrm{dim}}{2}$, i.e., a Schwartz distribution of order $-\frac{5}{4}-$ (measured in a scale 
of H\"older spaces $\mathcal{C}^{\alpha}$ associated to this Carnot--Carathéodory metric; see Section \ref{sec:hoelder} 
below for the definition), where for $\alpha\in \RR$, we use the notation $\alpha-$ to denote 
$\alpha-\varepsilon$ for any $\varepsilon>0$ (suitably small). 

We now argue that the nonlinear term $uR_1\partial_2u$ on the left-hand side of \eqref{eq:ripple} is 
ill-defined: 
On the one hand, Schauder theory for the operator $\mathcal{L}$ improves regularity by $2$ degrees on the H\"older scale, indicating that the expected regularity of a solution $u$ is $(2-\frac{5}{4})- = \frac{3}{4}-$. 
On the other hand, in our anisotropic scaling, one derivative in the $x_2$ direction reduces regularity by $\frac{3}{2}$, while the Hilbert transform has a negligible effect on the regularity. Hence the regularity of the Schwartz distribution $R_1\partial_2u$ is $-\frac{3}{4}-$. 
It is well-known that the product of a function and a Schwartz distribution can be classically and unambiguously defined only if the regularities of the individual terms sum up to a strictly positive number. 
In the case of the product $uR_1\partial_2u$ of the function $u$ and the Schwartz distribution 
$R_1\partial_2u$, the sum of regularities is $0-$, not allowing its treatment by means of classical 
analysis.

This is a common problem in the theory of \emph{singular} Stochastic Partial 
Differential Equations (SPDEs), which has become a very active field in the recent years. Here the 
word singular relates to the fact that the driving noise of these equations is so 
irregular that their nonlinear terms (which usually involve products of the 
solution and its derivatives) are not classically defined. We refer the reader to \cite{Ha14} 
for a more detailed exposition of the theory of singular SPDEs. 

In \cite{IO19} the well-posedness of \eqref{eq:ripple} for noise strength 
$|\sigma|$ below a -- random -- threshold was studied based on Banach's fixed point argument. The ill-defined product $uR_1\partial_2u$ was 
treated via a more direct renormalization technique (in contrast to the more general one appearing 
in the framework of Regularity Structures), known as Wick renormalization. In fact, a similar 
technique had been introduced by Da Prato and Debussche in their work \cite{DD03} on the 
stochastic quantization equations of the $\mathcal{P}(\varphi)_2$-Euclidean Quantum Field theory. 

One of the goals of this paper is to get rid of the smallness condition from \cite{IO19}. \emph{Without loss of generality we may therefore assume that the parameter $\sigma=1$, which we will always do in the following.}\footnote{Note however, that all our results also hold for $\sigma\neq 1$ by considering $\sigma v$ instead of $v$ given in \eqref{eq:linearized_ripple}.} 
This means in particular that we have to give up the use of a fixed point theorem on the level of the Euler-Lagrange equation, and use instead the direct method of the calculus of variations on the level of the functional.
The functional $E_{tot}$ is in need of a renormalization. This is indicated by the fact that if $v$ is the unique 
solution with zero average in $x_1$ to the linearized Euler--Lagrange equation\footnote{From now on, given $\xi$, we denote by $v$ the unique solution with zero average in $x_1$ to the linearized equation $\mathcal{L}v = P\xi$ from \eqref{eq:linearized_ripple}.
}
\begin{equation}\label{eq:linearized_ripple}
	(-\partial_1^2-|\partial_1|^{-1}\partial_2^2) v=P\xi \quad \text{in } \TT^2,
\end{equation}
which is explicit on the level of its Fourier transform, one has that $E_{tot}(v) = -\infty$ almost surely, see Proposition~\ref{prop:inf-energy}. 
As in \cite{DD03} and \cite{IO19}, we decompose any admissible configuration $u$ in \eqref{eq:total-energy} into $u= v+w$, where the remainder $w$ is a periodic function
with vanishing average in $x_1$. 

As is usual in renormalization, we may approximate white noise by a probability measure that is supported on smooth $\xi$'s.
This allows for a pathwise approach: For smooth (and periodic) $\xi$, the solution $v$ of \eqref{eq:linearized_ripple} is smooth, too (see Lemma~\ref{lem:qualitative-regularity-linear}). 
In view of the almost sure divergence of $E_{tot}(v)$ in case of the white noise, we consider the renormalized functional in $w = u-v$,
\begin{equs}
 E_{ren}:= E_{tot}(v+\cdot) - E_{tot}(v) \quad  \text{with} \quad v = \mathcal{L}^{-1} P \xi. \label{eq:ren_energy_intro}
\end{equs}
It follows from Lemma~\ref{lem:qualitative-regularity-linear} that for smooth $\xi$, $E_{ren}$ is well-defined (with values in $\mathbb{R}$) on the space
\begin{equ}
\mathcal{W}:= \left\{w\in L^2(\TT^2) : \int_0^1 w\,\mathrm{d}x_1 = 0 \, \text{ for every } x_2\in[0,1), \, \mathcal{H}(w)<\infty 
\right\}, \label{eq:W}
\end{equ} 
endowed with the strong $L^2$-topology, where $\HH$ denotes the harmonic energy, i.e., the quadratic part of $E_{tot}$ given by
\begin{equs}
	\HH(w) &:= \int_{\TT^2} (\partial_1 w)^2\,\mathrm{d}x 
	+\int_{\TT^2}(|\partial_1|^{-\frac{1}{2}} \partial_2 w)^2\,\mathrm{d}x.
	\label{eq:harmonic-energy}
\end{equs}

Loosely speaking, the task now is to show that $E_{ren}$ can still be given a sense as we approximate the 
white noise. We will consider approximations that belong to the following class of probability measures:\footnote{In the following we specify the topology on the space where we define a (probability) measure, the $\sigma$-algebra always being given by the Borel $\sigma$-algebra 
associated to that topology.} 

\begin{assumption} \label{ass} We consider the class of probability measures $\langle\cdot\rangle$ on the space of periodic Schwartz distributions\footnote{In our notation, we do not distinguish between the probability measure and its expectation, and use $\langle\cdot\rangle$ to denote in particular the latter. 
In the probability jargon, $\xi$ plays the role of a dummy 
variable like the popular $\omega$. We prefer to adopt this point of view, but sometimes it is convenient to also think of $\xi$ as a random variable taking values in the space of periodic Schwartz distributions by identifying it with the canonical evaluation $\xi\mapsto \mathbf{ev}(\xi)$, where 
$\mathbf{ev}(\xi)(\varphi) := \xi(\varphi)$ for all $\varphi \in \mathcal{C}^{\infty}(\TT^2)$. In our notation, when we refer to the law of $\xi$, we mean the law of the random variable $\mathbf{ev}$
or rather the probability measure $\lng\cdot\rng$.} $\xi$ (endowed with the Schwartz topology),
satisfying the following:

\begin{enumerate}[label=(\roman*)]
	\item \label{item:Def-centered} $\langle\cdot\rangle$ is centered: $\langle \xi\rangle = 0$, that is,
	$\lng |\xi(\varphi)| \rng<\infty$ and $\langle \xi(\varphi)\rangle = 0$ for 
	all $\varphi \in \mathcal{C}^{\infty}(\TT^2)$.

	\item \label{item:Def-stationary} $\langle\cdot\rangle$ is stationary, that is, for every shift vector $h\in\mathbb{R}^2$, 
	$\xi$ and $\xi(\cdot+h)$ have the same law.\footnote{Namely, for any test function $\varphi\in\C^{\infty}(\TT^2)$ and shift vector $h\in\RR^2$, $\xi(\varphi) = \xi(\varphi(\cdot-h))$ in law.}
	\item \label{item:Def-reflection} $\langle\cdot\rangle$ is invariant under reflection in $x_1$, that is, $\xi$ and 
	$x\mapsto \xi(-x_1,x_2)$ have the same law.\footnote{That is, for any test function $\varphi\in\C^{\infty}(\TT^2)$, denoting $\tilde\varphi(x) = \varphi(-x_1, x_2)$, there holds $\xi(\varphi) = \xi(\tilde\varphi)$ in law. We note that for our results to hold one could also ask for invariance under reflection in $x_2$.}

	\item \label{item:Def-SG} $\langle\cdot\rangle$ satisfies the spectral gap inequality (SGI), meaning that\footnote{Without loss of generality we have set the constant equal to one.}
		\begin{equs}
		 			 \left\lng |G(\xi) - \lng G(\xi) \rng|^2\right\rng^{\frac{1}{2}}
			 \leq \big\lng \big\|\frac{\partial}{\partial \xi} G(\xi)\big\|_{L^2}^{2} 
			 \big\rng^{\frac{1}{2}}, \label{eq:SG}
		 		\end{equs} 
		for every functional $G$ on the space of Schwartz distributions such that $\lng |G(\xi)| \rng <\infty$ and $G$ is well-approximated by cylindrical functionals. More precisely, for cylindrical functionals
		$G(\xi)=g(\xi(\varphi_1), \dots, \xi(\varphi_n))$ with $g\in \mathcal{C}^{\infty}(\RR^n)$ which itself and all its derivatives have at most polynomial growth, 
		$\varphi_1, \dots, \varphi_n\in\mathcal{C}^{\infty}(\TT^2)$, and $n\in\NN$, we define the Malliavin derivative as the random field
		\begin{align*}
			\frac{\partial}{\partial \xi} G(\xi) := \sum_{i=1}^n \partial_i g(\xi(\varphi_1), \dots, \xi(\varphi_n)) \varphi_i
		\end{align*}
		and complete the space of all cylindrical functionals with respect to the norm 
		$\big\langle |G(\xi)|^2 \big\rangle^{\frac{1}{2}} + \big\langle \big\|\frac{\partial}{\partial \xi} G(\xi) \big\|_{L^2}^2 \big\rangle^{\frac{1}{2}}$. 
		Assuming closability of $\frac{\partial}{\partial \xi}$, i.e., the closure of the graph of $G(\xi) \mapsto \frac{\partial}{\partial \xi} G(\xi)$ on the space of cylindrical functionals with respect to the product topology of $\lng|\cdot|^2\rng^{\frac{1}{2}}$ and $\lng\|\cdot\|^2\rng^{\frac{1}{2}}$ remains a graph\footnote{Closability implies that the definition of the Malliavin derivative for a functional $G$ does not depend on the specific approximation of $G$ from cylindrical functionals. Note that closability in $L^2_{\lng\cdot\rng}\times L^2_{\lng\cdot\rng}L^2_x$ implies closability in $L^p_{\lng\cdot\rng}\times L^p_{\lng\cdot\rng}L^2_x$ for every $p\geq 2$.}, the resulting space can be identified with a subspace of $L^2_{\langle\cdot\rangle}$ for which \eqref{eq:SG} is well-defined\footnote{Incidentally, for cylindrical functionals $G$ we also have the relation 
		\begin{align*}
			\big\|\frac{\partial}{\partial \xi} G(\xi) \big\|_{L^2}^2 
			= \sup_{\substack{\delta\xi\in L^2(\TT^2)\\ \|\delta\xi\|_{L^2} \leq 1}} \lim_{t\to 0} \big| \frac{G(\xi + t \delta\xi)-G(\xi)}{t} \big|^2.
		\end{align*}
		}.
\end{enumerate}
\end{assumption}

\begin{remark}
Note that for a finite dimensional Gaussian probability measure on $\RR^n$ with covariance matrix $A^{-1}\in \RR^{n\times n}$, the spectral gap inequality holds in 
the form
\begin{equs}
 \left\lng |G - \lng G \rng|^2\right\rng^{\frac{1}{2}}
 \leq \left\lng \nabla G \cdot A^{-1} \nabla G \right\rng^{\frac{1}{2}}, 
\end{equs}
for every bounded continuously differentiable $G:\RR^n \to \RR$ (see for example \cite[Theorem 2.1]{Hel98}). In the case of white noise this translates
into \eqref{eq:SG}. Furthermore, it is easy to check that if we convolve white noise with a smooth mollifier, the resulting random field 
satisfies \eqref{eq:SG}.
\end{remark}

\begin{remark} \label{rem:wn}
For a linear functional $G$, i.e. $G$ of the form $G(\xi)=\xi(\varphi)$, for some $\varphi\in\C^\infty(\TT^2)$,
the spectral gap inequality \eqref{eq:SG} turns into $\lng\xi(\varphi)^2\rng\le\int_{\TT^2}\varphi^2\,\dd x$, which is a defining property of white noise turned
into an inequality. Note that this allows us to extend $\xi(\varphi)$ to $\varphi\in L^2(\TT^2)$ as a centered random variable in $L^2_{\lng\cdot\rng}$
which is admissible in \eqref{eq:SG}. In our application, the spectral gap inequality implies 
that $\lng\cdot\rng$ is supported on the H\"older space $\C^{-\frac{5}{4}-}$ (see Proposition~\ref{Satz1_pavlos} below), yielding the same regularity as in the
case of white noise. 

The merit of SGI is that it also applies to nonlinear $G$ (in this paper, we need it for quadratic $G$)\footnote{Actually, in order to obtain the $L^p$ version of SGI in 
Proposition~\ref{prop:SG_mult} we need \eqref{eq:SG} for more general $G$, but the main application concerns the quadratic functional in Lemma~\ref{lem:commutator}.}.
In addition, it allows us to obtain sharp stochastic estimates for non-Gaussian measures (see Proposition~\ref{prop:xi_v_clean} and Proposition~\ref{prop:F_approx_tight}) by providing a
substitute for Nelson's hyper-contractivity. To the best of our knowledge, the use of the spectral gap inequality is new in the context of singular SPDEs.
\end{remark}

The second, and more subtle, goal of this paper, is to establish universality of the ripple. 
By this we mean that the limiting law of the renormalized energy functional $E_{ren}$ in \eqref{eq:ren_energy_intro} is independent of the way white noise is 
approximated, provided the natural symmetry condition in form of Assumption~\ref{ass}~\ref{item:Def-reflection}, is satisfied.
In view of its physical origin, $\langle\cdot\rangle$ derives from the random orientation of the grains. Such a model could be based on random tessellations, which suggests a modelling through a non-Gaussian process.\footnote{Incidentally, random tessellations based on Poisson point processes are known to satisfy a variant of the spectral gap inequality, see \cite{DG20}.}  
 This motivates our interest in non-Gaussian approximations of white noise.
Our substitute for Gaussian calculus is the spectral gap inequality\footnote{In a Gaussian setting the right-hand side of \eqref{eq:SG} would correspond to having $L^2$ as the Cameron--Martin space.} \eqref{eq:SG}, see Assumption~\ref{ass}~\ref{item:Def-SG}.
 
Since we cannot expect almost-sure uniqueness of the absolute minimizer $w$ due to non-convexity of $E_{ren}$, this universality is better expressed on the level of 
the variational problems $E_{ren}$ themselves.
Hence, rather than considering the (ill-defined) random fields $w$ of minimal configurations, we consider the random functionals $E_{ren}$. 
The latter notion calls for a topology on the space of variational problems, that is, of lower semicontinuous functionals $E$ on ${\mathcal W}$
(taking values in $\RR\cup\{+\infty\}$) that have compact 
sublevel sets (with respect to the strong $L^2$-topology). 
The appropriate topology is the one generated by $\Gamma$-convergence\footnote{$\Gamma$-convergence with respect to the strong $L^2$-topology in the domain $\mathcal{W}$ on which the functionals $E$ are defined.}; this is 
tautological since $\Gamma$-convergence of functionals in this topology is essentially equivalent to convergence of the 
minimizers, which do exist provided the functionals have compact sublevel sets. 
Hence we are led to consider probability measures on the space of lower semicontinuous functionals $E$ on $\mathcal{W}$ 
endowed with the topology of $\Gamma$-convergence\footnote{The general framework of $\Gamma$-convergence (w.r.t.\ the convergence of minimizers and their existence for functionals with compact sublevel sets) is explained in the book of Dal Maso \cite[Theorem 7.8]{DM93}. Recall that the space of lower semicontinuous functionals 
$E:\mathcal{W} \to \RR\cup\{+\infty\}$ is a compact space, see \cite[Theorem 8.5 and Theorem 10.6]{DM93}.}. From 
this point of view, the universality of the ripple takes the following form:

\begin{theorem}\label{thm:Thm1}
Every probability measure $\langle\cdot\rangle$ on the space of periodic Schwartz distributions $\xi$ satisfying Definition 
\ref{ass} gives rise to a probability measure $\lng\cdot\rng^{\mathrm{ext}}$ on the product space of periodic Schwartz distributions $\xi$ in the 
H\"older space $\C^{-\frac{5}{4}-}$ and lower semicontinuous functionals $E$ on ${\mathcal W}$ endowed with the topology of $\Gamma$-convergence with
the following three properties:
\begin{enumerate}[label=(\roman*)]
	\item \label{item:thm:Thm1-xi_char} The law of $\xi$ under $\lng\cdot\rng^{\mathrm{ext}}$ is $\lng\cdot\rng$, that is, for every 
	continuous and bounded functional $G$ on the space of periodic Schwartz distributions, $\lng G(\xi) \rng^{\mathrm{ext}} = \lng G(\xi) \rng$.
	\item \label{item:thm:Thm1-smooth} If $\xi$ is smooth $\langle\cdot\rangle$-almost surely, then 
	$E=E_{ren}$ for $\lng\cdot\rng^{\mathrm{ext}}$-almost every $(\xi, E)$, where 
	$E_{ren}$ is given by \eqref{eq:ren_energy_intro}.
	\item \label{item:thm:Thm1-convergence} If a sequence $\{\langle\cdot\rangle_\ell\}_{\ell\downarrow0}$ of probability
	measures that satisfy Assumption~\ref{ass} converges weakly to $\langle\cdot\rangle$ (which automatically satisfies Assumption~\ref{ass}), then
	$\{\langle\cdot\rangle_\ell^{\mathrm{ext}}\}_{\ell\downarrow0}$ converges weakly to $\lng\cdot\rng^{\mathrm{ext}}$.
\end{enumerate}
\end{theorem}

\begin{remark} As we mentioned in Remark~\ref{rem:wn} (see also Proposition~\ref{Satz1_pavlos} below), Assumption~\ref{ass} implies that the original measure $\lng\cdot\rng$ in Theorem~\ref{thm:Thm1}
is supported on the H\"older space $\C^{-\frac{5}{4}-}$. From this point of view, the measure $\lng \cdot \rng^{\mathrm{ext}}$ in Theorem~\ref{thm:Thm1} 
can be seen as an extension to the product space of the H\"older space $\C^{-\frac{5}{4}-}$ and the space of lower semicontinuous functionals on ${\mathcal W}$.  
\end{remark}

\begin{remark} \label{rem:univers}
Let us explain why Theorem~\ref{thm:Thm1} expresses the desired universality of the ripple.
We are given a sequence $\{\langle\cdot\rangle_\ell\}_{\ell\downarrow0}$ which converges weakly to white noise
$\langle\cdot\rangle$, such that $\langle\cdot\rangle_\ell$ satisfies Assumption~\ref{ass}, and such that for $\ell>0$, 
$\xi$ is smooth $\langle\cdot\rangle_\ell$-almost surely.
In view of Theorem~\ref{thm:Thm1}~\ref{item:thm:Thm1-smooth}, as long as $\ell>0$, the pathwise defined $E_{ren}$, 
see \eqref{eq:ren_energy_intro}, can be identified with the random functional $E$ associated to $\langle\cdot\rangle_\ell^{\mathrm{ext}}$.
According to Theorem~\ref{thm:Thm1}~\ref{item:thm:Thm1-convergence}, as $\ell\downarrow 0$, the law of $(\xi, E_{ren})$ 
under $\langle\cdot\rangle_\ell^{\mathrm{ext}}$ converges weakly to $\lng\cdot\rng^{\mathrm{ext}}$ associated to the
law of white noise $\langle\cdot\rangle$.
\end{remark}

As a corollary of our results we have the following stronger statement.

\begin{corollary}\label{cor:regularity}
	Assume that the probability measure $\langle\cdot\rangle$ satisfies Assumption~\ref{ass} and consider its extension $\lng\cdot\rng^{\mathrm{ext}}$ to the 
	product space of periodic Schwartz distributions $\xi$ in the H\"older space $\C^{-\frac{5}{4}-}$ and lower semicontinuous
	functionals $E$ on $\mathcal{W}$ endowed with the topology of $\Gamma$-convergence. Then 
	minimizers of $E$ exist in $\mathcal{W}$ for $\lng\cdot\rng^{\mathrm{ext}}$-almost every $(\xi,E)$. Moreover, for every $1\leq p<\infty$, 
	the following estimate holds,\footnote{With the convention that the infimum in \eqref{eq:min_hoelder_bd} is $+\infty$ if 
	$\mathrm{argmin} E = \emptyset$.}
	\begin{equs}
	 \left\lng \inf_{w\in\mathrm{argmin}E} [w]_{\frac{5}{4}-}^p \right\rng^{\mathrm{ext}} \leq C,
	 \label{eq:min_hoelder_bd}
	\end{equs}
	for a constant $C$ that only depends on $p$, uniformly in the class of probability measures $\lng\cdot\rng$ satisfying Assumption~\ref{ass}.
\end{corollary}

\begin{remark} Since under $\lng\cdot\rng^{\mathrm{ext}}$ functionals $E$ are non-convex (see the discussion below 
Proposition~\ref{Satz2}) we do not expect uniqueness of minimizers. In that sense, Corollary~\ref{cor:regularity} 
shows the existence of minimizers $w\in \mathcal{W}$ with finite $\C^{\frac{5}{4}-}$-norm. However, we do not know
if all minimizers have this regularity. The $\C^{\frac{5}{4}-}$-regularity in \eqref{eq:min_hoelder_bd} 
relies on the Euler--Lagrange equation \eqref{eq:ripple_rem} (see the discussion below Proposition~\ref{Satz2}). 
If minimizers were unique, \eqref{eq:min_hoelder_bd} would imply tightness of their law in the class of probability measures satisfying 
Assumption~\ref{ass}. This in turn would imply that if a sequence $\{\langle\cdot\rangle_\ell\}_{\ell\downarrow0}$ satisfying Assumption~\ref{ass}
converges weakly to $\langle\cdot\rangle$, then the law of the unique minimizer under $\langle\cdot\rangle_\ell^{\mathrm{ext}}$ converges weakly to the law of 
the unique minimizer under $\langle\cdot\rangle^{\mathrm{ext}}$.
\end{remark}

In establishing Theorem \ref{thm:Thm1}, we follow very much the spirit of rough path theory of a clear separation between a stochastic and a deterministic (pathwise) ingredient. 
The genuinely stochastic ingredient is formulated in Proposition \ref{Satz1_pavlos}, where we extend the probability distribution of $\xi$'s to a (joint) law of $(\xi,v,F)$. Compared to rough paths, $v$ is the
analogue of (multi-dimensional) Brownian motion $B$, and $F$ similar to the iterated ``integrand''\footnote{As opposed to its integral $\int B \,\dd B$, which is called the iterated integral. We refer to 
\cite[Chapter 3]{FH14} for details.} $B\frac{dB}{dt}$, see Proposition~\ref{Satz1_pavlos}~\ref{it:Satz1_pavlos_F}, which relates $F$ to $v R_1\partial_2 v$.
The degree of indeterminacy reflected by the difference between Stratonovich (midpoint rule) and Itô (explicit) is suppressed by the symmetry condition in Assumption~\ref{ass}~\ref{item:Def-reflection}, 
which feeds into the characterizing property given by \eqref{eq:F_pavlos}. The crucial stability of this construction is provided by 
Proposition~\ref{Satz1_pavlos}~\ref{it:Satz1_pavlos_cont}.

\medskip

Like for rough paths this unfolding into several random building blocks allows for a pathwise solution theory, i.e.\ the construction of a continuous solution map on this augmented space. In our variational
case this turns into a continuous map from the space of $(\xi,v,F)$'s into the space of functionals, with the above-advertised topology of $\Gamma$-convergence, see Proposition \ref{Satz2}).

\begin{proposition}\label{Satz1_pavlos} Every probability measure
$\langle\cdot\rangle$ satisfying Assumption~\ref{ass} is supported on the H\"older space $\C^{-\frac{5}{4}-}$ and lifts to a probability 
measure $\lng\cdot\rng^{\mathrm{lift}}$ on the space of triples $(\xi,v,F)$ in ${\mathcal C}^{-\frac{5}{4}-}\times{\mathcal C}^{\frac{3}{4}-}\times {\mathcal C}^{-\frac{3}{4}-}$
with the following properties:
\begin{enumerate}[label=(\roman*)]
	\item \label{it:Satz1_pavlos_xi} The law of $\xi$ under $\lng\cdot\rng^{\mathrm{lift}}$ is $\lng\cdot\rng$. 
	\item \label{it:Satz1_pavlos_v} $v=\mathcal{L}^{-1}P\xi$ $\lng\cdot\rng^{\mathrm{lift}}$-almost surely.
	\item \label{it:Satz1_pavlos_F} The law of $F$ under $\lng\cdot\rng^{\mathrm{lift}}$ is characterized by  
		\begin{align}\label{eq:F_pavlos}
		\lim_{t=2^{-n}\downarrow0}\left\langle[F-vR_1\partial_2v_t]_{-\frac{3}{4}-}^p\right\rangle^{\mathrm{lift}}=0,
		\end{align}
		for every $1\leq p<\infty$, where $v_t$ denotes the convolution of $v$ with a suitable semigroup $\psi_t$.\footnote{Actually, $\psi_t$ is the semigroup associated to the 
		operator $|\partial_1|^3-\partial_2^2$ (see also \eqref{eq:psiThat}).} 
		Moreover, if $\xi$ is smooth $\lng\cdot\rng$-almost surely, we have that 
		$F=vR_1\partial_2v$ $\langle\cdot\rangle^{\mathrm{lift}}$-almost surely.
	\item \label{it:Satz1_pavlos_cont} Finally, if a sequence $\{\langle\cdot\rangle_\ell\}_{\ell\downarrow0}$ of probability measures that satisfy 
	Assumption~\ref{ass} converges weakly to a probability measure $\lng\cdot\rng$, then also $\{\lng \cdot\rng_\ell^{\mathrm{lift}}\}_{\ell\downarrow0}$
	converges weakly to $\lng\cdot\rng^{\mathrm{lift}}$.
\end{enumerate}
\end{proposition}

\begin{remark} Let us point out that \eqref{eq:F_pavlos} implies that $F$ is actually a $\langle\cdot\rangle$-measurable function of $\xi$. Indeed, by \ref{it:Satz1_pavlos_xi}, \ref{it:Satz1_pavlos_v},
and the triangle inequality for $s,t\in(0,1]$ dyadic (i.e.\ $t=2^{-n}$, $s = 2^{-m}$ for $m,n\in\NN$) we have
\begin{equs}
 \left\lng [vR_1\partial_2v_s - vR_1\partial_2v_t]_{-\frac{3}{4}-}^p \right\rng^{\frac{1}{p}} & = 
 \left(\left\lng [vR_1\partial_2v_s - vR_1\partial_2v_t]_{-\frac{3}{4}-}^p \right\rng^{\mathrm{lift}}\right)^{\frac{1}{p}} 
 \\
 & \leq
 \left(\left\lng [F - vR_1\partial_2v_s]_{-\frac{3}{4}-}^p \right\rng^{\mathrm{lift}}\right)^{\frac{1}{p}} 
 + \left(\left\lng [F - vR_1\partial_2v_t]_{-\frac{3}{4}-}^p \right\rng^{\mathrm{lift}}\right)^{\frac{1}{p}},
\end{equs}
which in turn implies that the sequence $\{vR_1\partial_2 v_t\}_{t=2^{-n}\downarrow0}$ is Cauchy in $L^p_{\lng\cdot\rng}\C^{-\frac{3}{4}-}$. Hence it converges to a random variable 
$F(\xi)\in L^p_{\lng\cdot\rng}\C^{-\frac{3}{4}-}$ and it is easy to check that $F = F(\xi)$ $\lng\cdot\rng^{\mathrm{lift}}$-almost surely. This allows us to identify the 
lift measure $\lng\cdot\rng^{\mathrm{lift}}$ as the joint law of $(\xi, \mathcal{L}^{-1}P\xi, F(\xi))$ under $\lng\cdot\rng$.
\end{remark}

The main idea of the deterministic ingredient, Proposition \ref{Satz2}, is to extend the definition \eqref{eq:ren_energy_intro} of $E_{ren}$ from 
only depending on $(\xi,v)$\footnote{In fact, $E_{ren}$ depends effectively only on $\xi$, as $v= \mathcal{L}^{-1}P\xi$. However, we drop $\xi$ in the notation for $E_{ren}$ and keep the dependence on $v$.} to depending on $(\xi,v,F)$, in such a way that the definitions (formally) 
coincide for $F=vR_1\partial_2v$. This is achieved by\footnote{If $u=v+w$, this can be seen by the identity $\eta_u=\eta_{v}+\eta_w-\partial_1(v w)$ for the Burgers operator
$\eta_u = \partial_2 u - \partial_1 \frac{1}{2} u^2$,
as well as the equality
\begin{align*}
 \int_{\TT^2} \left(\partial_1 w \, \partial_1 v +\partial_2 w \, |\partial_1|^{-1} \partial_2 
 v-  w  \, \xi\right) \, \dd x=0,
\end{align*}
which follows from testing \eqref{eq:linearized_ripple} with $w$.
}
\begin{align}\label{eq:renormalized-energy}
	E_{ren}(v,F;w):= \EE(w) + \GG(v, F;w),
\end{align}
where the anharmonic energy $\mathcal{E}$ is given by the first two contributions of $E_{tot}$, 
\begin{equs}
	\EE(w) &:= \int_{\TT^2} (\partial_1 w)^2\,\mathrm{d}x 
	+\int_{\TT^2}\left(|\partial_1|^{-\frac{1}{2}} (\partial_2 w - 
	\partial_1 \tfrac{1}{2} w^2)\right)^2\,\mathrm{d}x.
	\label{eq:energy}
\end{equs}
Note that $\EE$ contains the Burgers nonlinearity $\eta_w:= \partial_2 w - 
	\partial_1 \tfrac{1}{2} w^2$, which will play an important role in our analysis. Only the remainder $\mathcal{G}$ depends on $v$ and $F$, and is given by 
\begin{align}
	\mathcal{G}(v,F;w) 
	: = \int_{\TT^2} \Big( & w^2 R_1 \partial_2 v
	+ v^2 R_1 \eta_w 
	+ 2 v w R_1 \eta_w 
	+ 2 w F
	 - w  v R_1 \partial_1 v^2 
	+ (R_1 |\partial_1|^{\frac{1}{2}}(v w))^2 \Big) \,\mathrm{d}x.
	\label{eq:G}
\end{align}
Equipped with these definitions, we now may state the main deterministic ingredient.

\begin{proposition}\label{Satz2}
The application $(\xi,v,F)\mapsto E_{ren}$ described through \eqref{eq:renormalized-energy} is well-defined and continuous from the space of $(\xi,v,F)$ endowed with the norm ${\mathcal C}^{-\frac{5}{4}-}\times{\mathcal C}^{\frac{3}{4}-}\times{\mathcal C}^{-\frac{3}{4}-}$ into the space of lower semicontinuous functionals $E_{ren}$ on ${\mathcal W}$ endowed with the topology of $\Gamma$-convergence (with respect to the $L^2$-topology on $\mathcal{W}$).
\end{proposition}

\input{figure-2}

On the level of the Euler--Lagrange equation, minimizers $w$ of $E_{ren}(v,F;\cdot)$ are weak solutions of
\begin{equation}
\begin{split}
	\mathcal{L} w & + P \left(  F + w R_1\partial_2 v + v R_1\partial_2 w + w 
	R_1\partial_2  w - \frac{1}{2}  (v+w) R_1 \partial_1 (v+w)^2 \right) \\
	& + \frac{1}{2} R_1 \partial_2  (v+w)^2 = 0,
\end{split}
\label{eq:ripple_rem}
\end{equation}
whose existence is established by Theorem~\ref{thm:Thm1} (see also Theorem~\ref{thm:minimizers}~\ref{it:existence_min} for the 
validity of \eqref{eq:ripple_rem} in the sense of Schwartz distributions). By a simple power counting the
expected regularity of solutions $w$ to \eqref{eq:ripple_rem} is $\frac{5}{4}-$, which justifies the existence of minimizers 
$w\in \mathcal{W}$ with finite $\C^{\frac{5}{4}-}$ norm proved in Corollary \ref{cor:regularity}. This generalizes
the existence of solutions to \eqref{eq:ripple_rem} in \cite{IO19} which was shown for small values of the noise strength $|\sigma|$.

From a variational point of view, the main challenge is to establish the coercivity of the renormalized 
energy $E_{ren}(v, F;\cdot)$. Ideally, one would like to control the remainder 
$\GG(v, F;\cdot)$ by the ``good'' term of the renormalized energy, namely, the anharmonic 
energy $\EE(w)$ given in \eqref{eq:energy}. At first sight, this is not obvious since the remainder $\GG(v, F;\cdot)$
contains quadratic and cubic terms in $w$, and it is not immediate that the anharmonic energy $\EE(w)$ provides 
higher than quadratic control to absorb these terms. Hence, we need to exploit the control on the nonlinear
part coming from the Burgers operator $\eta_w$.\footnote{Incidentally, despite different physical origins,
the inviscid Burgers part $\eta_w$ arises as in the KPZ equation from expanding a square root nonlinearity. Not unlike there, 
the coercivity comes from the interaction between the first and second term in $\mathcal{E}(w)$,
the first term being the analogue to the viscosity in KPZ.} We do this using tools from fluid mechanics,
more precisely, the Howarth--K\'arm\'an--Monin
identities \eqref{eq:KHM_mod} and \eqref{eq:KHM_clas}, following \cite{GJO15}. Based on these identities we can prove that 
the anharmonic energy $\EE(w)$ grows cubically in suitable Besov spaces (see Proposition~\ref{prop:a_priori_est}).
This allows us to absorb $\GG(v, F;\cdot)$ and obtain the coercivity of the renormalized energy 
(see Theorem~\ref{thm:minimizers} \ref{it:coercivity}).

Let us point out that here, we prove existence of solutions for any value of the noise strength 
$\sigma$ using the coercivity of the renormalized energy functional through the direct method of the calculus of 
variations. Recent works on the dynamic $\Phi^4$ model (or stochastic Ginzburg--Landau model), 
where the system is favoured by the ``good'' sign of the cubic nonlinearity, have used coercivity 
on the level of the Euler--Lagrange equation. For example, in \cite{MW17b, TW18, MW17c} energy estimates have been used 
to obtain global-in-time existence in the parabolic case, while in \cite{GH19} both the parabolic and the 
elliptic cases have been treated based on a different approach that uses coercivity through a maximum principle.
A maximum principle has been used also in \cite{MW19} where the parabolic model is considered in the full subcritical regime.

A further challenge, which turns out to be more on the technical side, comes
from the fact that $\mathcal{L}$ is nonlocal. We recall that this feature arises
completely naturally from the magnetostatic energy in the thin-film limit 
(see \cite[Section 2]{IO19}), but resonates well with the recent surge in activity on nonlocal operators. 
It was worked out in \cite[Lemma 5]{IO19} that the robust approach of \cite{OW19} to negative (parabolic) 
H\"older spaces and Schauder theory extends to this situation.
This approach involves a suitable convolution semigroup $\psi_t$; the fact that it extends from the
smooth parabolic symbol $k_1^2+ik_2$ to our nonsmooth symbol $k_1^2+|k_1|^{-1}k_2^2$
is not obvious due to the poor decay properties of the corresponding convolution kernel.

Variational problems that in a singular limit require subtraction of a divergent term are 
well-known in deterministic settings. A famous example concerns $\SS^1$-valued harmonic maps 
defined in a two-dimensional smooth bounded simply-connected domain $D$. The aim there is to minimize the 
Dirichlet energy of maps $u:D\to \SS^1$ that satisfy a smooth boundary condition $g:\partial D \to 
\SS^1$. When $g$ carries a nontrivial winding number $N>0$\footnote{For simplicity, we assume $N>0$; the 
case $N<0$ follows by complex conjugation.}, the problem is 
singular, that is, every configuration $u$ has infinite energy as they generate vortex point singularities.
The question is 
to determine the least ``infinite'' Dirichlet energy of a harmonic $\SS^1$-valued map satisfying the 
boundary condition $g$ on $\partial D$. The seminal book of Bethuel--Brezis--H\'elein \cite{BBH94} 
presents two methods to achieve this goal, both reaching the same renormalized energy associated to 
the problem. 

\begin{enumerate}[label=]
\item \textsc{First approach:} One prescribes $N>0$ vortex 
points $a_1, \dots, a_N$ in $D$ and determines the unique harmonic $\SS^1$-valued map $u_*$ with 
$u_*=g$ on $\partial D$ that has the prescribed singularities $a_1, \dots, a_N$ in $D$, each one 
carrying a winding number equal to one\footnote{In fact, $u_*$ belongs to the Sobolev space $W^{1,1}(D, \SS^1)$ 
and the nonlinear PDE satisfied by $u_*$, i.e., $-\Delta u_*=|\nabla u_*|^2 u_*$ in $D$, can be written in a ``linear'' way in terms of the current 
$j(u_*)=u_*\times \nabla u_*\in L^1(D)$ of $u_*$ that satisfies the system $\nabla \times j(u_*)=2\pi \sum_k \delta_{a_k}$ 
in $D$, and $\nabla \cdot j(u_*)=0$ in $D$. In terms of the
so-called conjugate harmonic function $\phi$ given by $\nabla^\perp \phi=j(u_*)$, the problem becomes 
$-\Delta \phi=2\pi \sum_k \delta_{a_k}$ in $D$ and $\partial_\nu \phi=g\times \partial_\tau g$ on 
$\partial D$. One could think of $\phi$ as playing the role of our solution $v$ to the linearized 
Euler-Lagrange equation \eqref{eq:linearized_ripple} that carries the ``infinite'' part of the energy.}.
Then one cuts-off disks $B(a_k, r)$ centered at $a_k$ of small radius $r>0$ carrying the 
diverging logarithmic energy of $u_*$ and introduces the renormalized energy
\begin{equs}
W(a_1, \dots , a_N)=\lim_{r\to 0} \left(\int_{D\setminus \cup_k  B(a_k, r)} |\nabla u_*(x)|^2 \,\dd x -2\pi N\log \frac1r\right) .
\end{equs}
The minimum of the renormalized energy 
\begin{equs}
\min_{a_1, \dots a_N \in D} W(a_1, \dots , a_N) \label{min_reno}
\end{equs}
represents the minimal second order term in the expansion of the Dirichlet energy and yields 
optimal positions of the $N$ vortex point singularities (which might not be unique in general).

\item \textsc{Second approach:} One considers a nonlinear approximation of the harmonic map 
problem 
given by the Ginzburg-Landau model for a small parameter $\eps>0$:
\begin{equs}
E_\eps(u)=\int_D |\nabla u|^2+\frac1{\eps^2} (1-|u|^2)^2\, \dd x, \quad u:D\to \RR^2, \quad u=g 
\text{ on } \partial D.
\end{equs}
Note that the maps $u$ are no longer with values into $\SS^1$, but their distance to $\SS^1$ is 
strongly penalized as $\eps\to 0$. It is proved in \cite[Theorem X.1]{BBH94}
that if $u_\eps$ is a minimizer of the above Ginzburg--Landau problem, then for a subsequence, $u_\eps\rightharpoonup u_*$ weakly in $W^{1,1}(D)$ as $\eps\to 
0$ where $u_*$ is an $\SS^1$ valued harmonic map whose $N$ vortex points of winding number one 
correspond to a minimizer of the renormalized energy \eqref{min_reno}. Moreover,
\begin{equs}
E_\eps(u_\eps)=2\pi N \log \frac1\eps+\min_{a_1, \dots a_N \in D} W(a_1, \dots , a_N)+N\gamma+o(1), 
\quad \text{as } \eps\to 0,
\end{equs}
where $\gamma$ is a constant coming from the nonlinear penalization in $E_\eps$.
\end{enumerate}
We also refer to \cite{SS07,SS15}, \cite{K06}, \cite{IM16}, and \cite{IJ19} for similar renormalized energies. 

\subsection{Notation}
For a periodic function $f: \TT^2 \to \RR$ we define its Fourier coefficients by 
\begin{align*}
\widehat{f}(k)=\int_{\TT^2} e^{-\mathrm{i} k \cdot x} f(x)\, \mathrm{d}x \quad \text{for} \quad k\in (2\pi \ZZ)^2,
\end{align*}
which extends to periodic Schwartz distributions in the natural way. We also denote by $P$ the $L^2$-orthogonal projection onto 
the set of functions of vanishing average in $x_1$, extended in the natural way to periodic Schwartz distributions. 

For $p\in[1,\infty]$ we write $\|\cdot\|_{L^p}$ to denote the usual  
$L^p$ norm on $\TT^2$, unless indicated otherwise. For example, we write
$\|\cdot\|_{L^p(\RR^2)}$ for the $L^p$ norm of a function defined on $\RR^2$. We sometimes write 
$L^p_x$ (respectively $L^p_{x_j}$, $j=1,2$) to denote the $L^p$ space with 
respect to the $x$ (respectively $x_j$) variable. We also write $L^p_{\lng\cdot\rng}$ to denote 
the usual $L^p$ space with respect to the measure $\lng\cdot\rng$.

We will often make use of the notation $a \lesssim b$ meaning that there exists a constant 
$C>0$ such that $a \leq C b$. Moreover, for $\kappa\in \RR$, the notation $\lesssim_\kappa$ will be used to 
stress the dependence of the implicit constant $C$ on $\kappa$, i.e., $C\equiv C(\kappa)$.
Similarly, $a\sim b$ means $a\lesssim b$ and $b\lesssim a$. 

\subsubsection{Hilbert transform}
We will frequently make use of the 
\emph{Hilbert transform} $R_j$ for $j=1,2$, acting on periodic functions $f: \TT^2 \to \RR$ in $x_j$ as
\begin{align}\label{eq:hilbert}
R_j:=\frac{\partial_j}{|\partial_j|}, \quad \textrm{i.e.,} \quad \widehat{R_j f}(k)=\begin{cases} \mathrm{i} \sgn(k_j) \widehat{f}(k) & \text{if } k_j \in 2\pi \ZZ\setminus \{0\},\\
0 & \text{if } k_j=0,
\end{cases}
\end{align}
where $\sgn$ is the sign function. In particular, $R_jP=PR_j=R_j$.

\subsubsection{Anisotropic metric and kernel}
The leading-order operator $\mathcal{L}=-\partial_1^2-|\partial_1|^{-1}\partial_2^2$ suggests to endow the space $\TT^2$ with a Carnot--Carath\'eodory metric that is homogeneous with respect to the scaling 
$(x_1,x_2)=(\ell \hat x_1,\ell^\frac{3}{2}\hat x_2)$. 
The simplest expression is given by
\begin{align*}
d(x,y):=|x_1-y_1|+|x_2-y_2|^\frac{2}{3},\quad x, y\in {\TT^2},
\end{align*}
which in particular means that we take the $x_1$ variable as a reference.

We now introduce the convolution semigroup used in \cite{IO19}. This is the 
``heat kernel'' $\{\psi_T\}_{T>0}$ of the operator 
\begin{equs}
\mathcal{A}:=|\partial_1|^3-\partial_2^2=|\partial_1|\LL,
\end{equs}
which, in Fourier space $\RR^2$, is given by
\begin{align}\label{eq:psiThat}
\widehat{\psi_T}(k)=\exp(-T(|k_1|^3+k_2^2)), \quad \text{for all}\quad k\in \RR^2.
\end{align}
It is easy to check that the kernel has scaling properties in line with the metric $d$, that is,
\begin{align}
\psi_T(x_1,x_2)=\frac{1}{(T^{\frac{1}{3}})^{1+\frac{3}{2}}}
\psi\left(\frac{x_1}{T^{\frac{1}{3}}},\frac{x_2}{(T^{\frac{1}{3}})^\frac{3}{2}}\right), \quad \text{for all}  \quad x\in \RR^2,
\label{eq:heat_ker}
\end{align}
where for simplicity we write $\psi:=\psi_1$. Note that $\psi$ is a symmetric smooth function with integrable derivatives and we have for
every $p\in [1, \infty]$ (see \cite[Proof of Lemma 10]{IO19}),
\begin{equation}
\label{lp_psi}
\|D_1^\alpha D_2^\beta \psi_T\|_{L^p(\RR^2)}\sim (T^{\frac13})^{-\alpha-\frac32\beta-\frac52(1-\frac1p)},
\end{equation}
for every $ T>0$, $D_j\in \{\partial_j, |\partial_j|\}$,  $j=1,2$, and $\alpha, \beta\geq 0$.
For a periodic Schwartz distribution $f$, we denote by $f_T$ its convolution with $\psi_T$, i.e., $f_T=\psi_T*f$, which yields a smooth periodic function.
Notice that $\{\psi_T\}_{T>0}$ is a convolution semigroup, so that 
\begin{equ}
	(f_t)_T=f_{t+T}\quad\text{for all}\quad t,T>0.
\end{equ}
%

\begin{remark} \label{rem:perio}
By the space of periodic Schwartz distributions $f$ we understand the (topological) dual of the space of $\C^\infty$-functions $\varphi$ on the
torus (endowed with the family of seminorms $\{\|\partial_1^j\partial_2^l \varphi\|_{L^\infty}\}_{j,l\ge 0}$). 

For a $\C^\infty$-function $\psi$ on $\RR^2$ with integrable derivatives, i.e. 
$\int_{\mathbb{R}^2}|\partial_1^j\partial_2^l\psi| \, \dd x <\infty$ for all $j,l\ge 0$,
and a periodic Schwartz distribution $f$ we write $(f*\psi)(x)$  to denote $f(\Psi(x-\cdot))$, where $\Psi:=\sum_{z\in\mathbb{Z}^2}\psi(\cdot-z)$ is 
the periodization of $\psi$, which is well-defined and belongs to $\C^\infty(\TT^2)$.
In particular, if $\Psi_T$ denotes the periodization of our ``heat kernel'' $\psi_T$, then $\Psi_T$ is a smooth semigroup whose Fourier
coefficients are given by $\hat{\psi}_T(k)$ in \eqref{eq:psiThat} for $k\in (2\pi \ZZ)^2$, yielding for any $T\in (0,1]$, 
$D_j\in \{\partial_j, |\partial_j|\}$, $j=1,2$, and $\alpha, \beta\geq 0$:
\footnote{Indeed, for $p=1$, we have
\begin{equs}\textstyle
 \| D_1^\alpha D_2^\beta \Psi_T\|_{L^1} 
 \leq \sum_{z\in \ZZ^2}\int_{\TT^2} |D_1^\alpha D_2^\beta \psi_T(x-z)|\,\dd x
 = \|D_1^\alpha D_2^\beta \psi_T\|_{L^1(\RR^2)},
\end{equs}
while for $p=\infty$, 
\begin{equs}
 \|D_1^\alpha D_2^\beta \Psi_T\|_{L^\infty} & \leq 1 + \textstyle\sum_{k\in (2\pi \ZZ)^2\setminus \{(0,0)\}} |k_1|^\alpha |k_2|^\beta |\hat\psi_T(k)|
  \lesssim 1 + \int_{\RR^2}
 |\xi_1|^\alpha |\xi_2|^\beta \exp(-T(|\xi_1|^3+\xi_2^2))\, \dd \xi 
 \\
 & \lesssim 
 (T^{\frac13})^{-\alpha-\frac32\beta-\frac52}\stackrel{\eqref{lp_psi}}{\lesssim} \|D_1^\alpha D_2^\beta \psi_T\|_{L^\infty(\RR^2)}.
\end{equs}
For $p\in (1,\infty)$, one argues by interpolation.
}
\begin{equation}
\label{lp_psi_per}
\|D_1^\alpha D_2^\beta \Psi_T\|_{L^p} \lesssim \|D_1^\alpha D_2^\beta \psi_T\|_{L^p(\RR^2)}\stackrel{\eqref{lp_psi}}{\lesssim} 
(T^{\frac13})^{-\alpha-\frac32\beta-\frac52(1-\frac1p)}.
\end{equation}
Therefore, for a periodic function $f\in L^q$, we will often use Young's inequality for convolution with $1+\frac1r=\frac1p +\frac1q$
in the form 
\begin{equs}
 \|f*D_1^\alpha D_2^\beta \psi_T\|_{L^r} 
 &\leq \|f\|_{L^q}
 \|D_1^\alpha D_2^\beta \Psi_T\|_{L^p}
 \lesssim (T^{\frac13})^{-\alpha-\frac32\beta-\frac52(1-\frac1p)} \|f\|_{L^q}.
\end{equs}
\end{remark}

We sometimes write $\Gamma$ for the integral kernel of $\mathcal{L}^{-1}P$, given by
\begin{equs} \label{eq:gamma}
 \widehat{\Gamma}(k) = \frac{1}{k_1^2+ |k_1|^{-1} k_2^2}, \quad \text{for } k_1\neq 0, \quad \text{and} \quad \widehat{\Gamma}(k) = 0 \quad \text{for } k_1=0.
\end{equs}
Note that 
\begin{equs} \label{eq:gamma_finite}
 \|\Gamma\|_{L^2}^2 = \sum_{k_1\neq 0} \frac{1}{(k_1^2+ |k_1|^{-1} k_2^2)^2} 
 \lesssim  \sum_{k_1\neq 0} \frac{1}{d(0,k)^4} <\infty.  
\end{equs}

\subsubsection{Definition of H\"older spaces} \label{sec:hoelder}

We now introduce the scale of H\"older seminorms based on the distance function $d$,
where we restrict ourselves to the range $\alpha\in(0,\frac{3}{2})$ needed in this work (see \cite[Definition 1]{IO19}). 

\begin{definition}\label{def:pos_holder}
For a function $f:\TT^2 \to \RR$ and $\alpha\in(0,\frac{3}{2})$, we define
\begin{equs}
{[f]_\alpha}
:= 
\begin{cases}
	\sup_{x\not=y}\frac{|f(y)-f(x)|}{d^\alpha(y,x)} & \text{for }\alpha\in(0,1], \\[1.5ex]
	\sup_{x\not=y}\frac{|f(y)-f(x)-\partial_1f(x)(y-x)_1|}{d^\alpha(y,x)} & \text{for } \alpha\in(1,\frac{3}{2}).
\end{cases}
\end{equs}
We denote by $\C^\alpha$ the closure of periodic $\C^\infty$-functions $f:\TT^2 \to \RR$ with respect to the norm $\|f\|_{L^\infty}+[f]_\alpha$.
\end{definition}

We will also need the following H\"older spaces of negative exponents. We will restrict to the range required in this
work, namely $\beta\in(-\frac{3}{2},0)$ (see \cite[Definition 3]{IO19}).

\smallskip

\begin{definition}\label{def:neg_holder}
Let $f$ be a periodic Schwartz distribution on $\TT^2$. For $\beta\in(-1,0)$ we define
\begin{align*}
[f]_\beta:=\inf\{|c|+[g]_{\beta+1}+[h]_{\beta+\frac{3}{2}}\, : \, f=c+\partial_1g+\partial_2h\}
\end{align*}
and for $\beta\in(-\frac{3}{2},-1]$ we define
\begin{align*}
[f]_\beta:=\inf\{|c|+[g]_{\beta+2}+[h]_{\beta+\frac{3}{2}}\, : \, f=c+\partial_1^2g+\partial_2h\}.
\end{align*}
We denote by $\C^\beta$ the closure of periodic $\C^\infty$-functions $f:\TT^2\to\RR$ with respect to the norm $[f]_\beta$.
\end{definition}

In Appendix \ref{app:hoelder_spaces} we provide all the necessary estimates on H\"older spaces needed in this work.

\subsection{Strategy of the proofs}
Recall the set $\mathcal{W}$ defined in \eqref{eq:W}, endowed with the strong topology in $L^2(\TT^2)$. We 
will show that the harmonic energy $\mathcal{H}(w)$ defined in \eqref{eq:harmonic-energy} controls the 
anharmonic part $\mathcal{E}(w)$ defined in \eqref{eq:energy} of the total energy, that is,
\begin{equs}
 \mathcal{E}(w) \lesssim 1+\mathcal{H}(w)^{2},
\end{equs}
for every $w\in \mathcal{W}$, and vice-versa, the anharmonic energy controls the harmonic part, that is, 
for every $\kappa>0$ we have 
\begin{align*}
	\HH(w) \lesssim_{\kappa} 1+ \EE(w)^{\frac{3}{2}+\kappa},
\end{align*}
for any $w\in\mathcal{W}$, see Proposition~\ref{prop:bound-harmonic} below.
By standard embedding theorems (see Lemma~\ref{lem:frac-sobolev}), any sublevel 
set of $\mathcal{H}$ (respectively $\mathcal{E}$) over $\mathcal{W}$ is relatively compact in $L^2$ and $\mathcal{H}$ (respectively $\mathcal{E}$) is lower semicontinuous with respect to the $L^2$-norm (see \eqref{eq:lsc}). 

In the following, for $\varepsilon>0$ sufficiently small, we will also write
\begin{align*}
	\triple = \left\{(\xi, v, F) \in \mathcal{C}^{-\frac{5}{4}-\varepsilon} \times 
\mathcal{C}^{\frac{3}{4}-\varepsilon} \times \mathcal{C}^{-\frac{3}{4}-\varepsilon}\, : \, 
\mathcal{L}v = P\xi \right\}.
\end{align*}
Note that $\triple$ is a closed subspace of $\mathcal{C}^{-\frac{5}{4}-\varepsilon} 
\times 
\mathcal{C}^{\frac{3}{4}-\varepsilon} \times \mathcal{C}^{-\frac{3}{4}-\varepsilon}$ endowed with the norm given by
$\max\{[\cdot]_{-\frac{5}{4}-\varepsilon},[\cdot]_{\frac{3}{4}-\varepsilon},[\cdot]_{-\frac{3}{4}-\varepsilon}\}$. 

The deterministic ingredient in the proof of Theorem~\ref{thm:Thm1}, that is Proposition~\ref{Satz2}, is essentially a consequence of the following theorem. 

\begin{theorem} \label{thm:minimizers} \hfill
	\begin{enumerate}[label=(\roman*)]
		\item \label{it:coercivity} (Coercivity) For every $\lambda\in(0,1)$ and
		$M>0$, there exists a constant $C>0$ which depends on $\lambda$ and polynomially on $M$
		\footnote{We say that a constant $C>0$ depends polynomially on $M$ if there exist $c>0$ 
		and $N\geq 1$ such that $C \leq c (1+M^N)$.}
		such that for every $\varepsilon\in (0, \frac1{100})$  and 
		every $(\xi,v,F)\in \triple$ with
		$[\xi]_{-\frac{5}{4}-\varepsilon}, 
		[v]_{\frac{3}{4}-\varepsilon}, [F]_{-\frac{3}{4}-\varepsilon}\leq M$,
		the functional $\mathcal{G}$ defined in \eqref{eq:G} satisfies 
		\begin{align*}
			|\mathcal{G}(v, F; w)| \leq \lambda \mathcal{E}(w) + C, \quad \text{for 
			every } w\in\mathcal{W}.
		\end{align*}
		In particular, $E_{ren}(v,F;\cdot)$ defined in \eqref{eq:renormalized-energy} is coercive.
		\item \label{it:continuity} (Continuity) Let $\eps\in (0, \frac1{100})$ and $(\xi_{\ell}, v_{\ell}, F_{\ell}) \to (\xi, v, F)$ in $\triple$ and $w_{\ell} \to w$ in $\mathcal{W}$ 
		with the property that $\limsup_{\ell\to 0} \mathcal{E}(w_{\ell})<\infty$. Then 
		\begin{align*}
			\mathcal{G}(v_{\ell}, F_{\ell}; w_{\ell}) \to \mathcal{G}(v,F; w) \quad 
			\text{as} \quad \ell \to 0. 
		\end{align*}
		\item \label{it:compactness} (Compactness) Let $\eps\in (0, \frac1{100})$ and $(\xi,v,F) \in\triple$ be fixed. 
		Then for any $M\in\RR$ the sublevel sets of $E_{ren}(v,F;\cdot)$ defined in \eqref{eq:renormalized-energy}, given by 
		\begin{align*}
			\left\{w \in L^2 : \int_0^1 w \,\mathrm{d}x_1 = 0, \, 
			E_{ren}(v, F; w) \leq M \right\}, 
		\end{align*}
		are compact in the $L^2$-norm.
		\item \label{it:existence_min} (Existence of minimizers) If $\eps\in (0, \frac1{100})$ and $(\xi,v,F) \in\triple$, then there 
		exists a minimizer $w\in\mathcal{W}$ of the renormalized energy $E_{ren}(v,F;\cdot)$ which is a weak solution
		of \eqref{eq:ripple_rem}. 
	\end{enumerate}
\end{theorem}

Note that $E_{ren}(v,F;0) = 0$, therefore every minimizer of $E_{ren}$ belongs to the 
sublevel set $M=0$ of $E_{ren}$. Using Theorem~\ref{thm:minimizers}, we obtain the following 
$\Gamma$-convergence
result. 

\begin{corollary}[$\Gamma$-convergence] \label{cor:gamma-convergence}
Let $\eps\in (0, \frac1{100})$ and $(\xi_{\ell}, v_{\ell}, F_{\ell}) \to (\xi, v, F)$ in $\triple$. Then 
	\begin{equs}
		E_{ren}(v_{\ell}, F_{\ell}; w) \to E_{ren}(v, F; w) \quad \text{for every} \, \, w\in \mathcal W \, \, \textrm{ as} \quad \ell \to 0.  
	\end{equs}
Also, $E_{ren}(v_{\ell}, F_{\ell}; \cdot) \to E_{ren}(v, F; \cdot)$ in the sense of 
$\Gamma$-convergence over $\mathcal W$, that is, 
	\begin{enumerate}[label=(\roman*)]
	 \item $(\Gamma-\liminf)$ For all sequences $\{w_{\ell}\}_{\ell \downarrow 0} \subset 
		\mathcal{W}$ with $w_\ell \to w $ strongly in $L^2$, we have 
		\begin{align*}
		\liminf_{\ell\to 0} E_{ren}(v_{\ell}, F_{\ell}; w_\ell) \geq E_{ren}(v,F;w). 
		\end{align*}
	 \item $(\Gamma-\limsup)$ For every $w\in \mathcal{W}$, there exists a sequence $\{w_{\ell}\}_{\ell \downarrow 0} \subset \mathcal{W}$
	 with $w_\ell \to w$ strongly in $L^2$ such that 
	 	\begin{align*} 
	 		\lim_{\ell\to 0} E_{ren}(v_{\ell}, F_{\ell}; w_\ell) = E_{ren}(v,F;w).
	 	\end{align*}
	\end{enumerate}
	\end{corollary}

\begin{proof}[Proof of Proposition~\ref{Satz2}]
	Corollary \ref{cor:gamma-convergence} establishes the continuity of the map that associates
	to each $(\xi, v, F)\in\triple$ the functional $E_{ren}(v,F; \cdot)$, when the space of
	(lower semicontinuous) functionals over $\mathcal{W}$ is equipped with the topology of $\Gamma$-convergence
	(with respect to the $L^2$-topology on $\mathcal{W}$). In particular, this map is Borel measurable when $\triple$ is
	endowed with its Borel $\sigma$-algebra. 
	\qedhere
\end{proof}

Taking the main stochastic ingredient from Proposition~\ref{Satz1_pavlos} for granted (which we prove in Section~\ref{sec:nongaussian}), 
we can now give the proof of Theorem~\ref{thm:Thm1}.

\begin{proof}[Proof of Theorem~\ref{thm:Thm1}]  
Let $\langle\cdot\rangle$ be a probability measure on the space of periodic Schwartz distributions $\xi$ that satisfies
Assumption~\ref{ass}. By Proposition~\ref{Satz1_pavlos} $\langle\cdot\rangle$ lifts 
to a probability measure $\langle\cdot\rangle^{\mathrm{lift}}$ on the space of triples 
$(\xi,v,F)\in{\mathcal C}^{-\frac{5}{4}-}\times{\mathcal C}^{\frac{3}{4}-}\times {\mathcal C}^{-\frac{3}{4}-}$.

By Proposition~\ref{Satz2} the mapping $(\xi,v,F) \mapsto E_{ren}(v,F;\cdot)$ is continuous. Hence, the push-forward $(E_{ren})_{\#}\lng\cdot\rng^{\mathrm{lift}}$
is well-defined as a probability measure on the space of lower semicontinuous functionals equipped with the 
Borel $\sigma$-algebra corresponding to the topology of $\Gamma$-convergence (based on the strong $L^2$-topology).
We now define $\lng\cdot\rng^{\mathrm{ext}}$ as the joint law of $\xi$ and $E_{ren}(v,F;\cdot)$.\footnote{Here we
understand $E_{ren}(v,F;\cdot)$ as a measurable function of $\xi$, which is a composition of the measurable function 
$\xi\mapsto (\xi,v,F)$ and the continuous function $(\xi,v,F) \mapsto E_{ren}(v,F;\cdot)$.} 
\begin{enumerate}[label=(\roman*)]
 \item This is immediate by the definition of $\lng\cdot\rng^{\mathrm{ext}}$ and the fact that by Proposition~\ref{Satz1_pavlos}, $\lng\cdot\rng$ is 
 supported on the H\"older space 
 $\C^{-\frac{5}{4}-}$.
 \item If $\xi$ is smooth $\langle\cdot\rangle$-almost surely, by Proposition~\ref{Satz1_pavlos}~\ref{it:Satz1_pavlos_F} 
 we have that $F=v\partial_2 R_1 v$ $\langle\cdot\rangle$-almost surely. In this case, $E_{ren}(v,F;\cdot) = E_{ren}(v, v\partial_2 R_1 v;\cdot)$
 $\lng \cdot \rng$-almost surely and agrees with the definition given in \eqref{eq:ren_energy_intro}.
 \item Let $\{\langle\cdot\rangle_{\ell}\}_{\ell\downarrow0}$ be a sequence of probability measures 
 that satisfy Assumption~\ref{ass} and converges weakly to
 $\langle\cdot\rangle$, which then automatically satisfies Assumption~\ref{ass}.
 Then by Proposition~\ref{Satz1_pavlos}~\ref{it:Satz1_pavlos_cont}, the sequence
 $\{\langle\cdot\rangle_{\ell}^{\mathrm{lift}}\}_{\ell\downarrow 0}$ converges weakly to $\langle\cdot\rangle^{\mathrm{lift}}$.
 Given a bounded continuous function $G:(\xi,E)\mapsto G(\xi,E)\in \RR$ we have that
 \begin{equs}
  \lng G \rng^{\mathrm{ext}}_\ell 
  = \lng G(\xi, E_{ren}(v,F;\cdot)) \rng_\ell^{\mathrm{lift}} \stackrel{\ell\downarrow0}{\longrightarrow} 
  \lng G(\xi, E_{ren}(v,F;\cdot)) \rng^{\mathrm{lift}} =  \lng G \rng^{\mathrm{ext}},
 \end{equs}
 which in turn implies that $\lng \cdot \rng^{\mathrm{ext}}_\ell\to \lng \cdot \rng^{\mathrm{ext}}$
 weakly as $\ell\downarrow0$. \qedhere
\end{enumerate}
\end{proof}

Finally, we have an a priori estimate for the $\C^{\frac{5}{4}-}$ norm of minimizers of $E_{ren}(v, F; \cdot)$, 
which we prove in Section \ref{sec:regularity}. 

\begin{proposition}[H\"older regularity] \label{prop:hoelder_regularity} For any $M>0$ and $\varepsilon\in (0, \frac1{100})$, there exists a constant $C>0$ 
		which depends on $\eps$ and polynomially on $M$ such that 
		\begin{align*}
		[w]_{\frac{5}{4}-2\varepsilon} \leq C,
		\end{align*}
		for every minimizer $w\in \mathcal W\cap \C^{\frac{5}{4}-2\epsilon}$ of $E_{ren}(v,F;\cdot)$ with $(\xi,v,F)\in \triple$ satisfying
		the bound $[\xi]_{-\frac{5}{4}-\varepsilon}, [v]_{\frac{3}{4}-\varepsilon}, [F]_{-\frac{3}{4}-\varepsilon}\leq M$.
\end{proposition}

Combined with an approximation argument, this is the main ingredient in the proof of Corollary~\ref{cor:regularity}, which we 
give now. 

\begin{proof}[Proof of Corollary~\ref{cor:regularity}] For $\eps\in (0,\frac{1}{100})$, we define the functional $g:\mathcal{W} \to \RR\cup \{+\infty\}$ given by
\begin{equs}
 g(w) := 
 \begin{cases}
  [w]_{\frac{5}{4}-2\eps}, & \text{if} \quad w\in \C^{\frac{5}{4}-2\eps},
  \\
  +\infty, & \text{otherwise}.
 \end{cases}
\end{equs}
By \cite[Lemma 13]{IO19} we know that $g$ is lower semicontinuous on $\mathcal{W}$ endowed with the strong topology in $L^2(\TT^2)$. 
We define the non-negative functional $G:E\mapsto G(E)$ on the space of lower semicontinuous functionals $E$ on $\mathcal{W}$ by
\begin{equs}
 G(E) := 
 \begin{cases}
  \displaystyle{\inf_{w\in \mathrm{argmin} E} g(w)}, & \text{if} \quad \mathrm{argmin} E \neq \emptyset,
  \\
  +\infty, & \text{otherwise}.
 \end{cases}
\end{equs}
We claim that $G$ is lower semicontinuous, that is, if $E_\ell\to E$ as $\ell\downarrow0$ in the sense of $\Gamma$-convergence, then 
\begin{equs}
 G(E) \leq \liminf_{\ell\downarrow0} G(E_\ell). 
\end{equs}
Indeed, without loss of generality we may assume that $G(E_\ell) \to\liminf_{\ell\downarrow0} G(E_\ell)<\infty$ by possibly 
extracting a subsequence. This implies that $\sup_{\ell} G(E_\ell)<\infty$, hence by the definition of $G$ there exists a sequence
of minimizers $w_\ell$ of $E_\ell$ such that 
\begin{equs}
 {[w_\ell]}_{\frac{5}{4}-2\eps} \leq G(E_\ell) + \ell \leq \sup_{\ell} G(E_\ell) +1. 
\end{equs}
By Lemma \ref{lem:comp_emb} there exists $w\in \C^{\frac{5}{4}-2\eps}$ such that $w_\ell \to w$ in $\C^{\frac{5}{4}-3\eps}$ along a subsequence,
and 
\begin{equs}
 {} [w]_{\frac{5}{4}-2\eps} \leq \liminf_{\ell\downarrow0} [w_\ell]_{\frac{5}{4}-2\eps}.
\end{equs}
This, in particular, implies that $w_\ell\to w$ strongly in $L^2(\TT^2)$ and since $E_\ell\to E$ in the sense of $\Gamma$-convergence, 
$w$ is a minimizer of $E$. Thus, we have the estimate
\begin{equs}
 G(E) \leq [w]_{\frac{5}{4}-2\eps} \leq \liminf_{\ell\downarrow0} [w_\ell]_{\frac{5}{4}-2\eps} \leq \liminf_{\ell\downarrow0} G(E_\ell),  
\end{equs}
which proves the desired claim. 

Let now $\{\lng\cdot\rng_{\ell}\}_{\ell\downarrow0}$ be a sequence of probability measures such that $\lng\cdot\rng_\ell\to \lng\cdot\rng$ weakly and
for every $\ell\in(0,1]$, $\xi$ is smooth $\lng\cdot\rng_\ell$-almost surely. Since under $\lng\cdot\rng_{\ell}$, $\xi$ is smooth, 
by Lemma \ref{lem:qualitative-regularity-linear} $v$ is smooth. By Theorem \ref{thm:minimizers} \ref{it:existence_min} there
exists a minimizer $w\in \mathcal{W}$ of $E_{ren}(v,F;\cdot)$, which is a weak solution to \eqref{eq:ripple_rem}. If we let $u=v+w$, 
then $u\in \mathcal{W}$ and since $F=vR_1\partial_1 v$ $\lng\cdot\rng_{\ell}$-almost surely (see Proposition~\ref{Satz1_pavlos}~\ref{it:Satz1_pavlos_F}), $u$ is a weak solution to \eqref{eq:ripple}. 
By Proposition \ref{prop:H3} we know that $\||\partial_1|^s u\|_2 + \||\partial_2|^{\frac{2}{3}s}u\|_2 \lesssim 1$ for every $s<3$,
hence by Lemma~\ref{lem:hoelder-embedding} $u\in \C^{\frac{5}{4}-2\eps}$, which in turn implies that $w=u-v\in \C^{\frac{5}{4}-2\eps}$. By
Proposition~\ref{prop:hoelder_regularity} we have the estimate 
\begin{equs}
 {} [w]_{\frac{5}{4}-2\eps} \leq C,
\end{equs}
where the constant $C$ depends polynomially on $\max\{[\xi]_{-\frac{5}{4}-\eps}, [v]_{\frac{3}{4}-\eps}, [F]_{-\frac{3}{4}-\eps}\}$. In particular,
this implies that
\begin{equs}
 G(E_{ren}(v,F;\cdot)) \leq C.
\end{equs}
By Corollaries~\ref{cor:xi_v} and \ref{cor:F} we know that for every $1\leq p<\infty$,
\begin{equs}
 \sup_{\ell} \lng C^p \rng^{\mathrm{lift}}_\ell \lesssim_p 1.
\end{equs}
Hence, for the functional $G$ we have that
\begin{equs}
 \sup_{\ell}\lng G(E)^p \rng_\ell^{\mathrm{ext}} = \sup_{\ell}\lng G(E_{ren}(v,F;\cdot))^p \rng^{\mathrm{lift}}_\ell 
 \leq \sup_{\ell} \lng C^p \rng^{\mathrm{lift}}_\ell \lesssim_p 1.
\end{equs}
Since by Theorem \ref{thm:Thm1} \ref{item:thm:Thm1-convergence} $\lng\cdot\rng^{\mathrm{ext}}_\ell\to \lng\cdot\rng^{\mathrm{ext}}$ weakly and 
$G$ is lower semicontinuous we have that

\begin{equs}
 \lng G(E)^p \rng^{\mathrm{ext}} \leq \liminf_{\ell\downarrow0} \lng G^p \rng_\ell^{\mathrm{ext}} \lesssim_p 1,
\end{equs}
which completes the proof. 
\end{proof}

\subsection{Outline} In Section \ref{s:burgers} we show how the {Howarth--K\'arm\'an--Monin} 
identities can be used to control certain Besov and $L^p$ norms by the anharmonic energy 
$\mathcal{E}$. 

In Section \ref{sec:gamma-convergence} we prove Theorem~\ref{thm:minimizers} 
and the $\Gamma$-convergence result for the renormalized energy, see Corollary \ref{cor:gamma-convergence}. 

In Section \ref{sec:regularity} we prove the optimal H\"older regularity $\frac{5}{4}-$ of minimizers
of the renormalized energy, see Proposition~\ref{prop:hoelder_regularity}. 

In Section \ref{sec:nongaussian}, based on the spectral gap inequality \eqref{eq:SG}, we provide the 
stochastic arguments to prove Proposition~\ref{Satz1_pavlos}.

Last, in the \hyperlink{appendix}{Appendix} we include some technical results and proofs.  


%% file: figure-2.tex
\begin{figure}[ht]
\begin{tikzpicture}

\node at (0,0) {$\lng \cdot\rng\sim\xi$};
\node at (.66,2.25) {$\lng\cdot\rng^{\mathrm{lift}}\sim(\xi,v,F)$};
\node at (6.25,2.25) {$\lng\cdot\rng^{\mathrm{ext}}\sim \left(\xi,E_{ren}(v,F;\cdot)\right)$};

\draw[->] (-.32,.35) -- (-.32,1.85);
\draw[->] (2,2.25) -- (4.25,2.25);
\draw[->,dashed] (0,.25) -- (4.4,1.9);
 
\end{tikzpicture}
\caption{\label{fig:ripple} Construction of the extension measure $\lng\cdot\rng^{\mathrm{ext}}$. The vertical arrow corresponds to the 
probabilistic step (Proposition~\ref{Satz1_pavlos}), while the horizontal arrow is the deterministic step (Proposition~\ref{Satz2}). 
For smooth $\xi$'s, $F$ is given by $vR_1\partial_2 v$.}
\end{figure}
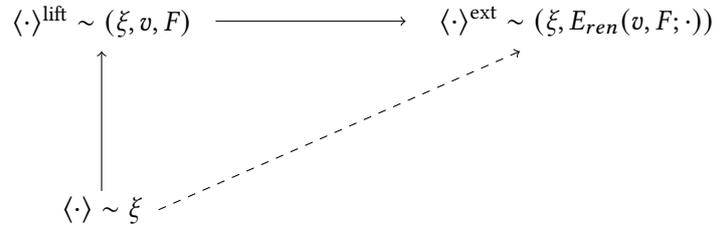

%% file: burgers.tex
\section{Estimates for the Burgers equation} \label{s:burgers}

In this section we bound certain Besov and $L^p$ norms of a function $w\in\mathcal{W}$ by the anharmonic energy $\mathcal{E}(w)$. These bounds will be used in later
sections to study the $\Gamma$-convergence of the renormalized energy \eqref{eq:renormalized-energy} and regularity properties of its minimizers 
(see Sections \ref{sec:gamma-convergence} and \ref{sec:regularity} below). The proof of these estimates is based on the application of the {Howarth--K\'arm\'an--Monin} identity for the Burgers operator.   

We first need to introduce (directional) Besov spaces. These spaces appear naturally through the application of the {Howarth--K\'arm\'an--Monin}
identity (see Proposition \ref{prop:HKM}). 

Throughout this section, for a function $f : \TT^2 \to \RR$ we write 
\begin{align*}
\partial_j^h f(x) := f(x+he_j)-f(x)
\end{align*}
where $x\in \TT^2$, $j\in\{1,2\}$, $h\in \RR$, $e_1=(1,0)$ and $e_2=(0,1)$. 

\begin{definition} \label{def:besov_s_p} 
	For a function $f:\TT^2\to\RR$, $j\in\{1,2\}$, $s\in (0,1]$ and $p\in [1, \infty)$ we define the following
	(directional) Besov seminorm\footnote{Note that one can take the supremum over all $h\in \RR$ in \eqref{eq:defSob} 
	by replacing $h^s$ with $|h|^s$. Indeed, if $h\in [-1,0)$ the quantity on the right-hand side of \eqref{eq:defSob} does 
	not change by symmetry. If $h\in \RR\setminus [-1,1]$, one writes $h=h_{\mathrm{fr}}+h_{\mathrm{int}}$ with $h_{\mathrm{fr}}\in (0,1]$ and $h_{\mathrm{int}}\in \ZZ$ and 
	uses that $\|\partial_j^h f\|_{L^p}=\|\partial_j^{h_{\mathrm{fr}}} f\|_{L^p}$ (by periodicity) while $\frac{1}{|h|^s}\leq \frac{1}{h_{\mathrm{fr}}^s}$. }
	\begin{equation}
		\|f\|_{\dotB^{s}_{p;j}}:= 
		\sup_{h\in(0,1]}\frac{1}{h^s}\left(\int_{\TT^2} |\partial_j^hf(x)|^p \, \dd x\right)^{\frac1p}. \label{eq:defSob}
	\end{equation}
\end{definition}

Notice that in comparison to standard Besov spaces our definition measures regularity in $x_1$ and $x_2$ 
separately. We have also omitted the second lower index which usually appears in standard Besov 
spaces since in our case it is always $\infty$ (corresponding to ${\dotB^{s}_{p,\infty}}$).  

\begin{remark}\label{rem:32}
	For $s\geq 0$, given a periodic function $f:\TT^2\to\RR$, we define $|\partial_j|^s f$ in Fourier space via 
	\begin{align*}
		\widehat{|\partial_j|^s f}(k) := |k_j|^s \widehat{f}(k), \quad k\in (2\pi \ZZ)^2.
	\end{align*}
	For $s<0$ and a periodic distribution $f$ of vanishing average in $x_j$\footnote{We say that a periodic distribution $f$ has vanishing average in $x_1$ if $f(\mathrm{e}^{-\ii k_2 \cdot}) = 0$ for all $k_2 \in 2\pi\ZZ$, and analogously for $f$ with vanishing average in $x_2$.}, we can define $|\partial_j|^s f$ in the same way for $k_j\neq 0$.
	For $p=2$, $s\in (0,1)$ and $s'\in(s,1)$, the Parseval identity implies the equivalence
	\footnote{Note that $C(s, s')\to \infty$ as $s'\searrow s$.}
	\begin{equation}
		\int_{\TT^2} \left||\partial_j|^s f\right|^2\,\mathrm{d}x = \sum_{k\in (2\pi \ZZ)^2} |k_j|^{2s} |\widehat{f}(k)|^2 
		= c_s \int_\RR \frac{1}{|h|^{2s}} \int_{\TT^2} |\partial_j^h f(x)|^2 \, \dd x \frac{\dd h}{|h|}
		\lesssim C(s,s') \|f\|_{\dotB^{s'}_{2;j}}^2 
		\label{eq:fourier}
	\end{equation}
	for some positive constant $C(s, s')$ depending only on $s$ and $s'$, where we used 
	\eqref{eq:sobolev_to_besov} below. 
\end{remark}

In the next proposition we prove two core estimates based on the 
{Howarth--K\'arm\'an--Monin} identities \cite[Lemma 4.1]{GJO15} 
for the Burgers operator. In \cite[Lemma 4.1]{GJO15} the authors deal with the operator $w\mapsto \partial_2 w + \partial_1 \frac{1}{2} w^2$, but the
same proof extends to our setting. 

\begin{proposition} \label{prop:HKM} There exists $C>0$ such that for every $w\in L^2(\TT^2)$ with vanishing average in $x_1$ and for every $h\in(0,1)$ we have
	\begin{align}
		\int_{\TT^2} |\partial_1^h w|^3 \, \dd x & \le C h^{\frac32} \EE(w), \label{eq:GJO_mod} 
		\\
		\sup_{x_2\in[0,1)} \frac{1}{h}\int_0^h \int_0^1 |\partial_1^{h'} w|^2 \,\mathrm{d}x_1 \mathrm{d} h' 
		& \leq C h^{\frac{1}{2}} \mathcal{E}(w).
		\label{eq:GJO_clas}
	\end{align}
\end{proposition}

\begin{proof} 
	By the {Howarth--K\'arm\'an--Monin} identities \cite[Lemma 4.1]{GJO15} for the Burgers operator we know that 
	for every $h'\in(0,1)$
	\begin{align}
		\partial_2\frac12\int_0^1|\partial_1^{h'} w| \partial_1^{h'} w \, \dd x_1-\partial_{h'}\frac16\int_0^1|\partial_1^{h'} w|^3 \, \dd x_1 & =
		\int_0^1 \partial_1^{h'} \eta_w |\partial_1^{h'} w| \, \dd x_1, \label{eq:KHM_mod} \\
		\partial_2 \frac{1}{2} \int_0^1 |\partial_1^{h'} w|^2 \,\mathrm{d}x_1 - \partial_{h'} \frac{1}{6} \int_0^1 (\partial_1^{h'} w)^3\,\mathrm{d}x_1
		& = \int_0^1 \partial_1^{h'} \eta_w \partial_1^{h'} w \,\mathrm{d}x_1. \label{eq:KHM_clas} 
	\end{align}

	To prove \eqref{eq:GJO_mod}, we integrate \eqref{eq:KHM_mod} over $x_2$ and use the periodicity of $w$ to obtain,
	\begin{align}
		\partial_{h'} \int_{\TT^2} |\partial_1^{h'} w|^3 \, \dd x = -6\int_{\TT^2} \eta_w \partial_1^{-h'} |\partial_1^{h'} w| \, \dd x. \label{eq:inter1}
	\end{align}
	The last term is estimated as follows,  
	\begin{align*}
		\left( \int_{\TT^2} \eta_w \partial_1^{-h'} |\partial_1^{h'} w| \, \dd x \right)^2 & \leq  \int_{\TT^2} (|\partial_1|^{-\frac12}\eta_w)^2 \, \dd x
		\int_{\TT^2} (|\partial_1|^{\frac12}\partial_1^{-h'} |\partial_1^{h'} w|)^2 \, \dd x\\
		&\lesssim |h'| \int_{\TT^2} (|\partial_1|^{-\frac12}\eta_w)^2 \, \dd x \int_{\TT^2} (\partial_1 w)^2 \, \dd x, 
	\end{align*}
	where we use that
	\begin{align*}
		\int_{\TT^2} (|\partial_1|^{\frac12}\partial_1^{-h'} |\partial_1^{h'} w|)^2 \, \dd x 
		& \lesssim \left(\int_{\TT^2} (\partial_1^{-h'} |\partial_1^{h'} w|)^2 \, \dd x\right)^{\frac12}
		\left(\int_{\TT^2} (\partial_1 |\partial_1^{h'} w|)^2 \, \dd x\right)^{\frac12} \\
		& \lesssim |h'| \int_{\TT^2} (\partial_1 |\partial_1^{h'} w|)^2 \, \dd x.
	\end{align*}
	Integrating \eqref{eq:inter1} over $h'\in (0,h)$, we obtain that
	\begin{align*}
		\int_{\TT^2} |\partial_1^h w(x)|^3 \, \dd x \lesssim h^{\frac32} \left( \int_{\TT^2} (\partial_1 w)^2 \, \dd x\right)^{\frac12}
		\left(\int_{\TT^2} (|\partial_1|^{-\frac12} \eta_w)^2 \, \dd x\right)^{\frac12}
	\end{align*}
	which in turns implies \eqref{eq:GJO_mod}.
	
	To prove \eqref{eq:GJO_clas}, we integrate \eqref{eq:KHM_clas} over $h'\in (0,h)$ to obtain with $\partial_1^0 w=0$
	\begin{align*}
		\partial_2 \frac{1}{2} \int_0^h \int_0^1 |\partial_1^{h'} w|^2 \,\mathrm{d}x_1 \mathrm{d} h' - \frac{1}{6} \int_0^1 (\partial_1^h w)^3\,\mathrm{d}x_1 
		= \int_0^h \int_0^1 \partial_1^{h'} \eta_w \partial_1^{h'} w \,\mathrm{d}x_1 \, \mathrm{d}h'. 
	\end{align*}
	By the Sobolev embedding $W^{1,1}(\TT) \subset L^{\infty}(\TT)$ on the torus $\TT=[0,1)$ in the form 
	\begin{align*}
		\sup_{z\in\TT}|f(z)| \leq \int_{\TT} |f(z)|\,\mathrm{d}z + \int_{\TT} |f'(z)|\,\mathrm{d}z,
	\end{align*}
	we can therefore estimate
	\begin{align*}
		\sup_{x_2\in[0,1)} \int_0^h \int_0^1 |\partial_1^{h'} w|^2 \,\mathrm{d}x_1 \mathrm{d}h' 
		&\lesssim \int_0^h \int_{\TT^2} |\partial_1^{h'} w|^2 \,\mathrm{d}x \mathrm{d}h' 
		+ \int_{\TT^2} |\partial_1^h w|^3\,\mathrm{d}x \\
		&\quad + \int_0^h \int_0^1 \left| \int_0^1 \partial_1^{h'} \eta_w \partial_1^{h'} w \,\mathrm{d}x_1 \right|\mathrm{d}x_2\mathrm{d}h'.
	\end{align*}
	The first term on the right-hand side can be bounded using 
	\begin{align*}
		\int_{\TT^2} |\partial_1^{h'} w|^2\,\mathrm{d}x \leq (h')^2 \int_{\TT^2} |\partial_1 w|^2\,\mathrm{d}x \leq (h')^2 \mathcal{E}(w).
	\end{align*}
	For the second term we use \eqref{eq:GJO_mod}. Last, for the third term, the same argument used to estimate the right-hand side of 
	\eqref{eq:inter1} leads to 
	\begin{align*}
		\int_0^h \int_0^1 \left| \int_0^1 \partial_1^{h'} \eta_w \partial_1^{h'} w \,\mathrm{d}x_1 \right|\mathrm{d}x_2\mathrm{d}h' 
		\lesssim \| |\partial_1|^{-\frac{1}{2}} \eta_w\|_{L^2} \|\partial_1 w\|_{L^2} \int_0^h (h')^{\frac{1}{2}} \,\mathrm{d}h' 
		\lesssim h^{\frac{3}{2}} \mathcal{E}(w). 
	\end{align*}
	Hence, we can bound 
	\begin{align*}
		\sup_{x_2\in[0,1)} \int_0^h \int_0^1 |\partial_1^{h'} w|^2 \,\mathrm{d}x_1 \mathrm{d}h' 
		&\lesssim (h^3 + h^{\frac{3}{2}}) \mathcal{E}(w) \lesssim h^{\frac{3}{2}} \mathcal{E}(w)
	\end{align*}
	for all $h\in(0,1)$. 
\end{proof}

We are now ready to prove the main result of this section. 

\begin{proposition} \label{prop:a_priori_est}
	We have the following estimates:
	\begin{align}
		& \text{(i)} && && && && & \|w\|_{\dotB^s_{3;1}} & \leq C \mathcal{E}(w)^{\frac{1}{3}}, \text{ for every } 
		s\in(0,\tfrac{1}{2}], && && && && \label{eq:besov_s_3_1_contr} \\
		& \text{(ii)} && && && && & \|w\|_{\dotB^s_{2;1}} &  \leq C \mathcal{E}(w)^{\frac{2s+1}{6}}, \text{ for every } s\in[\tfrac{1}{2},1], && && && && \label{eq:besov_s_2_1_contr} \\
		& \text{(iii)} && && && && & \|w\|_{L^p} & \leq C(p) \mathcal{E}(w)^{\frac{p-1}{2p}}, \text{ for every } p\in[3,7), && && && && \label{eq:L^p_contr} \\
		& \text{(iv)} && && && && & \|w^2\|_{\dotB^\frac{2p-6}{2p}_{2;1}} & \leq C(p) \mathcal{E}(w)^{\frac{2p-3}{2p}}, \text{ for every } p\in[6,7), && && && && \label{eq:besov_w^2_1/2_2_1_contr}
		%
		%
	\end{align}
	for every $w \in L^2(\TT^2)$ with vanishing average in $x_1$, where $C>0$ is a universal constant and $C(p)>0$ depends on $p$.
\end{proposition}

Note that a result similar to \eqref{eq:besov_s_3_1_contr} was obtained in \cite[Lemma 4]{OS10} using different 
techniques.

\begin{proof}
	\begin{enumerate}[label=(\roman*),leftmargin=0pt,labelsep=*,itemindent=*]
	\item This is immediate from \eqref{eq:GJO_mod}, Definition \ref{def:besov_s_p} and \eqref{eq:besov_s_p_q}.
		
	\item By interpolation we have for $s\in[\tfrac{1}{2},1]$:
	\begin{align*}
		\|w\|_{\dotB^s_{2;1}} \leq \|w\|_{\dotB^{\frac{1}{2}}_{2;1}}^{2(1-s)} \|w\|_{\dotB^1_{2;1}}^{2s-1}.
	\end{align*}
	Using \eqref{eq:besov_s_p_q} and \eqref{eq:besov_s_3_1_contr} we get
	\begin{align*}
		\|w\|_{\dotB^{\frac{1}{2}}_{2;1}} \leq 
		\|w\|_{\dotB^{\frac{1}{2}}_{3;1}} 
		\lesssim \mathcal{E}(w)^{\frac{1}{3}},
	\end{align*}
	with an implicit constant independent of $s$. We also have the bound
	\begin{align*}
		\|w\|_{\dotB^1_{2;1}} \leq \|\partial_1w\|_{L^2} \leq \mathcal{E}(w)^{\frac{1}{2}}.
	\end{align*}
	Combining these estimates implies \eqref{eq:besov_s_2_1_contr}.
	
	\item We divide the proof into several steps. 
	
	\begin{enumerate}[label=\textsc{Step \arabic*},leftmargin=0pt,labelsep=*,itemindent=*]
	\item \label{item:P24-1}\hskip-1ex: We first prove that
	\begin{align*}
		\sup_{x_2 \in [0,1)} \left(\int_0^1 |w(x)|^p \,\mathrm{d}x_1 \right)^{\frac{1}{p}} \lesssim_p \mathcal{E}(w)^{\frac{1}{2}} \quad \text{for all} \quad 2\leq p < 4.
	\end{align*}
	Indeed, by \eqref{eq:1d_besov_emb} for $p\in[2,4)$, $q=2$ and $f = w(\cdot,x_2)$ (with $x_2\in [0,1)$ fixed) we know that  
	\begin{align*}
		\sup_{x_2\in[0,1)} \left(\int_0^1 |w(x)|^p \, \dd x_1\right)^{\frac{1}{p}} \lesssim_p 
		\sup_{x_2\in[0,1)} \int_0^1 \left(\int_0^1 |\partial_1^hw(x)|^2 \, \dd x_1\right)^{\frac{1}{2}} 
		\frac{1}{h^{\frac{1}{2}-\frac{1}{p}}} \frac{\dd h}{h}.
	\end{align*}
	Since $p<4$ we have that
	\begin{equs}
		\sup_{x_2\in[0,1)} \int_0^1 \left(\int_0^1 |\partial_1^hw(x)|^2 \, \dd x_1\right)^{\frac{1}{2}} 
		\frac{1}{h^{\frac{1}{2}-\frac{1}{p}}} \frac{\dd h}{h} 
		\lesssim_p  \sup_{x_2\in[0,1)} \sup_{h\in(0,1]} \frac{1}{h^{\frac{1}{4}}} \left(\int_0^1 |\partial_1^hw(x)|^2 \, \dd x_1 \right)^{\frac{1}{2}}.
	\end{equs}
	By \eqref{eq:h_aver_contr} and \eqref{eq:GJO_clas} we also know that
	\begin{equs}
		&\sup_{x_2\in[0,1)} \sup_{h\in(0,1]} \frac{1}{h^{\frac{1}{4}}} \left(\int_0^1 |\partial_1^hw(x)|^2 \, \dd x_1 \right)^{\frac{1}{2}} \\
		&\quad \lesssim \sup_{x_2\in[0,1)} \sup_{h\in(0,1]} \frac{1}{h^{\frac{1}{4}}} \left(\frac{1}{h}\int_0^h \int_0^1 |\partial_1^{h^\prime}w(x)|^2 \, \dd x_1 \dd h^\prime\right)^{\frac{1}{2}} 
		 \lesssim \mathcal{E}(w)^{\frac{1}{2}},
	\end{equs}
	which combined with the previous estimates implies the desired estimate. 
	
	\item \label{item:P24-2}\hskip-1ex: We prove that
	\begin{align*} 
		\left(\int_0^1 \sup_{x_1\in[0,1)} |w(x)|^3 \,\mathrm{d}x_2 \right)^{\frac{1}{3}} \lesssim \mathcal{E}(w)^{\frac{1}{3}}.
	\end{align*}
	By \eqref{eq:1d_besov_emb} for $p=\infty$, $q=3$ and $f = w(\cdot,x_2)$, we know that
	\begin{align*}
		\sup_{x_1\in[0,1)}|w(x)| & \lesssim \int_0^1 \left( \frac{1}{h} \int_0^1 |\partial_1^h w(x)|^3 \,
		\mathrm{d} x_1\right)^{\frac{1}{3}} \frac{\mathrm{d} h}{h}
	\end{align*}
	for every $x_2\in[0,1)$. Using Minkowski's inequality we obtain the bound
	\begin{align*}
		\left( \int_0^1 \sup_{x_1\in[0,1)}|w(x)|^3 \, \mathrm{d} x_2 \right)^{\frac{1}{3}}
		& \lesssim \int_0^1 \left( \frac{1}{h} \int_{\TT^2} |\partial_1^h w(x)|^3 \dd x \right)^\frac{1}{3} \frac{\mathrm{d} h}{h}.
	\end{align*}
	Using \eqref{eq:GJO_mod}, the last term in the above inequality is bounded by
	\begin{align*}
		\int_0^1 \left( \int_{\TT^2} |\partial_1^h w(x)|^3 \dd x \frac{1}{h} \right)^\frac{1}{3} \frac{\mathrm{d} h}{h} 
		& \lesssim \int_0^1 \frac{1}{h^{\frac{5}{6}}} \mathcal{E}(w)^\frac{1}{3} \,\dd h 
	\end{align*}
	which implies the desired estimate.  
	
	\item \label{item:P24-3}\hskip-1ex: We are now ready to prove \eqref{eq:L^p_contr}. For $5\leq p < 7$ this is immediate from \ref{item:P24-1} and \ref{item:P24-2}, since we have that
	\begin{align*}
		&\left(\int_{\TT^2} |w(x)|^{p} \,\mathrm{d} x\right)^\frac{1}{p} 
		\lesssim \left( \int_0^1 \sup_{x_1\in[0,1)} |w(x)|^3 \, \mathrm{d} x_2 \sup_{x_2\in[0,1)} \int_0^1 |w(x)|^{p-3} \, \mathrm{d} x_1 \right)^\frac{1}{p}
		\lesssim_p \mathcal{E}(w)^\frac{p-1}{2p}.
	\end{align*}
	\ref{item:P24-2} also implies that $\|w\|_{L^3} \lesssim \mathcal{E}(w)^{\frac{1}{3}}$ which proves the bound for $p=3$, so it remains to prove the bound for $p\in (3,5]\subset [3,6]$. 
	We proceed using interpolation for $\frac1p=\frac13\frac{6-p}p+\frac16 (2-\frac6p)$ to bound
		\begin{align*}
		\|w\|_{L^p} \leq \|w\|_{L^3}^{\frac{6-p}p} \|w\|_{L^6}^{2-\frac6p} \lesssim 
		\mathcal{E}(w)^{\frac{6-p}{3p}} \mathcal{E}(w)^{(2-\frac6p)\frac5{12}} = \mathcal{E}(w)^{\frac{p-1}{2p}}.
	\end{align*}
	\end{enumerate}

	\item We first notice that by H\"older's inequality, with exponents $\frac{p-2}{p}+\frac{2}{p} = 1$,  translation invariance and Minkowski's inequality, we have that
	\begin{align}
		\int_{\TT^2} \left(\partial_1^h w^2(x)\right)^2 \, \mathrm{d} x 
		&\leq \left( \int_{\TT^2} |\partial_1^h w(x)|^{\frac{2p}{p-2}} \, \mathrm{d} x \right)^\frac{p-2}{p}
		\left( \int_{\TT^2} |w(x+he_1)+w(x)|^p \, \mathrm{d} x \right)^{\frac{2}{p}} \nonumber \\
		&\leq 4 \left( \int_{\TT^2} |\partial_1^h w(x)|^{\frac{2p}{p-2}} \, \mathrm{d} x \right)^\frac{p-2}{p}
		\left( \int_{\TT^2} |w(x)|^p \, \mathrm{d} x \right)^{\frac{2}{p}}. \label{eq:w^2_first} 
	\end{align}
	As $p\in [6,7)$, we have $\frac{2p}{p-2}\in (2,3]$, and by interpolation we obtain the bound
	\begin{align*}
		\int_{\TT^2} |\partial_1^h w(x)|^{\frac{2p}{p-2}} \, \mathrm{d} x 
		&\leq \left( \int_{\TT^2} |\partial_1^h w(x)|^2 \, \mathrm{d} x \right)^{\frac{p-6}{p-2}}
		\left( \int_{\TT^2} |\partial_1^h w(x)|^3 \, \mathrm{d} x \right)^{\frac{4}{p-2}}. 
	\end{align*}
	Using that $\int_{\TT^2} |\partial_1^h w(x)|^2 \, \mathrm{d} x \lesssim h^2 \mathcal{E}(w)$ and \eqref{eq:GJO_mod}, the last inequality implies that
	\begin{align}
		\int_{\TT^2} |\partial_1^h w(x)|^{\frac{2p}{p-2}} \, \mathrm{d} x 
		& \lesssim h^{2\left(\frac{p-6}{p-2}\right)+\frac{3}{2}\left(\frac{4}{p-2}\right)} \, \mathcal{E}(w). \label{eq:w^2_sec} 
	\end{align}
	Combining \eqref{eq:w^2_first}, \eqref{eq:w^2_sec} and \eqref{eq:L^p_contr} we get 
	\begin{align*}
		\int_{\TT^2} \left(\partial_1^h w^2(x)\right)^2 \, \mathrm{d} x  \lesssim_p \left( h^{2\left(\frac{p-6}{p-2}\right)+\frac{3}{2}\left(\frac{4}{p-2}\right)} \, \mathcal{E}(w) \right)^\frac{p-2}{p} \mathcal{E}(w)^\frac{p-1}{p} 
		= h^{\frac{2p-6}{p}} \mathcal{E}(w)^\frac{2p-3}{p}
	\end{align*}
	which implies \eqref{eq:besov_w^2_1/2_2_1_contr}.  \qedhere
\end{enumerate}
\end{proof}

As \eqref{eq:sobolev_to_besov} and \eqref{eq:besov_w^2_1/2_2_1_contr} imply that $\mathcal{E}(w)$ (to some power) controls the quantity $\||\partial_1|^{\frac{1}{2}} w^2\|_2$, it follows that the harmonic part $\mathcal{H}(w)$ given in \eqref{eq:harmonic-energy} of the energy $\mathcal{E}(w)$
%
is also controlled by $\mathcal{E}(w)$. Moreover, $\mathcal{E}(w)$ controls the $L^2$ norm of the $\frac{2}{3}$-fractional derivative in $x_2$ because the harmonic part $\mathcal{H}(w)$ does. We summarize this in the next proposition, where we also prove that $\mathcal{E}(w)$ is controlled by $\mathcal{H}(w)$.

\begin{proposition} \label{prop:bound-harmonic} \hfill
\begin{enumerate}[label=(\roman*)]
	\item For every $\kappa\in (0, \frac1{14})$, there exists a constant $C(\kappa)>0$ such that 
	\begin{align}\label{eq:bound-harmonic}
		\mathcal{H}(w) \leq C(\kappa)\bigg(1 + \mathcal{E}(w)^{\frac{3}{2}+\kappa}\bigg),
	\end{align}
	for every $w \in L^2(\TT^2)$ with vanishing average in $x_1$.
	In addition, there exists a constant $C>0$ such that
	\begin{align}\label{eq:bound-2/3}
		\int_{\TT^2} |\partial_1 w|^2\,\mathrm{d}x + \int_{\TT^2} | |\partial_2|^{\frac{2}{3}} w |^2 \mathrm{d}x \leq C \mathcal{H}(w),
	\end{align}
	for every $w \in L^2(\TT^2)$ with vanishing average in $x_1$. In particular, $[w]_{-\frac{1}{4}} \lesssim \mathcal{H}(w)^{\frac{1}{2}}$.
	\item There exists a constant $C>0$ such that 
	\begin{align}\label{eq:E-H-bound}
		\EE(w) \lesssim 1+ \HH(w)^2 \quad \text{for every } w\in\mathcal{W}.
	\end{align} 
\end{enumerate}
\end{proposition}

\begin{proof}
\begin{enumerate}[label=(\roman*)]
	\item 	Fix $\kappa\in (0, \frac1{14})$ and choose $p=p(\kappa)\in(6,7)$ such that $\frac{2p-3}{p} = \frac{3}{2}+\kappa$. Recalling that $\eta_w = \partial_2 w - \partial_1 \frac{1}{2} w^2$, by \eqref{eq:sobolev_to_besov} and the fact that 
	$\frac{2p-6}{2p} \in(\frac{1}{2},1)$ we have
	\begin{align*}
		\mathcal{H}(w) 
		&\lesssim \int_{\TT^2} |\partial_1 w|^2\,\mathrm{d}x + \int_{\TT^2} | |\partial_1|^{-\frac{1}{2}} \eta_w |^2 \mathrm{d}x + \int_{\TT^2} | |\partial_1|^{\frac{1}{2}} w^2 |^2 \mathrm{d}x 
		\lesssim_{\kappa} \mathcal{E}(w) + \|w^2\|_{\dotB^{\frac{2p-6}{2p}}_{2;1}}^2.
	\end{align*}
	By \eqref{eq:besov_w^2_1/2_2_1_contr} we know that $\|w^2\|_{\dotB^{\frac{2p-6}{2p}}_{2;1}}^2 
	\lesssim_{\kappa}\mathcal{E}(w)^{\frac{2p-3}{p}}$, thus \eqref{eq:bound-harmonic} follows by Young's inequality. 
	
	Inequality \eqref{eq:bound-2/3} is proved in a more general context in Lemma~\ref{lem:fractional-H}
	and the last statement follows from Lemma~\ref{lem:hoelder-embedding}.
	
	\item We have  
	\begin{align*}
		\mathcal{E}(w) &= \HH(w) - \int_{\TT^2} \left( |\partial_1|^{-\frac{1}{2}} \partial_2 w \, |\partial_1|^{-\frac{1}{2}} \partial_1 w^2 \right) \,\dd x  +  \frac{1}{4} \int_{\TT^2} \left( |\partial_1|^{-\frac{1}{2}} \partial_1 w^2 \right)^2\,\dd x \\
		&\leq \HH(w) + \HH(w)^{\frac{1}{2}} \, \left( \int_{\TT^2} \left( |\partial_1|^{\frac{1}{2}} w^2 \right)^2\,\dd x\right)^{\frac{1}{2}} + \frac{1}{4} \int_{\TT^2} \left( |\partial_1|^{\frac{1}{2}} w^2 \right)^2\,\dd x.
	\end{align*} 
	By Lemma \ref{lem:1/2-1-3/4+} the claimed inequality \eqref{eq:E-H-bound} follows. \hfill \qedhere
\end{enumerate}
\end{proof}


%% file: gamma-convergence.tex
\section{\texorpdfstring{$\Gamma$}{Gamma}-convergence of the renormalized energy} \label{sec:gamma-convergence}

In this section we study the $\Gamma$-convergence of the renormalized energy as the regularization of white noise is
removed, i.e., the limit $\ell\downarrow0$. As a consequence we will get the existence of minimizers of the limiting ``renormalized 
energy'', in particular, the existence of weak solutions of the Euler-Lagrange equation in \eqref{eq:ripple_rem}.


%
%

\subsection{Proof of Theorem~\ref{thm:minimizers}} 

We begin with the proof of coercivity statement \ref{it:coercivity} in Theorem~\ref{thm:minimizers}.

\begin{proof}[Proof of Theorem~\ref{thm:minimizers} \ref{it:coercivity}]
Let $(\xi, v, F) \in \triple$ be fixed. Since $v$ has vanishing average in $x_1$ we can 
estimate $\|v\|_{L^\infty} \lesssim [v]_{\frac{3}{4}-\varepsilon}$ (see e.g., \cite[Lemma 12]{IO19}), 
where the implicit constant is universal for small $\eps$ (e.g., $\eps\in (0, \frac1{100})$). We will use 
this estimate several times in what follows.
We split 
\begin{equs}
	 \mathcal{G}(v, F; w) 
	&  = \int_{\TT^2} \left( w^2 R_1 \partial_2 v 
	+ v^2 R_1 \eta_w 
	+ 2 v w R_1 \eta_w 
	+ 2 w  F
	- w  v R_1 \partial_1 v^2 
	+ (R_1 |\partial_1|^{\frac{1}{2}}(v w))^2 \right)\mathrm{d}x \\
	& =: \sum_{k=1}^6 \mathcal{G}_k(v, F; w), \label{eq:G_split}
\end{equs}
and bound each term separately: 
	\begin{enumerate}[label=\textbf{(T\arabic*)},leftmargin=0pt,labelsep=*,itemindent=*]
		\item \label{coercivityterm1}
		Notice that setting $g := |\partial_1|^{-1}\partial_2 v$ we have 
		\begin{align*}
			\partial_1 g = R_1 \partial_2 v
		\end{align*}
		and $g \in \mathcal{C}^{\frac{1}{4}-\varepsilon}$, with 
		$[g]_{\frac{1}{4}-\varepsilon} \lesssim [\xi]_{-\frac{5}{4}-\varepsilon}$, (see 
		Lemma~\ref{lem:regularised}). We can therefore integrate by parts 
		\begin{align*}
			\int_{\TT^2} w^2 R_1\partial_2 v \,\mathrm{d}x
			&= \int_{\TT^2} w^2 \partial_1 g \,\mathrm{d}x 
			= - 2\int_{\TT^2} w\partial_1 w \, g \,\mathrm{d}x
		\end{align*}
		and obtain the bound
		\begin{align*}
			\left| \mathcal{G}_1 (v, F; w) \right|
			&\leq \left| \int_{\TT^2} w^2 R_1\partial_2 v \,\mathrm{d}x \right| 
			\leq 2\|g\|_{L^{\infty}} \|w\|_{L^2} \|\partial_1 w\|_{L^2}
			\lesssim [g]_{\frac{1}{4}-\varepsilon} \|w\|_{L^3} \|\partial_1 w\|_{L^2} \\
			&\lesssim [g]_{\frac{1}{4}-\varepsilon} \mathcal{E}(w)^{\frac{1}{3} + 
			\frac{1}{2}} 
			\lesssim [\xi]_{-\frac{5}{4}-\varepsilon} \mathcal{E}(w)^{\frac{5}{6}},
		\end{align*}
		where we used Hölder's and Jensen's inequality, together with 
		\eqref{eq:L^p_contr}, 
		as well as $\|g\|_{L^{\infty}} \lesssim 
		[g]_{\frac{1}{4}-\varepsilon}$ because $g$ has zero average.
		By Young's inequality, it follows for any $\lambda\in (0,1)$
		\begin{align*}
			|\mathcal{G}_1(v, F; w)| \leq \lambda \mathcal{E}(w) + C_{\lambda}^{(1)} 
[\xi]_{-\frac{5}{4}-\varepsilon}^6.
		\end{align*}

		\item \label{coercivityterm2}
		For the term $\mathcal{G}_2$ we have that 
		\begin{align*} 
			\left|\mathcal{G}_2(v, F; w)\right| & = \left|\int_{\TT^2} (|\partial_1|^{\frac{1}{2}}v^2) \, (R_1|\partial_1|^{-\frac{1}{2}}\eta_w) \, \dd x \right|
			\leq \| |\partial_1|^{\frac{1}{2}}v^2 \|_{L^2} \| |\partial_1|^{-\frac{1}{2}}\eta_w \|_{L^2} \\
			&\lesssim \|v^2\|_{\dotB^{\frac{2}{3}}_{2;1}} 
			\mathcal{E}(w)^{\frac{1}{2}} 
			\lesssim [v^2]_{\frac{2}{3}}\mathcal{E}(w)^{\frac{1}{2}} 
			\lesssim [v]_{\frac{3}{4}-\varepsilon}^2 \mathcal{E}(w)^{\frac{1}{2}},
		\end{align*}
		where we used the Cauchy--Schwarz inequality, boundedness of $R_1$ on $L^2$, the estimates \eqref{eq:sobolev_to_besov}, 
		\eqref{eq:besov_to_holder_1} and \cite[Lemma 12]{IO19}. Hence, by Young's inequality, for any $\lambda\in (0,1)$,
		\begin{align*}
			|\mathcal{G}_2(v, F; w)| \leq \lambda \mathcal{E}(w) + C_{\lambda}^{(2)} 
			[v]_{\frac{3}{4}-\varepsilon}^4.
		\end{align*}
		\item \label{coercivityterm3}
		We estimate $\mathcal{G}_3$ using Cauchy--Schwarz, the boundedness of $R_1$ on $L^2$, and \eqref{eq:sobolev_to_besov} by
		\begin{align*}
			|\mathcal{G}_{3}(v, F; w)| 
			&= \left| 2\int_{\TT^2} ( |\partial_1|^{\frac{1}{2}} (v w)) (R_1 |\partial_1|^{-\frac{1}{2}} \eta_w)\,\mathrm{d}x \right| \\
			&\lesssim \||\partial_1|^{\frac{1}{2}}(vw)\|_{L^2} \||\partial_1|^{-\frac{1}{2}} \eta_w\|_{L^2} 
			\lesssim \|vw\|_{\dotB^{\frac23}_{2;1}} 
			\mathcal{E}(w)^{\frac{1}{2}}. 
		\end{align*}
		By the fractional Leibniz rule (Lemma~\ref{lem:leibniz} (i)) we can further bound
		\begin{align*}
			\|vw\|_{\dotB^{\frac23}_{2;1}}
			&\lesssim \|v\|_{L^\infty} \|w\|_{\dotB^{\frac23}_{2;1}} 
			+ [v]_{\frac23} \|w\|_{L^2}
			\lesssim [v]_{\frac{3}{4}-\varepsilon} 
			\left(\|w\|_{\dotB^{\frac23}_{2;1}} + \|w\|_{L^3} \right),
		\end{align*}
		where we also used Jensen's inequality. 
		Combined with \eqref{eq:besov_s_2_1_contr} and \eqref{eq:L^p_contr}, this gives
	\begin{align*}
		|\mathcal{G}_{3}(v, F; w)| & \lesssim [v]_{\frac{3}{4}-\varepsilon} 
		\left(\mathcal{E}(w)^{\frac{7}{18}} + \mathcal{E}(w)^{\frac{1}{3}}\right) 
		\mathcal{E}(w)^{\frac{1}{2}}
		\end{align*}
		so that Young's inequality yields for any $\lambda\in (0,1)$,
		\begin{align*}
			|\mathcal{G}_{3}(v, F; w)| \leq \lambda \mathcal{E}(w) + C_{\lambda}^{(3)}
			\left( [v]_{\frac{3}{4}-\varepsilon}^{6} + 
			[v]_{\frac{3}{4}-\varepsilon}^{9} \right).
		\end{align*}

		\item \label{coercivityterm4}
		By the duality Lemma~\ref{lem:dual_ineq}, $\mathcal{G}_{4}(v, F; w)=2\int_{\TT^2} wF\,\mathrm{d}x$ can be bounded 
				by 
		\begin{align*}
			|\mathcal{G}_{4}(v, F; w)| 
			&\lesssim \left(\|\partial_1 w\|_{L^2} 
			+ \| |\partial_2|^{\frac{2}{3}} w\|_{L^2} + 
			\|w\|_{L^1}\right) [F]_{-\frac{8}{9}} \\
			& \lesssim \left(\|\partial_1 w\|_{L^2} + 
			\| |\partial_2|^{\frac{2}{3}} w\|_{L^2} + \|w\|_{L^3}\right)
			[F]_{-\frac{3}{4}-\varepsilon}
		\end{align*}
		with a uniform implicit constant for every $\eps\in (0, \frac1{100})$,
		where in the second step we used Jensen's inequality and \cite[Remark 2]{IO19}. 
		With \eqref{eq:bound-harmonic}, \eqref{eq:bound-2/3} and \eqref{eq:L^p_contr}
		we obtain that
		\begin{align*}
			\|\partial_1 w\|_{L^2} + \| |\partial_2|^{\frac{2}{3}} w\|_{L^2} + \|w\|_{L^3} & 
			\lesssim 1+ \mathcal{E}(w)^{\frac{3}{4}+\kappa} + \mathcal{E}(w)^{\frac{1}{3}},
		\end{align*}
		where $\kappa>0$ can be chosen arbitrarily small (e.g., $\kappa=\frac1{100}$). This yields the estimate
		\begin{align*}
			|\mathcal{G}_{4}(v, F; w)| & \lesssim 
			\left(1 + \mathcal{E}(w)^{\frac{3}{4}+\kappa} + 
			\mathcal{E}(w)^{\frac{1}{3}}\right) [F]_{-\frac{3}{4}-\varepsilon}.
		\end{align*}
		It follows for any $\lambda\in (0,1)$,
		\begin{align*}
			|\mathcal{G}_{4}(v, F; w)| \leq \lambda \mathcal{E}(w) 
			+ C_{\lambda, \kappa}^{(4)} \left( [F]_{-\frac{3}{4}-\varepsilon} + 
			[F]_{-\frac{3}{4}-\varepsilon}^{\frac{3}{2}} + 
			[F]_{-\frac{3}{4}-\varepsilon}^{\frac{4}{1-4\kappa}} \right).
		\end{align*}
		\item \label{coercivityterm5}
		For $\mathcal{G}_5(v, F; w) = -\int_{\TT^2} w \, vR_1\partial_1 v^2 \, \dd x$, we use again the duality estimate
		Lemma~\ref{lem:dual_ineq},
		\begin{align*}
			|\mathcal{G}_{5}(v, F; w)| 
			&\lesssim \left(\|\partial_1 w\|_{L^2} 
			+ \| |\partial_2|^{\frac{2}{3}} w\|_{L^2} + \|w\|_{L^1}\right) 
			[vR_1\partial_1v^2]_{-\frac{2}{5}}.
		\end{align*}
		By \cite[Lemmata 6 and 12]{IO19} together with \eqref{eq:R_neg}, we have the uniform bound for any $\varepsilon\in (0, \frac1{100})$
		\begin{align*}
			[vR_1\partial_1v^2]_{-\frac{2}{5} } 
			\lesssim [v]_{\frac{1}{2}} [R_1 \partial_1 v^2]_{-\frac{2}{5}} 
			\lesssim [v]_{\frac{3}{4}-\varepsilon} [\partial_1 v^2]_{-\frac13}
			\lesssim [v]_{\frac{3}{4}-\varepsilon}^3.
		\end{align*}
		Hence, as in \ref{coercivityterm4}, we can bound $\mathcal{G}_{5}(v, F; w)$ for some small $\kappa>0$ (e.g., $\kappa=\frac1{100}$) by
		\begin{align*}
			|\mathcal{G}_{5}(v, F; w)| 
			&\lesssim \left(1+\mathcal{E}(w)^{\frac{3}{4}+\kappa} 
			+ \mathcal{E}(w)^{\frac{1}{3}}\right) [v]_{\frac{3}{4}-\varepsilon}^3.
		\end{align*}
		So, for any $\lambda\in (0,1)$, by Young's inequality,
		\begin{align*}
			|\mathcal{G}_{5}(v, F; w)| \leq \lambda \mathcal{E}(w) + C_{\lambda, \kappa}^{(5)}
			\left( [v]_{\frac{3}{4}-\varepsilon}^{\frac92} + [v]_{\frac{3}{4}-\varepsilon}^{3} 
			+ [v]_{\frac{3}{4}-\varepsilon}^{\frac{12}{1-4\kappa}} \right).
		\end{align*}
		
		\item \label{coercivityterm6}
		For the term $\mathcal{G}_6(v, F; w) = \int_{\TT^2} (R_1 |\partial_1|^{\frac{1}{2}}(v w))^2\,\mathrm{d}x$, we first notice 
		that by boundedness of $R_1$ on $L^2$ and the basic estimate \eqref{eq:sobolev_to_besov},
		\begin{align*}
			\mathcal{G}_6(v, F; w) = \||\partial_1|^{\frac{1}{2}}(v w)\|_{L^2}^2 \lesssim 
			\|vw\|_{\dotB^{\frac{2}{3}}_{2;1}}^2.
		\end{align*}
		Hence, by Lemma~\ref{lem:leibniz} and Jensen's inequality,
		\begin{align*}
			|\mathcal{G}_6(v, F; w)| 
			&\lesssim \|v\|_{L^\infty}^2 \|w\|_{\dotB^{\frac{2}{3}}_{2;1}}^2 
			+ [v]_{\frac{2}3}^2 \|w\|_{L^2}^2 
			\lesssim [v]_{\frac{3}{4}-\varepsilon}^2
			\left(\|w\|_{\dotB^{\frac{2}{3}}_{2;1}}^2
			+\|w\|_{L^3}^2 \right).
		\end{align*}
		Together with \eqref{eq:besov_w^2_1/2_2_1_contr} and \eqref{eq:L^p_contr} we can therefore estimate
		\begin{align*}
			|\mathcal{G}_6(v, F; w)| 
			&\lesssim [v]_{\frac{3}{4}-\varepsilon}^2 
			\left(\mathcal{E}(w)^{\frac79} 
			+ \mathcal{E}(w)^{\frac{2}{3}}\right).
		\end{align*}
		Young's inequality then yields for
		any $\lambda\in (0,1)$
		\begin{align*}
			|\mathcal{G}_6(v, F; w)| \leq \lambda \mathcal{E}(w) + C_{\lambda}^{(6)} 
			\left( [v]_{\frac{3}{4}-\varepsilon}^6 + 
			[v]_{\frac{3}{4}-\varepsilon}^{9} \right).
		\end{align*}
	\end{enumerate}
	\vspace{-0.5cm}
\end{proof}

In the proof of the continuity statement Theorem~\ref{thm:minimizers} \ref{it:continuity}, we need the 
following lemma.
 
\begin{lemma}\label{lem:strong-convergence}
	Let $\{w_{\ell}\}_{\ell\downarrow 0}\subset \mathcal{W}$ with uniformly bounded energy $\sup_{\ell}\mathcal{E}(w_{\ell})<\infty$, and assume that $w_{\ell} \to w$ strongly in $L^2$ as $\ell \to 0$ for some $w\in\mathcal{W}$. Then as $\ell \to 0$, 
	\begin{align*}
			\partial_1 w_{\ell} \rightharpoonup \partial_1 w,
			\quad 
			|\partial_1|^{-\frac{1}{2}} \partial_2 w_{\ell} \rightharpoonup |\partial_1|^{-\frac{1}{2}} \partial_2 w,
			\quad 
			|\partial_1|^{-\frac{1}{2}} \eta_{w_{\ell}} \rightharpoonup |\partial_1|^{-\frac{1}{2}} \eta_{w}, 
			\quad
			|\partial_2|^{\frac{2}{3}} w_{\ell} \rightharpoonup |\partial_2|^{\frac{2}{3}} w,  	
	\end{align*}
	weakly in $L^2$, and for any $s_1\in (0,1)$ and $s_2 \in (0, \frac{2}{3})$,
	\begin{align*}
		|\partial_1|^{s_1} w_{\ell} \to |\partial_1|^{s_1} w, \qquad |\partial_2|^{s_2} w_{\ell} \to |\partial_2|^{s_2} w 
		\qquad \text{strongly in } L^2.
	\end{align*}
\end{lemma}

\begin{proof}
	For the first part, we use the fact that a uniformly bounded sequence in $L^2$ converging in the distributional sense converges weakly in $L^2$.
	By Proposition~\ref{prop:bound-harmonic}, we have that $\sup_\ell \mathcal{H}(w_{\ell})<\infty$ . Therefore, $\{\partial_1 w_{\ell}\}_\ell$, $\{|\partial_1|^{-\frac{1}{2}}\partial_2 w_{\ell}\}_{\ell}$, 
	$\{|\partial_1|^{-\frac{1}{2}} \eta_{w_{\ell}}\}_\ell$, and $\{|\partial_2|^{\frac{2}{3}}w_{\ell}\}_\ell$ are uniformly bounded in $L^2$. They
	also converge in the distributional sense since $w_{\ell} \to w$ strongly in $L^2$ 
	(in particular, $(w_{\ell})^2 \to w^2$ strongly in $L^1$, so $\eta_{w_{\ell}}\to \eta_w$ in the distributional sense).
	For the second part, by Lemma~\ref{lem:frac-sobolev}, for any $s_1\in (0,1)$, $s_2 \in (0, \frac{2}{3})$, we have that 
	$\{|\partial_1|^{s_1} w_{\ell}\}_\ell$ is uniformly bounded in the homogeneous Sobolev space $\dot{H}^{\frac{2}{3}(1-s_1)}$, 
	and $\{|\partial_2|^{s_2} w_{\ell}\}_\ell$ is uniformly bounded in $\dot{H}^{\frac{2}{3}-s_2}$. The compact embedding 
	$\dot{H}^{\min\{\frac{2}{3}(1-s_1), \frac{2}{3}-s_2\}}\hookrightarrow L^2$ (of periodic functions with vanishing average) yields the conclusion.
\end{proof}

\begin{proof}[Proof of Theorem~\ref{thm:minimizers} \ref{it:continuity}]
First, all the convergence statements from Lemma~\ref{lem:strong-convergence} hold for the sequence $\{w_{\ell}\}_\ell$. As
$(\xi_{\ell}, v_{\ell}, F_{\ell}) \to (\xi, v, F)$ in $\triple$, by
Lemma~\ref{lem:regularised} we also have that $g_\ell:=|\partial_1|^{-1} \partial_2 v_{\ell} \to 
|\partial_1|^{-1} \partial_2 v=:g$ in $\mathcal{C}^{\frac{1}{4}-\varepsilon}$.
We will prove the continuity of $\GG$ using the decomposition $\mathcal{G} = \sum_{k=1}^6 \mathcal{G}_k$ in \eqref{eq:G_split} 
and study each term $\mathcal{G}_k$ separately.
\begin{enumerate}[label=\textbf{(T$'$\arabic*)},leftmargin=0pt,labelsep=*,itemindent=*]
	\item \label{convergenceterm1}
	For the term $\mathcal{G}_1(v_{\ell}, F_{\ell}; w_{\ell}) = \int_{\TT^2} (w_{\ell})^2 R_1 \partial_2 v_{\ell}\,\mathrm{d}x$, 
	we use integration by parts, and that $w_{\ell}\to w$ strongly in $L^2$, $\partial_1 w_{\ell} \rightharpoonup \partial_1 w$ 
	weakly in $L^2$, and $g_\ell\to g$ uniformly on $\TT^2$,
		\begin{align*}
		\mathcal{G}_1(v_{\ell}, F_{\ell}; w_{\ell})
		&= \int_{\TT^2} (w_{\ell})^2 \partial_1 g_{\ell} \,\mathrm{d}x = - 2\int_{\TT^2} w_{\ell}\partial_1 w_{\ell} \, g_{\ell} \,\mathrm{d}x \\
		&\to -2\int_{\TT^2} w \partial_1 w \, g\,\mathrm{d}x 
		= \int_{\TT^2} w^2 \partial_1 g \,\mathrm{d}x 
		= \int_{\TT^2} w^2 R_1 \partial_2 v \,\mathrm{d}x 
		= \mathcal{G}_1(v, F; w).
	\end{align*}
	
	\medskip
	\item \label{convergenceterm2}
	For the term $\mathcal{G}_2(v_{\ell}, F_{\ell}; w_{\ell}) = \int_{\TT^2} v_{\ell}^2 R_1 \eta_{w_{\ell}}\,\mathrm{d}x$ we use that 
	$|\partial_1|^{-\frac{1}{2}} \eta_{w_{\ell}} \rightharpoonup |\partial_1|^{-\frac{1}{2}} \eta_{w}$ weakly in $L^2$ 
	(hence also $R_1|\partial_1|^{-\frac{1}{2}} \eta_{w_{\ell}} \rightharpoonup R_1|\partial_1|^{-\frac{1}{2}} \eta_{w}$ weakly in $L^2$), and 
	$|\partial_1|^{\frac{1}{2}} v_{\ell}^2 \to |\partial_1|^{\frac{1}{2}} v^2$ strongly
	in $\mathcal{C}^{\frac{1}{4}-\varepsilon}$ (see Lemma~\ref{lem:frac-deriv-hoelder}),
		\begin{align*}
			\mathcal{G}_2(v_{\ell}, F_{\ell}; w_{\ell})
			&= \int_{\TT^2} (|\partial_1|^{\frac{1}{2}}v_{\ell}^2) \, (R_1|\partial_1|^{-\frac{1}{2}}\eta_{w_{\ell}}) \, \dd x 
			\to \int_{\TT^2} (|\partial_1|^{\frac{1}{2}}v^2) \, (R_1|\partial_1|^{-\frac{1}{2}}\eta_{w}) \, \dd x
			= \mathcal{G}_2(v, F; w).
		\end{align*}
	
	\medskip
	\item \label{convergenceterm3}
	Since $\mathcal{G}_3(v_{\ell}, F_{\ell}; w_{\ell}) = 2 \int_{\TT^2} v_{\ell}w_{\ell} R_1 \eta_{w_{\ell}}\,\mathrm{d}x =
	2 \int_{\TT^2} |\partial_1|^{\frac{1}{2}}(v_{\ell}w_{\ell}) \, R_1|\partial_1|^{-\frac{1}{2}} \eta_{w_{\ell}} \,\mathrm{d}x$
	and, as in \ref{convergenceterm2}, $R_1|\partial_1|^{-\frac{1}{2}} \eta_{w_{\ell}} \rightharpoonup R_1|\partial_1|^{-\frac{1}{2}} \eta_{w}$
	weakly in $L^2$, the claimed convergence follows if we show that $|\partial_1|^{\frac{1}{2}}(v_{\ell} w_{\ell}) \to 
	|\partial_1|^{\frac{1}{2}} (v w)$ strongly in $L^2$. For this, we use the triangle inequality, Lemma~\ref{lem:leibniz} \ref{item:leibniz-2},
	and that $w_{\ell} \to w$, $|\partial_1|^{\frac{1}{2}} w_{\ell} \to |\partial_1|^{\frac{1}{2}} w$ strongly in $L^2$ as well as $v_{\ell} \to v$
	in $\mathcal{C}^{\frac{3}{4}-\varepsilon}\subset \mathcal{C}^\frac{2}{3}$ which yield as $\ell \to 0$,
	\begin{equs}
		\| |\partial_1|^{\frac{1}{2}}(v_{\ell} w_{\ell}) - |\partial_1|^{\frac{1}{2}} (vw) \|_{L^2} 
		& \leq \| |\partial_1|^{\frac{1}{2}}((v_{\ell}-v) w_{\ell}) \|_{L^2} + \| |\partial_1|^{\frac{1}{2}} (v(w_{\ell}-w)) \|_{L^2}
		\\
		& \lesssim \|v_{\ell} - v\|_{L^{\infty}} \||\partial_1|^{\frac{1}{2}} w_{\ell} \|_{L^2} + [v_{\ell} - v]_{\frac{2}3} \|w_{\ell}\|_{L^2} 
		\\
		& \quad +\|v\|_{L^{\infty}} \| |\partial_1|^{\frac{1}{2}} (w_{\ell}-w) \|_{L^2} + [v]_{\frac{2}{3}} \|w_{\ell}-w\|_{L^2}
		\\
		& \lesssim [v_{\ell}-v]_{\frac{3}{4}-\varepsilon} \left( \||\partial_1|^{\frac{1}{2}} w_{\ell} \|_{L^2} + \|w_{\ell}\|_{L^2} \right)
		\\
		& \quad + [v]_{\frac{3}{4}-\varepsilon} \left( \| |\partial_1|^{\frac{1}{2}} (w_{\ell}-w) \|_{L^2} + \|w_{\ell}-w\|_{L^2}  \right)\to 0.
	\end{equs}
	
	\medskip
	\item \label{convergenceterm4}
	The term $\mathcal{G}_4(v_{\ell}, F_{\ell}; w_{\ell}) = 2 \int_{\TT^2} w_{\ell} F_{\ell} \,\mathrm{d}x$ is treated by duality. Since 
	$F_{\ell}\to F$ in $\mathcal{C}^{-\frac{3}{4}-\varepsilon}\subset \mathcal{C}^{-\frac45}$ (see e.g., \cite[Remark 2]{IO19}), by Lemmata
	\ref{lem:strong-convergence} and \ref{lem:dual_ineq} we have for $\ell \to 0$,
	\begin{align*}
		& \left|\mathcal{G}_4(v_{\ell}, F_{\ell}; w_{\ell}) - \mathcal{G}_4(v,F;w)\right|
		\\
		& \quad \lesssim \left| \int_{\TT^2} (w_{\ell}-w) F_{\ell} \,\mathrm{d}x \right| + \left| \int_{\TT^2} w (F_{\ell}-F)\,\mathrm{d}x \right|
		\\
		& \quad \lesssim [F_{\ell}]_{-\frac{4}{5}} 
		\left( \||\partial_1|^{\frac{5}{6}} 
		(w_{\ell}-w)\|_{L^2} + \||\partial_2|^{\frac{2}{3}\cdot \frac56} 
		(w_{\ell}-w)\|_{L^2} + \|w_{\ell}-w\|_{L^2} \right)
		\\
		& \quad \quad + [F_{\ell}-F]_{-\frac{4}{5}}
		\left(\||\partial_1|^{\frac{5}6} w\|_{L^2}
		+ \||\partial_2|^{\frac{2}{3}\cdot \frac56}w\|_{L^2} + \|w\|_{L^2} \right) \to 0.
	\end{align*}
	
	\medskip
	\item \label{convergenceterm5}
	For the continuity of the term $\mathcal{G}_5(v_{\ell}, F_{\ell}; w_{\ell}) = -\int_{\TT^2} w_{\ell} \, 
	v_{\ell}R_1\partial_1 v_{\ell}^2 \, \dd x$ we again use the duality Lemma~\ref{lem:dual_ineq}. Here, the 
	situation is even easier than in \ref{convergenceterm4}, as $v_{\ell} R_1 \partial_1 v_{\ell}^2$ converges
	to the nonsingular product $v R_1 \partial_1 v^2$ in $\mathcal{C}^{-\frac{1}{4}-2\varepsilon}$.
	This convergence follows by 
	\begin{align*}
		& [v_{\ell} R_1 \partial_1 v_{\ell}^2 - v R_1\partial_1 v^2]_{-\frac{1}{4}-{2\varepsilon}}
		\\
		& \quad =[(v_{\ell}-v) R_1 \partial_1 v^2 + v_{\ell} R_1 \partial_1 ((v_{\ell}-v)(v_{\ell}+v))]_{-\frac{1}{4}-{2\varepsilon}}
		\\
		& \quad \lesssim [v_{\ell}-v]_{\frac{3}{4}-\varepsilon} [R_1 \partial_1 
		v^2]_{-\frac{1}{4}-{2\varepsilon}} + [v_{\ell}]_{\frac{3}{4}-\varepsilon} [R_1 
		\partial_1 ((v_{\ell}-v)(v_{\ell}+ v))]_{-\frac{1}{4}-{2\varepsilon}} 
		\\
		& \quad \lesssim [v_{\ell}-v]_{\frac{3}{4}-\varepsilon} [v^2]_{\frac{3}{4}-\varepsilon} + 
		[v_{\ell}]_{\frac{3}{4}-\varepsilon} 
		[(v_{\ell}-v)(v_{\ell}+v)]_{\frac{3}{4}-\varepsilon} 
		\\
		& \quad \lesssim [v_{\ell}-v]_{\frac{3}{4}-\varepsilon} \left( 
		[v]_{\frac{3}{4}-\varepsilon}^2 
		+ [v_{\ell}]_{\frac{3}{4}-\varepsilon}^2 \right) \to 0,
	\end{align*}
	where we used that $v_{\ell} \to v$ in $\mathcal{C}^{\frac{3}{4}-\varepsilon}$ and \cite[Lemmata 6, 7, and 12]{IO19}. We conclude as for
	$\mathcal{G}_4$ (with $v_{\ell} R_1 \partial_1 v_{\ell}^2$ corresponding to $F_{\ell}$ and $v R_1 \partial_1 v^2$ to $F$, using also that
	$\mathcal{C}^{-\frac{1}{4}-{2\varepsilon}}\subset \mathcal{C}^{-\frac{3}{4}-\varepsilon}$).
	
	\medskip
	\item \label{convergenceterm6} Noting that $\mathcal{G}_6(v_{\ell}, F_{\ell}; w_{\ell}) = \int_{\TT^2} (R_1 |\partial_1|^{\frac{1}{2}}(v_{\ell}w_{\ell}))^2\,\mathrm{d}x 
	= \| |\partial_1|^{\frac{1}{2}}(v_{\ell}w_{\ell}) \|_{L^2}^2$, continuity follows since $|\partial_1|^{\frac{1}{2}}(v_{\ell} w_{\ell}) \to |\partial_1|^{\frac{1}{2}} (v w)$ 
	in $L^2$ used in \ref{convergenceterm3}. \qedhere
\end{enumerate}
\end{proof}

We now prove the compactness Theorem~\ref{thm:minimizers} \ref{it:compactness} of the sublevel sets of $E_{ren}$ with respect to the strong 
topology in $L^2$. 

\begin{proof}[Proof of Theorem~\ref{thm:minimizers} \ref{it:compactness}]
	By the coercivity Theorem~\ref{thm:minimizers} \ref{it:coercivity} for $\lambda=\frac{1}{2}$, it follows that 
	\begin{align}
	\label{coerciv}
		E_{ren}(v,F; w) = \mathcal{E}(w) + \mathcal{G}(v,F; w) \geq \frac12
		\mathcal{E}(w) - C\geq -C.
	\end{align}
Thus $E_{ren}(v,F; \cdot)$ is bounded from below and the sublevel set $\{E_{ren}(v,F;\cdot)\leq M\}$ over $\mathcal W$ is included in a sublevel 
set of  $\mathcal E$ over $\mathcal W$ which is relatively compact in $L^2$ by Lemma~\ref{lem:frac-sobolev}. It remains to prove that the sublevel 
set $\{E_{ren}(v,F;\cdot)\leq M\}$ over $\mathcal W$ is closed in $L^2$. By the continuity of $\mathcal{G}(v,F;\cdot)$ (Theorem~\ref{thm:minimizers}
\ref{it:continuity}), it suffices to show that $\mathcal E$ is lower semicontinuous in $\mathcal W$, i.e., for every $\{w^\ell\}_{\ell \downarrow 0} 
\subset \mathcal W$ with $w^\ell\to w$ in $L^2$, there holds 
\begin{align}\label{eq:lsc}
		\liminf_{\ell \downarrow 0} \mathcal{E}(w^\ell) \geq \mathcal{E}(w).
	\end{align}
	Indeed, since $a^2 \geq b^2 + 2 (a-b) b$, it follows that 
	\begin{align*}
		\mathcal{E}(w^\ell) \geq \mathcal{E}(w) + 2 \int_{\TT^2} \left(\partial_1 w^\ell - \partial_1 w\right) \partial_1w \, \mathrm{d} x
		+ 2 \int_{\TT^2} \left(|\partial_1|^{-\frac12} \eta_{w^\ell} - |\partial_1|^{-\frac{1}{2}} \eta_w\right) 
		|\partial_1|^{-\frac{1}{2}} \eta_w \, \mathrm{d} x.
	\end{align*}
	Without loss of generality, we may assume that $\liminf_{\ell \downarrow 0} \mathcal{E}(w^\ell)=\limsup_{\ell \downarrow 0} \mathcal{E}(w^\ell)<\infty$.
	Hence, by Lemma~\ref{lem:strong-convergence}, $\partial_1 w^\ell \rightharpoonup \partial_1 w$ and $|\partial_1|^{-\frac{1}{2}} \eta_{w^\ell} \rightharpoonup |\partial_1|^{-\frac{1}{2}} \eta_w$
	weakly in $L^2$, and thus, \eqref{eq:lsc} follows.
	The same argument shows that $\HH$ is lower semicontinuous in $\mathcal{W}$.
\end{proof}

We are now ready to prove the existence of minimizers Theorem~\ref{thm:minimizers} \ref{it:existence_min}. 

\begin{proof}[Proof of Theorem~\ref{thm:minimizers} \ref{it:existence_min}]
Note that $E_{ren}(v,F; 0)=0$ and recall that $E_{ren}(v, F; \cdot)$ is bounded from below (see \eqref{coerciv}) and lower semicontinuous in $L^2$ over 
its zero sublevel set (due to \eqref{eq:renormalized-energy}, $\mathcal{G}(v, F; \cdot)$ being continuous over any sublevel set of $\mathcal E$ and 
$\mathcal{E}$ being lower semicontinuous in $L^2$). By the $L^2$-compactness of the zero sublevel set of $E_{ren}(v, F; \cdot)$ over $\mathcal W$ 
(Theorem~\ref{thm:minimizers} \ref{it:compactness}), the direct method in the calculus of variations yields the existence of minimizers. 

In order to show that a minimizer is a distributional solution of the Euler--Lagrange equation \eqref{eq:ripple_rem}, we do the following splitting: 
\begin{align*}
	E_{ren}(v,F;w) = \HH(w) + \left( \EE(w) -\HH(w) \right) + \mathcal{G}(v,F;w) 
	= \HH(w) + \sum_{j=1}^4 L_j(w),
\end{align*}
where 
\begin{align*}
	L_1(w) &= \int_{\TT^2} \left(2 w  F - w  v R_1 \partial_1 v^2 + v^2 R_1 \partial_2 w \right)\,\dd x \\
	L_2(w) &= \int_{\TT^2} \left( w^2 R_1 \partial_2 v - \frac{1}{2} v^2 R_1 \partial_1 w^2 + 2 v w R_1 \partial_2 w + (R_1 |\partial_1|^{\frac{1}{2}}(v w))^2 \right) \,\dd x \\
	L_3(w) &= \int_{\TT^2} \left(- R_1|\partial_1|^{\frac{1}{2}} w^2 |\partial_1|^{-\frac{1}{2}} \partial_2 w - v w R_1 \partial_1 w^2 \right) \,\dd x \\
	L_4(w) &= \int_{\TT^2} \frac{1}{4}\left(  |\partial_1|^{\frac{1}{2}} w^2 \right)^2.
\end{align*}
We will show that, given $(\xi, v, F)\in\triple$, the functional $E_{ren}(v,F;\cdot)$ is $\C^{\infty}$ on the space $\mathcal{W}$ endowed with the norm $\HH^{\frac{1}{2}}$, denoted by $(\mathcal{W}, \HH^{\frac{1}{2}})$. 

\begin{enumerate}[label=\textsc{Step \arabic*},leftmargin=0pt,labelsep=*,itemindent=*]
	\item (Estimating the linear functional $L_1$). We claim that $L_1$ is a continuous linear functional on $(\mathcal{W}, \HH^{\frac{1}{2}})$, i.e.,
	\begin{align}
		|L_1(w)| \leq C \HH(w)^{\frac{1}{2}},
	\end{align} 
	where $C$ depends polynomially on $[v]_{\frac{3}{4}-\epsilon}, [F]_{-\frac{3}{4}-\epsilon}$. 
	Indeed, as in (T4) in the proof of \ref{it:coercivity}, by the duality Lemma \ref{lem:dual_ineq} and Poincaré's inequality, we may bound
	\begin{align*}
		\left|\int_{\TT^2} 2 w  F \,\dd x \right| \lesssim \left( \|\partial_1 w\|_{L^2} + \||\partial_2|^{\frac{2}{3}} w\|_{L^2} + \|w\|_{L^2} \right) [F]_{-\frac{8}{9}} 
		\lesssim [F]_{-\frac{3}{4}-\epsilon} \HH(w)^{\frac{1}{2}}.
	\end{align*}
	By the same argument, see also (T5) above, we have 
	\begin{align*}
		\left| \int_{\TT^2} w  v R_1 \partial_1 v^2 \,\dd x \right| \lesssim \left( \|\partial_1 w\|_{L^2} + \||\partial_2|^{\frac{2}{3}} w\|_{L^2} + \|w\|_{L^2} \right) [v R_1\partial_1 v^2]_{-\frac{2}{5}} \lesssim [v]_{\frac{3}{4}-\epsilon}^3 \HH(w)^{\frac{1}{2}}.
	\end{align*}
	As in (T2), the last term of $L_1$ is estimated using Cauchy--Schwarz, \eqref{eq:sobolev_to_besov}, and \eqref{eq:besov_to_holder_1}, by 
	\begin{align*}
		\left| \int_{\TT^2}v^2 R_1 \partial_2 w \,\dd x \right| \leq \||\partial_1|^{\frac{1}{2}} v^2 \|_{L^2} \||\partial_1|^{-\frac{1}{2}} \partial_2 w \|_{L^2} 
		\lesssim [v]_{\frac{3}{4}-\epsilon}^2 \HH(w)^{\frac{1}{2}}.
	\end{align*}
	
	\item (Estimating the quadratic functional $L_2$). We claim that $L_2$ is a continuous quadratic functional on $(\mathcal{W}, \HH^{\frac{1}{2}})$, i.e., there exists a continuous bilinear functional $M_2$ given by
	\begin{align*}
		&M_2(w_1, w_2) \\ &\quad= \int_{\TT^2} \left( w_1 w_2 R_1 \partial_2 v - \frac{1}{2} v^2 R_1 \partial_1 (w_1 w_2) + 2 v w_1 R_1 \partial_2 w_2 + (|\partial_1|^{\frac{1}{2}}(v w_1))(|\partial_1|^{\frac{1}{2}}(v w_2)) \right) \,\dd x
	\end{align*}
	such that $L_2(w) = M_2(w,w)$, and satisfying the inequality 
	\begin{align}\label{eq:bilinear}
		|M_2(w_1, w_2)| \leq C \HH(w_1)^{\frac{1}{2}} \HH(w_2)^{\frac{1}{2}},
	\end{align}
	where $C$ depends polynomially on $[\xi]_{-\frac{5}{4}-\epsilon}, [v]_{\frac{3}{4}-\epsilon}, [F]_{-\frac{3}{4}-\epsilon}$.
	To prove \eqref{eq:bilinear}, we again treat each term separately. Similarly to (T1), let $g:= |\partial_1|^{-1}\partial_2 v$, such that $R_1 \partial_2 v = \partial_1 g$. Recall that by Lemma \ref{lem:regularised}, $[g]_{\frac{1}{4}-\epsilon} \lesssim [\xi]_{-\frac{5}{4}-\epsilon}$. Then integration by parts and Cauchy--Schwarz, together with Poincaré's inequality, gives 
	\begin{align*}
		\left| \int_{\TT^2} w_1 w_2 R_1\partial_2 v \,\dd x \right| 
		&= \left| \int_{\TT^2} \partial_1(w_1 w_2) g \,\dd x \right|
		\lesssim [g]_{\frac{1}{4}-\epsilon} \int_{\TT^2} |w_1\partial_1 w_2 + w_2 \partial_1 w_1|\,\dd x \\
		&\lesssim [\xi]_{-\frac{5}{4}-\epsilon} \HH(w_1)^{\frac{1}{2}} \HH(w_2)^{\frac{1}{2}}.
	\end{align*}
	Similarly, the second term can be estimated by 
	\begin{align*}
		\left| \int_{\TT^2} \frac{1}{2} v^2 R_1 \partial_1 (w_1 w_2) \,\dd x \right| 
		&= \left| \int_{\TT^2} \frac{1}{2} R_1 v^2 (w_1 \partial_1  w_2 + w_2 \partial_1 w_1) \,\dd x \right|
		\lesssim [R_1 v^2]_{\frac{3}{4}-2\epsilon} \HH(w_1)^{\frac{1}{2}} \HH(w_2)^{\frac{1}{2}} \\
		&\lesssim [v]_{\frac{3}{4}-\epsilon}^2 \HH(w_1)^{\frac{1}{2}} \HH(w_2)^{\frac{1}{2}}.
	\end{align*}
	By Cauchy--Schwarz, Lemma \ref{lem:leibniz} (i), interpolation and Poincaré's inequality, we can bound the third term by
	\begin{align*}
		&\left| \int_{\TT^2} 2 v w_1 R_1 \partial_2 w_2 \,\dd x \right| 
		= 2 \left| \int_{\TT^2} |\partial_1|^{\frac{1}{2}}(v w_1) R_1 |\partial_1|^{-\frac{1}{2}}(\partial_2 w_2) \,\dd x \right| 
		\leq 2 \| |\partial_1|^{\frac{1}{2}}(v w_1) \|_{L^2} \||\partial_1|^{-\frac{1}{2}} \partial_2 w\|_{L^2} \\
		&\qquad\qquad\qquad\qquad\lesssim \left( \||\partial_1|^{\frac{1}{2}} w_1 \|_{L^2} \|v\|_{L^{\infty}} + \|w_1\|_{L^2} [v]_{\frac{1}{2}+\epsilon} \right) \HH(w_2)^{\frac{1}{2}} 
		\lesssim [v]_{\frac{3}{4}-\epsilon} \HH(w_1)^{\frac{1}{2}} \HH(w_2)^{\frac{1}{2}}.
	\end{align*}
	Analogously, the fourth term is estimated by
	\begin{align*}
		\left| \int_{\TT^2} (|\partial_1|^{\frac{1}{2}}(v w_1))(|\partial_1|^{\frac{1}{2}}(v w_2)) \,\dd x \right| 
		&\leq \||\partial_1|^{\frac{1}{2}}(v w_1)\|_{L^2} \||\partial_1|^{\frac{1}{2}}(v w_2)\|_{L^2} 
		\lesssim [v]_{\frac{3}{4}-\epsilon}^2 \HH(w_1)^{\frac{1}{2}} \HH(w_2)^{\frac{1}{2}}.
	\end{align*}
	
	\item (Estimating the cubic functional $L_3$). We claim that $L_3$ is a continuous cubic functional on $(\mathcal{W}, \HH^{\frac{1}{2}})$, i.e., there exists a continuous three-linear functional $M_3$ given by 
	\begin{align*}
		M_3(w_1, w_2,w_3) = - \int_{\TT^2} \left( R_1|\partial_1|^{\frac{1}{2}} (w_1 w_2) |\partial_1|^{-\frac{1}{2}} \partial_2 w_3 + v w_1 R_1 \partial_1 (w_2 w_3) \right) \,\dd x,
	\end{align*}
	such that $L_3(w) = M_3(w,w,w)$, and $M_3$ is controlled by 
	\begin{align}\label{eq:trilinear}
		|M_3(w_1, w_2, w_3)| \leq C \HH(w_1)^{\frac{1}{2}} \HH(w_2)^{\frac{1}{2}} \HH(w_3)^{\frac{1}{2}} ,
	\end{align}
	where $C$ depends polynomially on $[v]_{\frac{3}{4}-\epsilon}, [F]_{-\frac{3}{4}-\epsilon}$.
	Indeed, the first term is estimated using Cauchy--Schwarz and Lemma \ref{lem:1/2-1-3/4+},
	\begin{align*}
		\left| \int_{\TT^2} R_1|\partial_1|^{\frac{1}{2}} (w_1 w_2) |\partial_1|^{-\frac{1}{2}} \partial_2 w_3 \,\dd x \right| 
		&\leq \| |\partial_1|^{\frac{1}{2}} (w_1 w_2) \|_{L^2} \| |\partial_1|^{-\frac{1}{2}} \partial_2 w_3\|_{L^2} \\
		&\lesssim \HH(w_1)^{\frac{1}{2}} \HH(w_2)^{\frac{1}{2}} \HH(w_3)^{\frac{1}{2}}.
	\end{align*}
	Similarly, Lemma \ref{lem:leibniz} (i) and Lemma \ref{lem:1/2-1-3/4+} imply that 
	\begin{align*}
		\left| \int_{\TT^2} v w_1 R_1 \partial_1 (w_2 w_3)\,\dd x \right| 
		\leq \| |\partial_1|^{\frac{1}{2}} (v w_1) \|_{L^2} \| |\partial_1|^{\frac{1}{2}} (w_2 w_3) \|_{L^2} 
		\lesssim [v]_{\frac{3}{4}-\epsilon} \HH(w_1)^{\frac{1}{2}} \HH(w_2)^{\frac{1}{2}} \HH(w_3)^{\frac{1}{2}}.
	\end{align*}
	
	\item (Estimating the quartic functional $L_4$). We claim that $L_4$ is a continuous quartic functional on $(\mathcal{W}, \HH^{\frac{1}{2}})$, i.e., there exists a continuous four-linear functional $M_4$ given by 
	\begin{align*}
		M_4(w_1, w_2,w_3, w_4) = \int_{\TT^2} \frac{1}{4}\left(  |\partial_1|^{\frac{1}{2}}(w_1 w_2) |\partial_1|^{\frac{1}{2}}(w_3 w_4) \right) \,\dd x,
	\end{align*}
	such that $L_4(w) = M_4(w, w, w, w)$. Indeed, Cauchy--Schwarz and Lemma \ref{lem:1/2-1-3/4+} implies that
	\begin{align}\label{eq:fourlinear}
		|M_4(w_1, w_2, w_3, w_4)| \leq C \HH(w_1)^{\frac{1}{2}} \HH(w_2)^{\frac{1}{2}} \HH(w_3)^{\frac{1}{2}} \HH(w_4)^{\frac{1}{2}},
	\end{align}
	where $C$ depends polynomially on $[v]_{\frac{3}{4}-\epsilon}, [F]_{-\frac{3}{4}-\epsilon}$. 
\end{enumerate}

Therefore, the gradient $\nabla_w E_{ren}(v,F;w)$ belongs to the dual space of $(\mathcal{W}, \HH^{\frac{1}{2}})$, in particular, it is a distribution, so that 
\begin{align*}
	\nabla_w E_{ren}(v,F;w) = 0
\end{align*}
is the Euler--Lagrange equation \eqref{eq:ripple_rem}.
\end{proof}

\subsection{\texorpdfstring{$\Gamma$}{Gamma}-convergence}

In view of Theorem~\ref{thm:minimizers}, we give the proof of $\Gamma$-convergence of the renormalized energy for sequences 
$(\xi_{\ell}, v_{\ell}, F_{\ell}) \to (\xi, v, F)$ in $\triple$ as $\ell\to0$. 

\begin{proof}[Proof of Corollary \ref{cor:gamma-convergence}]
	Assume that $(\xi_{\ell}, v_{\ell}, F_{\ell}) \to (\xi, v, F)$ in $\triple$ as $\ell\to 0$. By the decomposition of $E_{ren}$ in
	\eqref{eq:renormalized-energy} and the continuity of $\mathcal{G}$ in Theorem~\ref{thm:minimizers} \ref{it:continuity}, the pointwise
	convergence $E_{ren}(v_{\ell}, F_{\ell}; \cdot)\to E_{ren}(v, F; \cdot)$ over $\mathcal W$ is immediate. We proceed with the proof of the
	remaining statements. 
	\begin{enumerate}[label=(\roman*)]
	\item ($\Gamma-\liminf$): Without loss of generality, we may assume that 
	\begin{equs}
	 \textstyle\liminf_{\ell\to 0} E_{ren}(v_{\ell}, F_{\ell}; w^\ell)=\limsup_{\ell\to 0} E_{ren}(v_{\ell}, F_{\ell}; w^\ell)<\infty.
	\end{equs} 
	As $\{(\xi_{\ell}, v_{\ell}, F_{\ell})\}_\ell$ is uniformly bounded in $\triple$, the coercivity Theorem~\ref{thm:minimizers} \ref{it:coercivity}
	implies via \eqref{coerciv} the existence of a constant $C>0$ (uniform in $\ell$) such that $E_{ren}(v_{\ell}, F_{\ell}; w^\ell)\geq \frac12 \mathcal{E}(w^\ell)-C$,
	i.e., $\limsup_{\ell\to 0} \mathcal{E}(w^\ell)<\infty$. The desired inequality is a consequence of \eqref{eq:renormalized-energy} combined with the continuity of
	$\mathcal{G}$ (Theorem~\ref{thm:minimizers} \ref{it:continuity}) and the lower semicontinuity of $\mathcal E$ over $\mathcal W$ in \eqref{eq:lsc}. 
	
	\item ($\Gamma-\limsup$): For $w\in \mathcal W$, one sets $w_{\ell} = w$ for all $\ell\in(0,1]$ and the conclusion follows by the pointwise 
	convergence of $E_{ren}(v_{\ell}, F_{\ell}; \cdot)$ to $E_{ren}(v, F; \cdot)$. 
	
	\item (Convergence of minimizers):
	Let $\{w_{\ell}\}_{\ell\downarrow 0} \subset \mathcal{W}$ be a sequence of minimizers of the sequence of functionals $\{E_{ren}(v_{\ell}, F_{\ell}; \cdot)\}_{\ell\downarrow 0}$ 
	(the existence of minimizers follows from Theorem~\ref{thm:minimizers} \ref{it:existence_min}). As $\{(\xi_{\ell}, v_{\ell}, F_{\ell})\}_\ell$ is
	uniformly bounded in  $\triple$ as $\ell\to 0$, the coercivity Theorem~\ref{thm:minimizers} \ref{it:coercivity} implies via \eqref{coerciv}
	the existence of a constant $C>0$ (uniform in $\ell$) such that for all $\ell$
	\begin{align*}
	0=E_{ren}(v_{\ell}, F_{\ell}, 0)\geq E_{ren}(v_{\ell}, F_{\ell}; w^\ell)
	\geq \frac12 \mathcal{E}(w^\ell)-C.
	\end{align*}
	This implies that $\{w_{\ell}\}_{\ell\downarrow 0}$ belongs to the sublevel set $2C$ of the energy $\mathcal{E}$. Hence, by Lemma 
	\ref{lem:frac-sobolev}, there exists $w\in\mathcal{W}$ such that, upon a subsequence, $w_{\ell} \to w$ strongly in $L^2$. Moreover, $w$ 
	is a minimizer of $E_{ren}(v,F; \cdot)$ over $\mathcal W$ because for every $w_0\in \mathcal W$, by the $\Gamma-\liminf$ inequality 
	and the pointwise convergence of $E_{ren}(v_{\ell}, F_{\ell}; \cdot)$ to $E_{ren}(v, F; \cdot)$, we have
	\begin{align*}
	E_{ren}(v, F; w)&\leq \liminf_{\ell\to 0} E_{ren}(v_{\ell}, F_{\ell}; w^\ell)\leq \limsup_{\ell\to 0} E_{ren}(v_{\ell}, F_{\ell}; w^\ell)\\
	&\leq \limsup_{\ell\to 0} E_{ren}(v_{\ell}, F_{\ell}; w_0)=E_{ren}(v, F; w_0).
	\end{align*}
	Choosing $w_0=w$ in the above relation, we deduce that 
	\begin{equs}
	 E_{ren}(v, F; w)= \lim_{\ell\to 0} E_{ren}(v_{\ell}, F_{\ell}; w^\ell). 
	\end{equs}
	\qedhere
	\end{enumerate}
\end{proof}


%% file: regularity.tex
\section{A priori estimate for minimizers in Hölder spaces} \label{sec:regularity}

In this section we prove an a priori estimate for minimizers of the renormalized energy $E_{ren}$, as stated in 
Proposition~\ref{prop:hoelder_regularity}. We first need the following proposition. 

\begin{proposition} \label{prop:conv-commutator} There exists $C>0$ such that for every $w\in \mathcal W$ and periodic distribution $f$,
\begin{equs}
 {[wf]}_{-\frac{3}{4}} \leq C \mathcal{H}(w)^{\frac{1}{2}} [f]_{-\frac{1}{2}}. 
\end{equs}
\end{proposition}

%
 %
 %
%

\begin{proof} 
\begin{enumerate}[label=\textsc{Case \arabic*},leftmargin=0pt,labelsep=*,itemindent=*]
	\item ($f\in L^2\cap \mathcal{C}^{-\frac{1}{2}}(\TT^2)$). Since $w\in\mathcal{W}$, the product $wf$ belongs to $L^1(\TT^2)$. 
	We estimate $ {[wf]}_{-\frac{3}{4}}$ via \eqref{eq:neg_hoelder_char} by studying the blow-up of $\|(wf)_T\|_{L^\infty}$ for $T\in (0,1]$. We use the ``telescopic'' decomposition
	\begin{equs}
	 (wf)_T = (w f_{\frac{T}{2}})_{\frac{T}{2}} +\sum_{k\geq 2, \, t = \frac{T}{2^k}} \big( (wf_t)_{T-t} - (w f_{2t})_{T-2t} \big).
	 \label{eq:telescopic_sum}
	\end{equs}

	\noindent\textsc{Step 1} (Bound on $\|(w f_{\frac{T}{2}})_{\frac{T}{2}}\|_{L^{\infty}}$): For $p=10$, Young's inequality for convolution in 
	Remark \ref{rem:perio}, Lemma \ref{lem:an_embedding}, \eqref{eq:bound-2/3} and \eqref{eq:supl_numa} yield for every $T\in (0,1]$,
	\begin{equs}
	 \|(w f_{\frac{T}{2}})_{\frac{T}{2}}\|_{L^{\infty}} & \lesssim \left(T^{\frac{1}{3}}\right)^{-\frac{5}{2p}} \|w f_{\frac{T}{2}}\|_{L^p} 
	 \lesssim \left(T^{\frac{1}{3}}\right)^{-\frac{1}{4}} \|w\|_{L^{10}} \|f_{\frac{T}{2}}\|_{L^\infty} \\
	 & \lesssim \left(T^{\frac{1}{3}}\right)^{-\frac14-\frac{1}{2}} \mathcal{H}(w)^{\frac{1}{2}} [f]_{-\frac{1}{2}}
	 = \left(T^{\frac{1}{3}}\right)^{-\frac{3}{4}} \mathcal{H}(w)^{\frac{1}{2}} [f]_{-\frac{1}{2}}.
	\end{equs}
	\noindent\textsc{Step 2} (Bound on the telescopic sum): By Young's inequality for convolution in Remark \ref{rem:perio} and Lemma 
	\ref{lem:L^2_com_est} (see below), we obtain via \eqref{eq:supl_numa} for every $T\in (0,1]$,
	\begin{align*}
	& \Big\|\sum_{k\geq 2, \,  t = \frac{T}{2^k}} \big( (wf_t)_{T-t} - (w f_{2t})_{T-2t} \big) \Big\|_{L^\infty} 
	\\
	& \quad \leq \sum_{k\geq 2, \, t = \frac{T}{2^k}} \big\|\big( (wf_t)_t - w f_{2t} \big)_{T-2t}\big\|_{L^\infty}
	\lesssim \sum_{k\geq 2, \,  t = \frac{T}{2^k}}  \left((T-2t)^{\frac{1}{3}}\right)^{-\frac{5}{4}} \|(wf_t)_t - w f_{2t}\|_{L^2} 
	\\
	& \quad \quad \lesssim \left(T^\frac{1}{3}\right)^{-\frac{5}{4}} \sum_{k\geq 2, \,  t = \frac{T}{2^k}} t^{\frac{1}{3}} \mathcal{H}(w)^{\frac{1}{2}} \|f_t\|_{L^\infty}
	\lesssim \left(T^\frac{1}{3}\right)^{-\frac{5}{4}} \sum_{k\geq 2, \,  t = \frac{T}{2^k}} (t^{\frac{1}{3}})^{\frac{1}{2}} \mathcal{H}(w)^{\frac{1}{2}} [f]_{-\frac{1}{2}} 
	\\
	& \quad \quad \lesssim \left(T^\frac{1}{3}\right)^{-\frac{3}{4}} \mathcal{H}(w)^{\frac{1}{2}} [f]_{-\frac{1}{2}}.
	\end{align*}
	\noindent\textsc{Step 3} (H\"older regularity): By \textsc{Step} 1 and \textsc{Step} 2 we know that for every $T\in (0,1]$,
	\begin{equs}
	 \|(wf)_T\|_{L^\infty} & \lesssim \left(T^{\frac{1}{3}}\right)^{-\frac{3}{4}} \mathcal{H}(w)^{\frac{1}{2}} [f]_{-\frac{1}{2}},
	\end{equs}
	which combined with \eqref{eq:neg_hoelder_char} completes the proof. 

	\medskip
	\item ($f\in \mathcal{C}^{-\frac{1}{2}}(\TT^2)$). We consider an arbitrary approximation $f_{\ell} \in L^2\cap\mathcal{C}^{-\frac{1}{2}}(\TT^2)$ of $f$ with respect to $[\cdot]_{-\frac{1}{2}}$. By \textsc{Case 1} we deduce that $w f_{\ell}$ is a Cauchy sequence in $\mathcal{C}^{-\frac{3}{4}}$, therefore it converges to the product $wf$ by the same argument as in \cite[Lemma 6]{IO19}. \hfill \qedhere
	\end{enumerate}
\end{proof}

\begin{lemma} \label{lem:L^2_com_est} 
There exists a constant $C>0$ such that for every $t\in (0,1]$, $w\in \mathcal W$ and periodic distribution $f$,
	\begin{align*}
		\|(w f_t)_t - w f_{2t}\|_{L^2} \leq C t^{\frac{1}{3}} \mathcal{H}(w)^{\frac{1}{2}} \|f_t\|_{L^\infty}.
	\end{align*}
\end{lemma}

\begin{proof} We start with the identity
$$\big((w f_t)_t - w f_{2t}\big)(x)=\int_{\RR^2} \psi_t(y) (w(x-y)-w(x))f_t(x-y)\, dy, \quad x\in \TT^2.$$
By Minkowski's inequality, we deduce
\begin{align*}
	\|(w f_t)_t - w f_{2t}\|_{L^2} 
	& \leq \|f_t\|_{L^\infty} \int_{\RR^2} |\psi_t(y)| \|w(\cdot - y) - w(\cdot)\|_{L^2} \, \dd y
	\\
	& \leq \|f_t\|_{L^\infty} \int_{\RR^2} |\psi_t(y)| \|\partial_1^{-y_1}w(x_1,x_2 - y_2)\|_{L^2_x} \, \dd y 
	\\
	& \quad + \|f_t\|_{L^\infty} \int_{\RR^2} |\psi_t(y)| 
	\|\partial_2^{-y_2}w\|_{L^2}\, \dd y,
\end{align*}
where we used that $w(x-y) - w(x) = \partial_1^{-y_1}w(x_1,x_2 - y_2) + \partial_2^{-y_2} w(x_1, x_2)$ for every $x\in \TT^2$ and $y\in \RR^2$. 

The first integral can be estimated using the mean value theorem and translation invariance of the torus by
\begin{align*}
	\int_{\RR^2} |\psi_t(y)| \left(\int_{\TT^2} |\partial_1^{-y_1}w(x_1,x_2 - y_2)|^2 \, \dd x\right)^{\frac{1}{2}} \, \dd y & \leq \int_{\RR^2} |y_1\psi_t(y)| \dd y \, \|\partial_1w\|_{L^2} \lesssim t^{\frac{1}{3}} \, \|\partial_1w\|_{L^2},
\end{align*}
since by Step 1 in \cite[proof of Lemma 10]{IO19} we know that $y \mapsto y_1 \psi(y)\in L^1(\RR^2)$. 

The second integral can be estimated using the Cauchy-Schwarz inequality and \eqref{eq:fourier} by
\begin{equs}
	& \int_{\RR^2} |\psi_t(y)| \left(\int_{\TT^2} |\partial_2^{-y_2}w(x)|^2 \, \dd x\right)^{\frac{1}{2}} \, \dd y \\
	& 
	\quad = (t^{\frac{1}{3}})^{\frac{3}{2}\left(\frac{2}{3} +\frac{1}{2}\right)}\int_\RR  \int_\RR 
	\left|\frac{y_2}{(t^{\frac13})^\frac32}\right|^{\frac{2}{3}+\frac{1}{2}}|\psi_t(y)| \, \dd y_1
	\left(\int_{\TT^2} \frac{|\partial_2^{-y_2}w(x)|^2}{|y_2|^{\frac{4}{3}}} \, \dd x\right)^{\frac{1}{2}} \, \frac{\dd y_2}{|y_2|^{\frac{1}{2}}} 
	\\
	& \quad \lesssim (t^{\frac{1}{3}})^{\frac{3}{2}\left(\frac{2}{3} +\frac{1}{2}\right)} 
	\left\|\int_\RR \left|\frac{y_2}{(t^{\frac13})^\frac32}\right|^{\frac{2}{3}+\frac{1}{2}}|\psi_t(y)| \, \dd y_1\right\|_{L^2_{y_2}(\RR)} 
	\||\partial_2|^{\frac{2}{3}} w\|_{L^2} 
	\lesssim t^{\frac{1}{3}} \||\partial_2|^{\frac{2}{3}} w\|_{L^2},
\end{equs}
where we also used Minkowski's inequality and a change of variables to deduce that
\begin{equs}
 \left\|\int_\RR \left|\frac{y_2}{(t^{\frac13})^\frac32}\right|^{\frac{2}{3}+\frac{1}{2}}|\psi_t(y)| \, \dd y_1\right\|_{L^2_{y_2}(\RR)} 
 \leq \frac{1}{(t^{\frac{1}{3}})^{\frac{3}{4}}} \big\|\| |y_2|^{\frac{2}{3}+\frac{1}{2}}\psi(y_1, y_2)\|_{L^2_{y_2}(\RR)}\big\|_{L^1_{y_1}(\RR)},
 \label{eq:y_2_5/6_bound}
\end{equs}
along with the fact that $y_1 \mapsto \| |y_2|^{\frac{2}{3}+\frac{1}{2}}\psi(y_1, y_2)\|_{L^2_{y_2}(\RR)}\in L^1_{y_1}(\RR)$.
\footnote{\label{footnote:y_2_5/6_bound} This follows easily from the bound  
\begin{equs}
 \big\|\| |y_2|^{\frac{2}{3}+\frac{1}{2}}\psi(y_1, y_2)\|_{L^2_{y_2}(\RR)}\big\|_{L^1_{y_1}(\RR)} 
 & \lesssim  \|(1+|y_1|)(1+|y_2|^2)\psi(y)\|_{L^2_y(\RR^2)}
\end{equs}
and Plancherel's identity, using that $\widehat{\psi}(k) = \mathrm{e}^{-|k_1|^3 - k_2^2}$,
see also Step 1 in \cite[proof of Lemma 10]{IO19} and Footnote \ref{footnote:tensorization}.}

Combining the previous estimates with \eqref{eq:bound-2/3} implies the desired bound. 
\end{proof}

We are now ready to prove Proposition~\ref{prop:hoelder_regularity}.

\begin{proof}[Proof of Proposition~\ref{prop:hoelder_regularity}] 
	By Theorem~\ref{thm:minimizers}~\ref{it:existence_min} if $w\in \mathcal{W}$ is a minimizer of $E_{ren}(v,F;\cdot)$, then $w$ is a weak solution to 
	\eqref{eq:ripple_rem}. By the Schauder theory for the operator $\mathcal{L}$ (see \cite[Lemma 5]{IO19}), if $w\in \mathcal W\cap \C^{\frac{5}{4}-2\epsilon}$
	satisfies \eqref{eq:ripple_rem}, we have that
	\begin{align}\label{eq:hoelder-schauder}
	\begin{split}
		[w]_{\frac{5}{4} - 2\varepsilon} 
		& \lesssim \bigg[P \big( F + w R_1 \partial_2 v 
		+ v R_1 \partial_2 w + w R_1 \partial_2 w \\
		& \qquad- \frac{1}{2} (v+w) R_1 \partial_1 
		(v+w)^2\big)
		+ \frac{1}{2} \partial_2 R_1 (v+w)^2 \bigg]_{-\frac{3}{4} - 2\varepsilon}.
	\end{split}
	\end{align}
	We estimate each term on the right-hand side of \eqref{eq:hoelder-schauder} separately. 
	The idea is to bound any term containing $w$ in the seminorm 
	$[\cdot]_{-\frac{3}{4} - 2\varepsilon}$ by a product 
	$\mathcal{H}(w)^{\gamma(\theta)} [w]_{\frac{5}{4} - 2\varepsilon}^{\theta}$ with $\theta\in (0,1)$
	and $\gamma(\theta)>0$. 
	To prove the statement, the main tools are 
	\begin{enumerate}[label=$\circ$,leftmargin=1em]
		\item the interpolation inequality in Lemma \ref{lem:interpolation},
		\item Lemma \ref{lem:hoelder-embedding} which yields that $\mathcal{H}(w)$ controls $[w]_{-\frac{1}{4}}^2$,
		\item \cite[Lemmata 6 and 12]{IO19} stating that for a distribution 
		$f\in \mathcal{C}^{\beta}$, $\beta\in (-\frac32,0)\setminus \{-1, -\frac12\}$, and two functions 
		$g\in\mathcal{C}^{\gamma}$, $\tilde g\in\mathcal{C}^{\tilde \gamma}$ with $\gamma, \tilde \gamma\in (0, \frac32)$ both of vanishing
		average, provided that $\beta + \gamma >0$ and $\tilde \gamma\geq \gamma$, the following estimates hold 
	\begin{align*}
		[fg]_{\beta} \lesssim [f]_{\beta} [g]_{\gamma}
		\quad \textrm{and} \quad
		[g\tilde g]_{\gamma} \lesssim [g]_{\gamma} [\tilde g]_{\tilde \gamma}.
	\end{align*}
	\end{enumerate}
	Also, we use that $\C^\alpha\subset \C^\beta$ for any $-\frac32<\beta<\alpha<\frac32$ with $\alpha, \beta\neq 0$, see \cite[Remark 2]{IO19}, and 
	Lemma \ref{lem:R} which implies that the Hilbert transform reduces the regularity by $\varepsilon$ on H\"older spaces. 
	 
	For $[\xi]_{-\frac{5}{4}-\varepsilon}, [v]_{\frac{3}{4}-\varepsilon}, [F]_{-\frac{3}{4}-\varepsilon}\leq M$, the following estimates 
	hold with an implicit constant depending on $M$ and $\eps$.
	\begin{enumerate}[label=\textbf{\arabic*.},leftmargin=0pt,labelsep=*,itemindent=*]
		\item \textbf{Terms independent of $w$:}
		First, we notice that $[PF]_{-\frac{3}{4}-2\varepsilon}\leq [F]_{-\frac{3}{4}-2\varepsilon} \lesssim [F]_{-\frac{3}{4}-\varepsilon}
		\lesssim1$. Also, by Definition \ref{def:neg_holder}, we have that
		\begin{align*}
			[\partial_2 R_1 v^2]_{-\frac{3}{4}-2\varepsilon} 
			&\lesssim [R_1 v^2]_{\frac{3}{4}-2\varepsilon} 
			\lesssim [v^2]_{\frac{3}{4}-\varepsilon} 
			\lesssim [v]_{\frac{3}{4}-\varepsilon}^2
			\lesssim 1, \\
			[v \partial_1 R_1v^2]_{-\frac{3}{4}-2\varepsilon} 
			&\lesssim [v \partial_1 R_1v^2]_{-\frac{1}{4}-2\varepsilon} 
			\lesssim [v]_{\frac{3}{4}-\varepsilon} [\partial_1 R_1v^2]_{-\frac{1}{4}-2\varepsilon} 
			\lesssim [v]_{\frac{3}{4}-\varepsilon} [v^2]_{\frac{3}{4}-\varepsilon} 
			\lesssim [v]_{\frac{3}{4}-\varepsilon}^3
			\lesssim 1.
		\end{align*}
		
		\item \textbf{Linear terms in $w$:} By the interpolation estimate in Lemma \ref{lem:interpolation} and Lemma \ref{lem:hoelder-embedding},
		we have 
		\begin{equs}
		 {[w]}_{\frac{3}{4}+3\varepsilon} & \lesssim [w]_{-\frac{1}{4}}^{\frac{1}{3} - 2\kappa_1} [w]_{\frac{5}{4} - 2\varepsilon}^{\frac{2}{3} + 2\kappa_1} 
		\lesssim \mathcal{H}(w)^{\frac{1}{6} - \kappa_1} [w]_{\frac{5}{4} - 2\varepsilon}^{\frac{2}{3} + 2\kappa_1},
		 \label{eq:interpolation-3/4} \\
		 {[w]}_{\frac{1}{4}+3\varepsilon} &\lesssim [w]_{-\frac{1}{4}}^{\frac{2}{3} - 2\kappa_2} 
		 [w]_{\frac{5}{4} - 2\varepsilon}^{\frac{1}{3} + 2\kappa_2} \lesssim \mathcal{H}(w)^{\frac{1}{3} - \kappa_2}
		 [w]_{\frac{5}{4} - 2\varepsilon}^{\frac{1}{3} + 2\kappa_2},
		 \label{eq:interpolation-1/4}
		\end{equs}
		where $\kappa_1, \kappa_2>0$ are small (as functions of $\varepsilon$) for $\varepsilon>0$ small enough. 
		This yields
		\begin{align*}
			[w\partial_2 R_1v]_{-\frac{3}{4}-2\varepsilon} 
			&\lesssim [w]_{\frac{3}{4}+3\varepsilon} [\partial_2 R_1 v]_{-\frac{3}{4}-2\varepsilon} 
			\lesssim \mathcal{H}(w)^{\frac{1}{6}-\kappa_1} [w]_{\frac{5}{4} - 2\varepsilon}^{\frac{2}{3}+2\kappa_1} [v]_{\frac{3}{4}-\varepsilon}
			\\
			& \lesssim \mathcal{H}(w)^{\frac{1}{6}-\kappa_1} [w]_{\frac{5}{4} - 2\varepsilon}^{\frac{2}{3}+2\kappa_1},
			\\
			[w \partial_1 R_1 v^2]_{-\frac{3}{4}-2\varepsilon} 
			& \lesssim [w \partial_1 R_1 v^2]_{-\frac{1}{4}-2\varepsilon} 
			\lesssim [w]_{\frac{1}{4}+3\varepsilon} [\partial_1 R_1 v^2]_{-\frac{1}{4}-2\varepsilon} 
			\\
			& \lesssim \mathcal{H}(w)^{\frac{1}{3}-\kappa_2} [w]_{\frac{5}{4} - 2\varepsilon}^{\frac{1}{3}+2\kappa_2} [v]_{\frac{3}{4}-\varepsilon}^2
			\lesssim \mathcal{H}(w)^{\frac{1}{3}-\kappa_2} [w]_{\frac{5}{4} - 2\varepsilon}^{\frac{1}{3}+2\kappa_2},
			\\
			[v \partial_2 R_1w]_{-\frac{3}{4}-2\varepsilon} 
			&\lesssim [v \partial_2 R_1w]_{-\frac{3}{4}+2\varepsilon} 
			\lesssim [v]_{\frac{3}{4}-\varepsilon} [\partial_2 R_1w]_{-\frac{3}{4}+2\varepsilon} 
			\lesssim [w]_{\frac{3}{4}+3\varepsilon} 
			\\
			& \lesssim  \mathcal{H}(w)^{\frac{1}{6}-\kappa_1} [w]_{\frac{5}{4} - 
			2\varepsilon}^{\frac{2}{3}+2\kappa_1},
			\\
			[\partial_2 R_1(vw)]_{-\frac{3}{4}-2\varepsilon} 
			&\lesssim  [vw]_{\frac{3}{4}-\varepsilon} 
			\lesssim [v]_{\frac{3}{4}-\varepsilon} [w]_{\frac{3}{4}+3\varepsilon}
			\lesssim \mathcal{H}(w)^{\frac{1}{6}-\kappa_1} [w]_{\frac{5}{4} - 2\varepsilon}^{\frac{2}{3}+2\kappa_1}, 
			\\
			[v \partial_1 R_1(vw)]_{-\frac{3}{4}-2\varepsilon}
			&\lesssim [v \partial_1 R_1(vw)]_{-\frac{1}{4}-2\varepsilon}
			\lesssim [v]_{\frac{3}{4}-\varepsilon} [\partial_1 R_1(vw)]_{-\frac{1}{4}-2\varepsilon}
			\lesssim  [vw]_{\frac{3}{4}-\varepsilon} \\
			&\lesssim [v]_{\frac{3}{4}-\varepsilon} {[w]}_{\frac{3}{4}+3\varepsilon}
			\lesssim \mathcal{H}(w)^{\frac{1}{6}-\kappa_1} [w]_{\frac{5}{4} - 2\varepsilon}^{\frac{2}{3}+2\kappa_1}.
		\end{align*}

		\item \textbf{Quadratic terms in $w$:}
		    
		\begin{enumerate}
			\item We start with the term $\partial_2 R_1w^2$. By the interpolation estimate in Lemma \ref{lem:interpolation} 
			for $\gamma=0$ and Lemma \ref{lem:hoelder-embedding}, we have
			\begin{equs}
			 \|w\|_{L^\infty} \lesssim [w]_{-\frac{1}{4}}^{\frac{5}{6}-2\kappa_3} 
			 [w]_{\frac{5}{4} - 2\varepsilon}^{\frac{1}{6}+2\kappa_3} \lesssim \mathcal{H}(w)^{\frac{5}{12}-\kappa_3} 
			 [w]_{\frac{5}{4} - 2\varepsilon}^{\frac{1}{6}+2\kappa_3}, \label{eq:interpolation-infty}
			\end{equs}
			where $\kappa_3>0$ is small for $\varepsilon>0$ small enough. 
			Together with \eqref{eq:interpolation-3/4}, it follows that
			\begin{align*}
				[\partial_2 R_1w^2]_{-\frac{3}{4}-2\varepsilon}
				&\lesssim [\partial_2 R_1w^2]_{-\frac{3}{4}+2\varepsilon} 
				\lesssim [w^2]_{\frac{3}{4}+3\varepsilon} 
				\lesssim \|w\|_{L^{\infty}} [w]_{\frac{3}{4}+3\varepsilon} \\
				&\lesssim \mathcal{H}(w)^{\frac{7}{12}-(\kappa_1+\kappa_3)} [w]_{\frac{5}{4} - 2\varepsilon}^{\frac{5}{6}+2(\kappa_1+\kappa_3)}.
			\end{align*}
			Similarly, we estimate 
			\begin{align*}
				[v \partial_1 R_1 w^2]_{-\frac{3}{4}-2\varepsilon} 
				&\lesssim [v \partial_1 R_1 w^2]_{-\frac{1}{4}+2\varepsilon} 
				\lesssim [v]_{\frac{3}{4}-\varepsilon} [\partial_1 R_1 w^2]_{-\frac{1}{4}+2\varepsilon} 
				\lesssim [w^2]_{\frac{3}{4}+3\varepsilon} \\
				&\lesssim  \mathcal{H}(w)^{\frac{7}{12}-(\kappa_1+\kappa_3)} [w]_{\frac{5}{4} - 2\varepsilon}^{\frac{5}{6}+2(\kappa_1+\kappa_3)}.
			\end{align*}
			      
			\item The term $w \partial_2R_1w$ is treated via Proposition \ref{prop:conv-commutator},
			\begin{equation*}
			[w \partial_2R_1w]_{-\frac{3}{4}} \lesssim \mathcal{H}(w)^{\frac{1}{2}}
			[\partial_2R_1 w]_{-\frac{1}{2}}\lesssim \mathcal{H}(w)^{\frac{1}{2}} [R_1w]_{1} \lesssim 
			\mathcal{H}(w)^{\frac{1}{2}} [w]_{1+\varepsilon}.
			\end{equation*}
			Then, Lemma \ref{lem:interpolation} and Lemma \ref{lem:hoelder-embedding} yield 
			\begin{equs}
			 {[w]}_{1+\varepsilon} \lesssim  [w]_{-\frac{1}{4}}^{\frac{1}{6} - 2\kappa_4} [w]_{\frac{5}{4}- 2\varepsilon}^{\frac{5}{6} + 2\kappa_4}\lesssim \mathcal{H}(w)^{\frac{1}{12} - \kappa_4} [w]_{\frac{5}{4}- 2\varepsilon}^{\frac{5}{6} + 2\kappa_4},
			\end{equs}
			where $\kappa_4>0$ is small for $\varepsilon>0$ small enough. Hence we have 
			\begin{equs}
			 {[w \partial_2R_1w]}_{-\frac{3}{4}-2\eps} \lesssim  {[w \partial_2R_1w]}_{-\frac{3}{4}} \lesssim \mathcal{H}(w)^{\frac{7}{12} - \kappa_4} [w]_{\frac{5}{4}- 2 \varepsilon}^{\frac{5}{6} + 2\kappa_4}.
			\end{equs}
			\item  We decompose $wR_1\partial_1(vw)$ into $wR_1(w\partial_1 v) + w R_1(v\partial_1 w)$ and treat each term separately. 
			By Proposition \ref{prop:conv-commutator}, we have  
			\begin{align*}
			{[}wR_1(v\partial_1 w){]}_{-\frac{3}{4}} & \lesssim \mathcal{H}(w)^{\frac{1}{2}} [R_1(v\partial_1 w)]_{-\frac{1}{2}} 
			\lesssim \mathcal{H}(w)^{\frac{1}{2}} [R_1(v\partial_1 w)]_{-\frac{1}{2}+\varepsilon} \\
			&\lesssim \mathcal{H}(w)^{\frac{1}{2}} [v\partial_1 w]_{-\frac{1}{2}+2\varepsilon}
			\lesssim  \mathcal{H}(w)^{\frac{1}{2}} [v]_{\frac{3}{4} - \varepsilon} [\partial_1 w]_{-\frac{1}{2}+2\varepsilon}
			\lesssim
			\mathcal{H}(w)^{\frac{1}{2}} [w]_{\frac{1}{2}+2\varepsilon}. 
			\end{align*}
			Again, Lemma \ref{lem:interpolation} and Lemma \ref{lem:hoelder-embedding} yield  
			\begin{equs}
			 {[w]}_{\frac{1}{2} + 2\varepsilon} & \lesssim [w]_{-\frac{1}{4}}^{\frac{1}{2}-2\kappa_5} [w]_{\frac{5}{4} - 2\varepsilon}^{\frac{1}{2}+2\kappa_5}
			 \lesssim \mathcal{H}(w)^{\frac{1}{4}-\kappa_5} [w]_{\frac{5}{4} - 2\varepsilon}^{\frac{1}{2}+2\kappa_5},
			 \label{eq:interpolation-1/2}
			\end{equs}
			where $\kappa_5>0$ is small for $\eps>0$ small enough. Hence we obtain that
			\begin{equs}
		 	 {[}wR_1(v\partial_1 w){]}_{-\frac{3}{4}-2\eps}\lesssim	 {[}wR_1(v\partial_1 w){]}_{-\frac{3}{4}} 
		 	 & \lesssim
			 \mathcal{H}(w)^{\frac{3}{4}-\kappa_5} [w]_{\frac{5}{4} - 2\varepsilon}^{\frac{1}{2}+2\kappa_5}.
			\end{equs}

			The term $wR_1(w\partial_1 v)$ can be estimated using \eqref{eq:interpolation-1/4} as
			\begin{align*}
			[wR_1(w\partial_1 v)]_{-\frac{3}{4}-2\varepsilon} 
			& \lesssim [wR_1(w\partial_1 v)]_{-\frac{1}{4}-2\varepsilon} 
			\lesssim [w]_{\frac{1}{4}+3\varepsilon} [R_1(w\partial_1v)]_{-\frac{1}{4} -2\varepsilon} 
			\\ 
			&\lesssim [w]_{\frac{1}{4}+3\varepsilon} [w\partial_1v]_{-\frac{1}{4} - \varepsilon} 
			\lesssim [w]_{\frac{1}{4}+3\varepsilon}^2  [\partial_1v]_{-\frac{1}{4} - \varepsilon} 
			\\
			&\lesssim \mathcal{H}(w)^{\frac{2}{3} - 2\kappa_2} 
			[w]_{\frac{5}{4} - 2\varepsilon}^{\frac{2}{3} + 4\kappa_2}[v]_{\frac{3}{4} -\varepsilon} 
			\lesssim \mathcal{H}(w)^{\frac{2}{3} - 2\kappa_2} [w]_{\frac{5}{4} - 2\varepsilon}^{\frac{2}{3} + 4\kappa_2}.
			\end{align*}
		\end{enumerate}
		
		\item \textbf{Cubic term in $w$:}
		The cubic term $w\partial_1R_1w^2$ is treated by Proposition \ref{prop:conv-commutator} which yields
		\begin{align*}
			[w\partial_1 R_1w^2]_{-\frac{3}{4}} 
			& \lesssim \mathcal{H}(w)^{\frac{1}{2}} [R_1 \partial_1 w^2]_{-\frac{1}{2}} 
			\lesssim \mathcal{H}(w)^{\frac{1}{2}} [R_1\partial_1 w^2]_{-\frac{1}{2}+ \varepsilon}
			\lesssim \mathcal{H}(w)^{\frac{1}{2}} [\partial_1 w^2]_{-\frac{1}{2}+2\varepsilon}\\
			& \lesssim \mathcal{H}(w)^{\frac{1}{2}} [w^2]_{\frac{1}{2}+2\varepsilon}  
			\lesssim \mathcal{H}(w)^{\frac{1}{2}} 
			\|w\|_{L^\infty} [w]_{\frac{1}{2}+2\varepsilon}.
		\end{align*}
		By \eqref{eq:interpolation-infty} and \eqref{eq:interpolation-1/2} we have that
		\begin{equs}
		 \|w\|_{L^\infty} [w]_{\frac{1}{2}+2\varepsilon} & \lesssim \mathcal{H}(w)^{\frac{2}{3} - (\kappa_3+\kappa_5)} 
		 [w]_{\frac{5}{4} - 2\varepsilon}^{\frac{2}{3} + 2(\kappa_3+\kappa_5)},
		\end{equs}
		which in turn implies that 
		\begin{equs}
		 {[w\partial_1 R_1w^2]}_{-\frac{3}{4}-2\eps} \lesssim {[w\partial_1 R_1w^2]}_{-\frac{3}{4}} & \lesssim \mathcal{H}(w)^{\frac{3}{2} - (\kappa_3+\kappa_5)} 
		 [w]_{\frac{5}{4} - 2\varepsilon}^{\frac{2}{3} + 2(\kappa_3+\kappa_5)}.
		\end{equs}
	\end{enumerate}
	Summing up, Young's inequality yields the bound
	\begin{equs}
	 {} [w]_{\frac{5}{4}-2\eps} \leq C (1+\mathcal{H}(w))^N, \label{eq:harm_contr} 
	\end{equs}
	for some $N\geq 1$, and by our estimates it is clear that the constant $C$ depends polynomially on $M$. To conclude, using the fact that 
	$w$ is a minimizer of $E_{ren}(v,F;\cdot)$, we have that $E_{ren}(v,F;w)\leq E_{ren}(v,F;0)=0$. Since $E_{ren}(v.F;\cdot) = \mathcal{E}+\mathcal{G}(v,F;\cdot)$
	and by Theorem~\ref{thm:minimizers}~\ref{it:coercivity} we know that $|\mathcal{G}(v,F;w)| \leq \frac{1}{2} \mathcal{E}(w) + C$ for some constant $C$ 
	which also depends polynomially on $M$, we obtain that $\mathcal{E}(w) \leq 2 C$. By \eqref{eq:bound-harmonic}, this implies that $H(w) \leq C$ for some constant
	$C$ which depends polynomially on $M$ and combining with \eqref{eq:harm_contr} we obtain the desired bound.
	\end{proof}


%% file: spectralgap-6.tex
\section{Approximations to white noise under the spectral gap assumption}\label{sec:nongaussian} 

The main goal of this section is to prove Proposition \ref{Satz1_pavlos}. Given a probability measure $\lng \cdot \rng$ which satisfies Assumption~\ref{ass}, 
we prove that $\lng\cdot\rng$ is concentrated on $\C^{-\frac{5}{4}-}$, thus, by Schauder theory for the operator $\mathcal{L}$, 
$v:=\mathcal{L}^{-1}P\xi\in\C^{\frac{3}{4}-}$ $\lng\cdot\rng$-almost surely, and we construct $F$ as the $L^p_{\lng\cdot\rng}\C^{-\frac{3}{4}-}$-limit of 
the sequence $\{v\partial_2R_1 v_t\}_{t=2^{-n}\downarrow 0}$. This will allow us to lift $\langle\cdot\rangle$
to a probability measure $\langle\cdot\rangle^{\mathrm{lift}}$ on $\C^{-\frac{5}{4}-} \times \C^{\frac{3}{4}-} \times \C^{-\frac{3}{4}-}$ in a continuous way, 
i.e., given a sequence of probability measures $\{\lng \cdot \rng_{\ell}\}_{\ell\downarrow0}$ which satisfy Assumption~\ref{ass} and converge weakly 
to a limit $\langle\cdot\rangle$ as $\ell \downarrow 0$, then $\{\langle\cdot\rangle_{\ell}^{\mathrm{lift}}\}_{\ell\downarrow0}$ converges 
weakly to $\langle\cdot\rangle^{\mathrm{lift}}$.

The proof of Proposition \ref{Satz1_pavlos} is based on suitable estimates on the $p$-moments 
of multilinear expressions in the corresponding stochastic objects. In the case of Gaussian
approximations, it is enough to bound the second moments, since we can use Nelson's hypercontractivity estimate to bound the $p$-moments in a finite Wiener chaos by the second moments for every $p>2$ (see \cite[Lemmata 4 and 8]{IO19}). On the other hand, for non-Gaussian approximations 
one has to find alternative methods to estimate the $p$-moments. This has been achieved with great 
success in the last few years under very mild assumptions on the random field, see for example 
\cite{HS17,CH16}. In these works a direct computation of the $p$-moments is made through 
explicit formulas in terms of the cumulant functions of the random field and the final bounds are 
obtained by combinatorial arguments. 

Here we are interested in approximations of white noise that satisfy the spectral gap inequality 
\eqref{eq:SG}, uniformly in the approximation parameter $\ell$. This covers the Gaussian
case, but it allows for more general random fields. The basic observation is 
that one can bound the $p$-moments directly by estimating the derivative with respect to the noise,
based on the following consequence of the spectral gap assumption \eqref{eq:SG}.

\begin{proposition} \label{prop:SG_mult} The spectral gap inequality \eqref{eq:SG} implies 
\begin{equs}
 \left\lng \left|G(\xi) -
 \left\lng G(\xi) 
 \right\rng\right|^{2p}\right\rng^{\frac{1}{2p}}
 \leq C(p) \left\lng \left\|\frac{\partial}{\partial\xi} G(\xi) \right\|_{L^2}^{2p} 
 \right\rng^{\frac{1}{2p}},
 \label{eq:SG_mult}
\end{equs}
for every $1\leq p<\infty$ and every functional $G$ on periodic Schwartz distributions which can be approximated by 
cylindrical functionals with respect to the norm 
$\lng|G(\xi)|^{2p}\rng^{\frac{1}{2p}} + \lng \|\frac{\partial}{\partial\xi} G(\xi)\|_{L^2}^{2p} \rng^{\frac{1}{2p}}$,
where the constant $C(p)>0$ depends only on $p$.
\end{proposition}

\begin{remark} \label{rem:Lp_sg} As in Remark~\ref{rem:wn}, using \eqref{eq:SG_mult} we can extend $\xi(\varphi)$ for $\varphi\in L^2(\TT^2)$ as a centered random variable in $L^{2p}_{\lng\cdot\rng}$, 
admissible in \eqref{eq:SG_mult} for any $1\leq p <\infty$.
\end{remark}

\begin{proof} The proof follows \cite[Lemma 3.1]{JO20}. Let $p>1$. We assume that $G$ is a cylindrical functional on periodic Schwartz distributions of finite norm 
$\lng|G(\xi)|^{2p}\rng^{\frac{1}{2p}} + \lng \|\frac{\partial}{\partial\xi} G(\xi)\|_{L^2}^{2p} \rng^{\frac{1}{2p}}$. The general case 
follows by approximation. Without loss of generality we can assume that $\lng G(\xi) \rng =0$. For $\lambda\in(0,1]$ we consider the functional
	\begin{align*}
		F_\lambda(\xi) := (G(\xi)^2+\lambda^2)^\frac{p}{2}.
	\end{align*}
Noting that $\frac{\partial}{\partial \xi} F_\lambda(\xi) = p F_\lambda(\xi)^\frac{p-2}{p} G(\xi) \frac{\partial}{\partial \xi} G(\xi)$ and
using H\"older's inequality, \eqref{eq:SG} applied to $F_{\lambda}$ yields
\begin{equs}
 \left \lng \big| F_\lambda(\xi) - \lng F_\lambda(\xi) \rng \big|^2 \right\rng
 & \leq p^2 \left \lng F_\lambda(\xi)^\frac{2(p-2)}{p} G(\xi)^2 \left\|\frac{\partial}{\partial \xi} G(\xi)\right\|_{L^2}^2 \right \rng
   \leq p^2 \left \lng F_\lambda(\xi)^\frac{2(p-1)}{p} \left\|\frac{\partial}{\partial \xi} G(\xi)\right\|_{L^2}^2 \right \rng
 \\
 & \leq p^2 \left \lng F_\lambda(\xi)^2 \right \rng^{\frac{p-1}{p}}
 \left \lng \left\|\frac{\partial}{\partial \xi} G(\xi)\right\|_{L^2}^{2p} \right \rng^{\frac{1}{p}}.
 \label{eq:bnd1}
\end{equs}
By Cauchy--Schwarz we see that
\begin{equs}
 \lng F_\lambda(\xi) \rng^2 \leq \left\lng (G(\xi)^2+\lambda^2)^{p-1} \right\rng
 \left\lng G(\xi)^2+\lambda^2 \right\rng.
\end{equs}
By Hölder's inequality and \eqref{eq:SG} applied to $G(\xi)$ (recall that $\lng G(\xi) \rng = 0$) 
the two terms on the right hand side can be estimated as
\begin{equs}
 & \left\lng (G(\xi)^2+\lambda^2)^{p-1} \right\rng
 \leq \lng F_\lambda(\xi)^2 \rng^{\frac{p-1}{p}},
 \\
 & \left\lng G(\xi)^2+\lambda^2 \right\rng
 \leq \left\lng \left\|\frac{\partial}{\partial \xi} G(\xi)\right\|_{L^2}^2 \right\rng + \lambda^2
 \leq \left\lng \left\|\frac{\partial}{\partial \xi} G(\xi)\right\|_{L^2}^{2p}\right\rng^{\frac{1}{p}}
 + \lambda^2.
\end{equs}
Combining with \eqref{eq:bnd1} we get 
\begin{equs}
 \left \lng F_\lambda(\xi)^2 \right \rng
 & \lesssim \left\lng |F_\lambda(\xi) - \lng F_\lambda(\xi) \rng_\ell|^2 \right \rng
 + \left \lng F_\lambda(\xi) \right \rng^2
 \\
 & \lesssim_p \left \lng F_\lambda(\xi)^2 \right \rng^{\frac{p-1}{p}}
 \left \lng \left\|\frac{\partial}{\partial \xi} G(\xi)\right\|_{L^2}^{2p} \right \rng^{\frac{1}{p}}
 + \lambda^2 \lng F_\lambda(\xi)^2 \rng^{\frac{p-1}{p}}. 
\end{equs}
By Young's inequality we finally obtain that
\begin{equs}
 \left \lng (G(\xi)^2+\lambda^2)^p \right \rng = \left \lng F_\lambda(\xi)^2 \right \rng
 \lesssim_p 
 \left\lng \left\|\frac{\partial}{\partial \xi} G(\xi)\right\|_{L^2}^{2p}\right\rng + \lambda^{2p}.
\end{equs}
The conclusion follows by the monotone convergence theorem as $\lambda \searrow 0$.  
\end{proof}

Proposition \ref{prop:SG_mult} allows us to estimate the $p$-moments of multilinear expressions in $\xi$ (shifted by 
their expectation) by the operator norm of (random) linear functionals on $L^2(\TT^2)$ (in the case of white noise this is the 
Cameron--Martin space), after taking one derivative with respect to $\xi$. To estimate the operator norm of these linear 
functionals we use the regularizing properties of $\mathcal{L}^{-1}$ in Sobolev spaces. 

\subsection{Estimates on \texorpdfstring{$\xi$}{xi} and \texorpdfstring{$v$}{v}} \label{s:xi_v}

In this section, we prove several stochastic estimates for $\xi$ and $v:=\mathcal{L}^{-1}P\xi$ which are uniform in the class 
of probability measures satisfying Assumption~\ref{ass}.
As a corollary, we obtain that the law of $(\xi,v)$ is concentrated on $\C^{-\frac{5}{4}-\eps}\times \C^{\frac{3}{4}-\eps}$ (see Corollary~\ref{cor:xi_v}). A similar result
was proved in \cite[Lemma 4]{IO19} for $\lng\cdot\rng$ being the law of white noise. Here we consider more general probability measures which are not necessarily Gaussian. 
Some of the results of this section will be used in Section \ref{s:F} below in the construction of $F$.

We start with the following proposition. 

\begin{proposition} \label{prop:xi_v_clean} Let $\lng\cdot\rng$ satisfy Assumption~\ref{ass} and let 
$v := \mathcal{L}^{-1} P \xi$. For every $1\leq p<\infty$, $T\in(0,1]$, and $y\in\TT^2$ we have that
\begin{equs}
 \sup_{x\in \TT^2} \left \lng |\xi_T(x)|^p \right \rng^{\frac{1}{p}}
 & \leq C \left(T^{\frac{1}{3}}\right)^{-\frac{5}{4}}, \label{eq:xi_clean} 
 \\
 \sup_{x\in \TT^2} \left \lng |v(x-y) -v(x)|^p \right \rng^{\frac{1}{p}} 
 & \leq C d(0,y)^{\frac{3}{4}}, \label{eq:v_hoelder_clean}
 \\
 \sup_{x\in \TT^2} \left \lng |(\partial_2R_1v)_T(x)|^p \right \rng^{\frac{1}{p}}
 & \leq C \left(T^{\frac{1}{3}}\right)^{-\frac{3}{4}}, \label{eq:d2Rv_clean}
\end{equs}
where the constant $C$ depends only on $p$. 
\end{proposition}

\begin{proof} The proof is based on a direct application of the spectral gap inequality \eqref{eq:SG_mult}. In the following, 
for every $\delta\xi\in L^2(\TT^2)$ we consider $\delta v$ the unique solution of zero average in $x_1$ of 
$\mathcal{L} \delta v = P \delta \xi$. 

\smallskip

\noindent\emph{Proof of \eqref{eq:xi_clean}}: Let $x_0\in\TT^2$ be fixed. We consider the linear functional $G(\xi)=\xi_T(x_0)$.
Since $\xi_T = \xi * \psi_T$, by Assumption~\ref{ass} \ref{item:Def-centered} we have that 
$\lng \xi_T(x_0) \rng=0$. Then, by \eqref{eq:SG_mult} applied to $G$ (which is a cylindrical functional), we get that for every $q\geq 1$
\begin{equs}
 \left \lng |\xi_T(x_0)|^{2q} \right \rng^{\frac{1}{2q}} \lesssim 
 \left\lng \left\| \frac{\partial}{\partial \xi} \xi_T(x_0)\right\|_{L^2}^{2q}\right \rng^{\frac{1}{2q}}. 
\end{equs}
It is easy to check that $\frac{\partial}{\partial \xi} \xi_T(x_0)=\Psi_T(x_0-\cdot)\in L^2(\TT^2)$ 
(with $\Psi_T$ the periodization of $\psi_T$), and by Remark \ref{rem:perio}
\begin{equs}
 \left\| \frac{\partial}{\partial \xi} \xi_T(x_0)\right\|_{L^2} 
 = \|\Psi_T(x_0-\cdot)\|_{L^2} = \|\Psi_T\|_{L^2}
 \lesssim \left(T^{\frac{1}{3}}\right)^{-\frac{5}{4}},
\end{equs}
which proves \eqref{eq:xi_clean} for every $2\leq p<\infty$, as the implicit constant above does not depend on $x_0$. For $p\in[1,2)$, 
the conclusion then follows by Jensen's inequality.

\smallskip

\noindent\emph{Proof of \eqref{eq:v_hoelder_clean}}: Let $x_0\in\TT^2$ be fixed. For every $y\in\TT^2$, recalling the kernel $\Gamma$ in \eqref{eq:gamma}, we consider the linear 
functional $G(\xi) = v(x_0-y) -v(x_0) = \xi(\Gamma(x_0-y-\cdot)) - \xi(\Gamma(x_0-\cdot))$, which is well-defined as a centered random variable
in $L^{2q}_{\lng\cdot\rng}$ for any $q\geq 1$ by Remark~\ref{rem:Lp_sg} and \eqref{eq:gamma_finite} and is admissible in \eqref{eq:SG_mult}. Then
\begin{equs}
 \frac{\partial}{\partial \xi}G(\xi) : \delta \xi \mapsto \delta v(x_0-y) -\delta v(x_0).
\end{equs}
By \eqref{eq:hoelder_reg} we know that 
\begin{equs}
 |\delta v(x_0-y) -\delta v(x_0)| \leq [\delta v]_{\frac{3}{4}} d(0,y)^{\frac{3}{4}} \lesssim d(0,y)^{\frac{3}{4}} \|\delta \xi\|_{L^2},
\end{equs}
which in turn implies that 
\begin{equs}
 \left\|\frac{\partial}{\partial \xi}G(\xi) \right\|_{L^2} \lesssim d(0,y)^{\frac{3}{4}}.
\end{equs}
%

Thus, by \eqref{eq:SG_mult} applied to $G$, we have for every $q\geq 1$,
\begin{equs}
 \left \lng |v(x_0-y) -v(x_0)|^{2q} \right \rng^{\frac{1}{2q}} \lesssim d(0,y)^{\frac{3}{4}},
\end{equs}
which proves \eqref{eq:v_hoelder_clean} for $2\leq p <\infty$, since the implicit constant does not depend on $x_0$. For $p \in[1,2)$, the
conclusion follows by Jensen's inequality. 

\smallskip

\noindent\emph{Proof of \eqref{eq:d2Rv_clean}}: Let $x_0\in\TT^2$ be fixed. We consider the linear functional $G(\xi) = (\partial_2R_1v)_T(x_0)$
(which is cylindrical). We have that
\begin{equs}
 \frac{\partial}{\partial \xi} G(\xi): \delta \xi \mapsto (\partial_2R_1\delta v)_T(x_0).
\end{equs}
By Young's inequality for convolution (see Remark \ref{rem:perio}), \eqref{eq:sobolev_reg_3}, and the fact that $R_1$ is bounded on
$L^{\frac{10}{3}}(\TT^2)$, we get that
\begin{equs}
 |(\partial_2R_1\delta v)_T(x_0)| \lesssim \left(T^{\frac{1}{3}}\right)^{-\frac{3}{4}} \|\partial_2R_1\delta v\|_{L^{\frac{10}{3}}} 
 \lesssim \left(T^{\frac{1}{3}}\right)^{-\frac{3}{4}} \|\partial_2\delta v\|_{L^{\frac{10}{3}}} \lesssim 
 \left(T^{\frac{1}{3}}\right)^{-\frac{3}{4}} \|\delta \xi\|_{L^2},
\end{equs}
yielding the estimate
\begin{equs}
 \left\|\frac{\partial}{\partial \xi} G(\xi) \right\|_{L^2} \lesssim \left(T^{\frac{1}{3}}\right)^{-\frac{3}{4}}.
\end{equs}
We also note that $\lng G(\xi)\rng = 0$. Indeed, since $\lng v(x)\rng = \lng \xi(\Gamma(x-\cdot)) \rng = 0$, for every $x\in \TT^2$, 
we get that $\lng(\partial_2R_1v)_T(x_0)\rng = 0$. The conclusion then follows as above. 
\end{proof}

As a corollary of \eqref{eq:xi_clean} and the Schauder theory for the operator $\mathcal{L}$, we prove that the 
laws of $\xi$ and $v$ are concentrated on $\C^{-\frac{5}{4}-\eps}$ and $\C^{\frac{3}{4}-\eps}$. 

\begin{corollary} \label{cor:xi_v} Let $\lng\cdot \rng$ satisfy Assumption~\ref{ass}. For every $\varepsilon\in (0, \frac1{100})$
and $1\leq p<\infty$ there holds
\begin{equs}
 & \left\lng [\xi]_{-\frac{5}{4}-\varepsilon}^p \right\rng^{\frac{1}{p}} \leq C, \label{eq:xi_bd}
 \\
 & \left\lng [v]_{\frac{3}{4}-\varepsilon}^p \right\rng^{\frac{1}{p}} \leq C, \label{eq:v_bd}
\end{equs}
where the constant $C$ depends only on $p$ and $\eps$. Moreover, $\lng v(x) \rng = 0$ for every $x\in \TT^2$.
\end{corollary} 

\begin{proof} The proof of \eqref{eq:xi_bd} follows by \eqref{eq:xi_clean} and Lemma \ref{lem:mod}. To prove 
\eqref{eq:v_bd}, we use Schauder theory for the operator $\mathcal{L}$ (see \cite[Lemma 5]{IO19}), which implies 
\begin{equs}
 {} [v]_{\frac{3}{4}-\eps} = [\mathcal{L}^{-1} P\xi]_{\frac{3}{4}-\varepsilon} \lesssim {[P\xi]}_{-\frac{5}{4}-\varepsilon} 
 \lesssim {[\xi]}_{-\frac{5}{4}-\varepsilon}.
\end{equs}
As in the proof of Proposition \ref{prop:xi_v_clean}, $\lng v(x) \rng = \lng \xi(\Gamma(x-\cdot)) \rng=0$ for every $x\in \TT^2$,
by Remark~\ref{rem:Lp_sg} and \eqref{eq:gamma_finite}. 
\end{proof}

\subsection{Estimates on \texorpdfstring{$F$}{F}} \label{s:F}

In this section we use the spectral gap inequality \eqref{eq:SG_mult} and \eqref{eq:v_hoelder_clean} and \eqref{eq:d2Rv_clean} to construct $F$ as the 
$L^p_{\lng\cdot\rng} \C^{-\frac{3}{4}-\eps}$-limit of the sequence of random variables $\{v \partial_2 R_1 v_t\}_{t\downarrow 0}$. A similar result was proved
in \cite[Lemma 8]{IO19} for $\lng\cdot\rng$ being the law of white noise (but instead considering approximations $\{v_\ell \partial_2R_1v_\ell\}_{\ell\downarrow 0}$,
where $v_\ell:=\mathcal{L}^{-1}P\xi_\ell$ and $\xi_\ell:=\phi_\ell*\xi$, for a suitable mollifier $\phi_\ell$) using Nelson's hypercontractivity estimate. 
As in Section \ref{s:xi_v}, our estimate holds for more general probability measures which are not necessarily Gaussian. 

In what follows, we use the convolution-commutator
\begin{equs}
 \lceil v, (\cdot)_s\rceil (\partial_2R_1v)_s :=  v(\partial_2 R_1v)_{2s}-(v(\partial_2 R_1v)_s)_s. 
\end{equs}

We will need the following lemma, based on \eqref{eq:v_hoelder_clean} and \eqref{eq:d2Rv_clean}. 

\begin{lemma}\label{lem:commutator}
	Let $\lng\cdot\rng$ satisfy Assumption~\ref{ass} and for a periodic Schwartz distribution $\xi$ let 
	$v=\mathcal{L}^{-1}P\xi$. Then for any $1\leq p<\infty$ and $s,S\in(0,1]$ there holds 
\begin{equs}
 \sup_{x\in\TT^2} \left \lng \left|\left(\lceil v, (\cdot)_s\rceil (\partial_2R_1v)_s\right)_{S}(x)\right|^p \right \rng^{\frac{1}{p}}
 \leq C \left(\left(S^{\frac{1}{3}}\right)^{-\frac{1}{4}} \left(s^{\frac{1}{3}}\right)^{\frac{1}{4}} +
 \left(S^{\frac{1}{3}}\right)^{-\frac{3}{4}} \left(s^{\frac{1}{3}}\right)^{\frac{3}{4}}\right),
 \label{eq:F_main_com}
\end{equs}
where the constant $C$ depends only on $p$. 
\end{lemma}

\begin{proof} \eqref{eq:F_main_com} is a consequence of the following two claims:
\begin{enumerate}[label=\textsc{Claim \arabic*},leftmargin=0pt,labelsep=*,itemindent=*]
	\item\label{claim1} For every $s, S\in(0,1]$ and $q\geq 1$,
\begin{equs}
 \sup_{x\in\TT^2} \left\lng\left\|\frac{\partial}{\partial \xi}\big(\lceil v, (\cdot)_s\rceil (\partial_2R_1v)_s\big)_{S}(x)\right\|_{L^2}^{2q}\right\rng^{\frac{1}{2q}}
 & \lesssim \left(S^{\frac{1}{3}}\right)^{-\frac{1}{4}} \left(s^{\frac{1}{3}}\right)^{\frac{1}{4}} +
 \left(S^{\frac{1}{3}}\right)^{-\frac{3}{4}} \left(s^{\frac{1}{3}}\right)^{\frac{3}{4}}. \quad
 \label{eq:F_main_com_1st_der}
\end{equs}
	\item\label{claim2} For every $s,S\in(0,1]$ and $x_0 \in \TT^2$,
\begin{equs}
 \lng \left(\lceil v, (\cdot)_s\rceil (\partial_2R_1v)_s\right)_{S}(x_0) \rng
 = \lng (v(\partial_2R_1v)_{2s})_S(x_0) \rng - 
 \lng (v(\partial_2R_1v)_{s})_{s+S}(x_0) \rng =0. \quad \label{eq:exp_van} 
\end{equs}
\end{enumerate}

Assuming that these claims hold, we may apply for fixed $x_0\in\TT^2$ the $L^p$-version \eqref{eq:SG_mult} of the spectral gap inequality to the functional  
$G: \xi\mapsto\left(\lceil v, (\cdot)_s\rceil (\partial_2R_1v)_s\right)_{S}(x_0)$, which yields \eqref{eq:F_main_com} by \ref{claim1} and \ref{claim2}
for $2\leq p<\infty$. For $p\in[1,2)$ \eqref{eq:F_main_com} follows by Jensen's inequality. 

It is easy to see that 
\begin{align*}
	G(\xi) = \left(\lceil v, (\cdot)_s\rceil (\partial_2R_1v)_s\right)_{S}(x_0)
	= \iint k(z,z') \xi(z) \xi(z') \, \dd z \, \dd z',
\end{align*}
where $k(z,z')$ is smooth in both variables and $\iint k(z,z')^2 \,\dd z \, \dd z' <\infty$. A straightforward calculation using the
Cauchy--Schwarz inequality with respect to $\lng\cdot\rng$ and \eqref{eq:v_hoelder_clean} and \eqref{eq:d2Rv_clean} we also get that
$\lng |G(\xi)| \rng \lesssim 1$. Hence, we may apply Lemma~\ref{lem:cylinder} which implies that $G$ is admissible in the spectral gap
inequality \eqref{eq:SG_mult}.  

\smallskip
\noindent\emph{Proof of \ref{claim1}:}
Let $x_0\in\TT^2$ be fixed. We first notice that for every $s, S\in(0,1]$, the derivative of the quadratic functional $G$ is given by
\begin{equs}
 \frac{\partial}{\partial \xi} G(\xi):
 \delta \xi \mapsto \left(\lceil \delta v, (\cdot)_s\rceil (\partial_2R_1 v)_s\right)_{S}(x_0)
 + \left(\lceil v, (\cdot)_s\rceil (\partial_2R_1 \delta v)_s\right)_{S}(x_0),
\end{equs}
where $\delta v$ is the unique solution of zero average in $x_1$ to $\mathcal{L}\delta v = P\delta \xi$ for $\delta \xi\in L^2(\TT^2)$.

\smallskip

\noindent\textsc{Step 1}: We first show that
\begin{align}
 \sup_{\|\delta\xi\|_{L^2} \leq 1} \left| \left(\lceil \delta v, (\cdot)_s\rceil (\partial_2R_1 v)_s\right)_{S}(x_0) \right| 
  &\lesssim \int_{\RR^2} |y_1||\psi_s(y)| \|\Psi_S(x_0-\cdot) (\partial_2R_1v)_s(\cdot-y)\|_{L^\frac{10}{9}} \, \dd y \nonumber
 \\
 & \hskip-5ex + \int_{\RR^2} |y_2| |\psi_s(y)| \|\Psi_S(x_0-\cdot) (\partial_2R_1v)_s(\cdot-y)\|_{L^\frac{10}{7}} \, \dd y,
 \label{eq:der_1}
 \end{align}
 \begin{align}
 \sup_{\|\delta\xi\|_{L^2} \leq 1} \left|\left(\lceil v, (\cdot)_s\rceil (\partial_2R_1 \delta v)_s\right)_{S}(x_0)\right|
 & \lesssim \int_{\RR^2} |\psi_s(y)| \|\Psi_S(x_0-\cdot) \left(v(\cdot-y) - v\right)\|_{L^{\frac{10}{7}}} \, \dd y.
 \label{eq:der_2}
\end{align}

\smallskip

\noindent\textsc{Step 1a} (Proof of \eqref{eq:der_1}): By Fubini and the mean value theorem we have 
\begin{equs}
 & \left(\lceil \delta v, (\cdot)_s\rceil (\partial_2R_1 v)_s\right)_{S}(x_0)
 \\
 &= \int_{\RR^2} \psi_S(x_0-x) \int_{\RR^2} \psi_s(y) \left(\delta v(x) -
  \delta v(x-y)\right) (\partial_2R_1 v)_s(x-y)  \, \dd y \, \dd x
 \\
 &=
 \int_{\RR^2} \psi_s(y) \int_{\TT^2} \Psi_S(x_0-x) \int_0^1 \left( y_1 \partial_1 \delta v(x-ty) + y_2 \partial_2 \delta v(x-ty) \right)\dd t \, (\partial_2R_1 v)_s(x-y) \, \dd x \, \dd y,
\end{equs}
where $\Psi_S$ is the periodization of $\psi_S$. By H\"older's and Minkowski's inequalities, as well as translation invariance, it follows that
\begin{equs}
 &\left|\left(\lceil \delta v, (\cdot)_s\rceil (\partial_2R_1 v)_s\right)_{S}(x_0)\right|
 \\ 
 & \quad \leq 
 \int_{\RR^2} |\psi_s(y)| \|\Psi_S(x_0-x) (\partial_2R_1v)_s(x-y)\|_{L^\frac{10}{9}_x} 
 \left\|\int_0^1 y_1 \partial_1\delta v(x-ty)\,\dd t \right\|_{L^{10}_x} \, \dd y
 \\
 & \quad \quad + 
 \int_{\RR^2} |\psi_s(y)| \|\Psi_S(x_0-x) (\partial_2R_1v)_s(x-y)\|_{L^\frac{10}{7}_x} 
 \left\|\int_0^1 y_2 \partial_2\delta v(x-ty)\,\dd t \right\|_{L^{\frac{10}{3}}_x} \, \dd y
 \\
 & \quad \lesssim 
 \int_{\RR^2} |y_1||\psi_s(y)| \|\Psi_S(x_0-x) (\partial_2R_1v)_s(x-y)\|_{L^\frac{10}{9}_x} \, \dd y
 \, \|\partial_1 \delta v\|_{L^{10}}
 \\
 & \quad \quad + 
 \int_{\RR^2} |y_2| |\psi_s(y)| \|\Psi_S(x_0-x) (\partial_2R_1v)_s(x-y)\|_{L^\frac{10}{7}_x} \, \dd y
 \, \|\partial_2 \delta v\|_{L^{\frac{10}{3}}}.
 %
\end{equs}
By \eqref{eq:sobolev_reg_1} and \eqref{eq:sobolev_reg_3} we have $\|\partial_1 \delta v\|_{L^{10}}, \|\partial_2 \delta v\|_{L^{\frac{10}{3}}} \lesssim \|\delta\xi\|_{L^2}$, hence 
\eqref{eq:der_1} follows after taking the supremum over $\|\delta \xi\|_{L^2}\leq 1$. 

\smallskip

\noindent\textsc{Step 1b} (Proof of \eqref{eq:der_2}): As in \textsc{Step 1A}, we have
\begin{equs}
 & \left(\lceil v, (\cdot)_s\rceil (\partial_2R_1 \delta v)_s\right)_{S}(x_0)
 \\
 & \quad = \int_{\RR^2} \psi_s(y) \int_{\TT^2} \Psi_S(x_0-x) \left(v(x) - v(x-y) \right) (\partial_2R_1 \delta v)_s(x-y)  
 \, \dd x \, \dd y,
\end{equs}
where $\Psi_S$ is the periodization of $\psi_s$. By H\"older's inequality, translation invariance, 
and the fact that $R_1$ is bounded on $L^{\frac{10}{3}}(\TT^2)$, we obtain that
\begin{equs}
 & |\left(\lceil v, (\cdot)_s\rceil (\partial_2R_1 \delta v)_s\right)_{S}(x_0)|
 \\
 & \quad \lesssim 
 \int_{\RR^2} |\psi_s(y)| \|\Psi_S(x_0-x) \left(v(x)-v(x-y)\right)\|_{L^{\frac{10}{7}}_x} 
 \|(\partial_2R_1 \delta v)_s(x-y)\|_{L^\frac{10}{3}_x} \, \dd y
 \\
 & \quad \lesssim 
 \int_{\RR^2} |\psi_s(y)| \|\Psi_S(x_0-x) \left(v(x)-v(x-y)\right)\|_{L^{\frac{10}{7}}_x} \, \dd y
 \, \|\partial_2\delta v\|_{L^\frac{10}{3}}.
%
\end{equs}
By \eqref{eq:sobolev_reg_3}, $\|\partial_2\delta v\|_{L^\frac{10}{3}} \lesssim \|\delta\xi\|_{L^2}$, which gives \eqref{eq:der_2} after taking the supremum over $\|\delta\xi\|_{L^2}\leq 1$. 

\smallskip

\noindent\textsc{Step 2}: For any $q\geq1$ and $x_0 \in \TT^2$, by \textsc{Step 1} and Minkowski's inequality (since $2q \geq \max\{\frac{10}{7}, \frac{10}{9}\}$), we get that
\begin{equs}
\left\lng\left\|\frac{\partial}{\partial \xi}G(\xi)\right\|_{L^2}^{2q}\right\rng^{\frac{1}{2q}}
&\lesssim \int_{\RR^2} |y_1||\psi_s(y)| \|\Psi_S(x_0-x) \lng|(\partial_2R_1v)_s(x-y)|^{2q}\rng^{\frac{1}{2q}}\|_{L^\frac{10}{9}_x} \, \dd y
 \\
 &\quad+ 
 \int_{\RR^2} |y_2| |\psi_s(y)| \|\Psi_S(x_0-x) \lng|(\partial_2R_1v)_s(x-y)|^{2q}\rng^{\frac{1}{2q}}\|_{L^\frac{10}{7}_x} \, \dd y
 \\
 &\quad +
 \int_{\RR^2} |\psi_s(y)| \|\Psi_S(x_0-x) \lng |v(x)-v(x-y)|^{2q} \rng^{\frac{1}{2q}} \|_{L^{\frac{10}{7}}_x} \, \dd y.
\end{equs}
By \eqref{eq:v_hoelder_clean} and \eqref{eq:d2Rv_clean} this implies the bound
\begin{equs}
 \left\lng\left\|\frac{\partial}{\partial \xi}G(\xi)\right\|_{L^2}^{2q}\right\rng^{\frac{1}{2q}}
 &\lesssim \|\Psi_S\|_{L^\frac{10}{9}} \left(s^{\frac{1}{3}}\right)^{-\frac{3}{4}} \int_{\RR^2} |y_1||\psi_s(y)| \, \dd y 
 + \|\Psi_S\|_{L^\frac{10}{7}} \left(s^{\frac{1}{3}}\right)^{-\frac{3}{4}} \int_{\RR^2} |y_2| |\psi_s(y)| \, \dd y 
 \\
 &\quad + \|\Psi_S\|_{L^{\frac{10}{7}}} \int_{\RR^2} d(0,y)^{\frac{3}{4}} |\psi_s(y)| \, \dd y
 \\
 &\lesssim \left(S^{\frac{1}{3}} \right)^{-\frac{1}{4}} \left(s^{\frac{1}{3}}\right)^{1-\frac{3}{4}} 
 + \left(S^{\frac{1}{3}} \right)^{-\frac{3}{4}} \left(s^{\frac{1}{3}}\right)^{\frac{3}{2}-\frac{3}{4}}
 + \left(S^{\frac{1}{3}} \right)^{-\frac{3}{4}} \left(s^{\frac{1}{3}}\right)^{\frac{3}{4}},
\end{equs}
which proves \textsc{Claim 2}, since the implicit constant does not depend on $x_0$.

\smallskip
\noindent\emph{Proof of \ref{claim2}:} By stationarity we know that $\lng (v (\partial_2R_1v)_{2s})_S(x_0) \rng = \lng (v (\partial_2R_1v)_{2s})_S(0) \rng$. 
If we denote by $\tilde v$ the solution to \eqref{eq:linearized_ripple} with $\xi$ replaced by $\tilde\xi(x) = \xi(-x_1,x_2)$, by the symmetry of $\Gamma$
(see \eqref{eq:gamma}) and the 
fact that $\xi$ and $\tilde\xi$ have the same law (see Assumption~\ref{ass}~\ref{item:Def-reflection}), we know that $\tilde v = v$ in law, hence 
$\lng (v (\partial_2R_1v)_{2s})_S(0) \rng = \lng (\tilde v (\partial_2R_1\tilde v)_{2s})_S(0) \rng$. By the symmetry of $\psi_t$ for $t\in(0,1]$ and the fact
that $R_1\tilde v(x) = - R_1 v(-x_1,x_2)$, we get $\lng (\tilde v (\partial_2R_1\tilde v)_{2s})_S(0) \rng = -\lng (v (\partial_2R_1v)_{2s})_S(0) \rng$, which 
implies that $\lng (v (\partial_2R_1v)_{2s})_S(0) \rng=0$. Similarly, $\lng (v(\partial_2R_1v)_{s})_{s+S}(0) \rng =0$.
\end{proof}

We then have the following proposition.

\begin{proposition}\label{prop:F_approx_tight} Under the assumptions of Lemma \ref{lem:commutator}, for every $\eps\in(0,\frac{1}{100})$,
$1\leq p<\infty$, and $0<\tau\leq t\leq 1$ dyadically related there holds
\begin{equs}
 \left\lng [(v \partial_2 R_1 v_\tau)_{t-\tau} - v \partial_2 R_1 v_t]_{-\frac{3}{4}-\eps}^p \right\rng^{\frac{1}{p}}
 \leq \left(t^{\frac{1}{3}}\right)^{\frac{1}{4}}. \label{eq:triple_com_est}
\end{equs}
Furthermore, the following bound holds,
\begin{equs}
 \sup_{t\in(0,1]} \left\lng [v\partial_2 R_1 v_t]_{-\frac{3}{4}-\eps}^p \right\rng^{\frac{1}{p}} \leq C. 
 \label{eq:F_approx_tight}
\end{equs}
In the above estimates, the constant $C$ depends only on $\eps$ and $p$. 
\end{proposition}

\begin{proof} 

\smallskip

\noindent\textsc{Step 1}: We first prove that for every $1\leq p<\infty$, $0<\tau\leq t \leq 1$ (note that by assumption $\tau=\frac{t}{2^n}$ for some $n\geq 0$) 
and $T\in(0,1]$ there holds,
\begin{equs}
 \sup_{x\in\TT^2} \left \lng \left|\left(\lceil v, (\cdot)_{t-\tau}\rceil (\partial_2R_1v)_\tau\right)_{T}(x)\right|^{p} \right \rng^{\frac{1}{p}}
 \leq C \left(\left(T^{\frac{1}{3}}\right)^{-\frac{1}{4}} \left(t^{\frac{1}{3}}\right)^{\frac{1}{4}} + 
 \left(T^{\frac{1}{3}}\right)^{-\frac{3}{4}} \left(t^{\frac{1}{3}}\right)^{\frac{3}{4}}\right). \quad 
 \label{eq:pre_triple_com_est}
\end{equs}
Together with Lemma \ref{lem:mod} this implies \eqref{eq:triple_com_est}. The telescopic sum 
identity
\begin{equs}
 \lceil v, (\cdot)_{t-\tau}\rceil (\partial_2R_1v)_\tau = 
 \sum_{\substack{0\leq k\leq n-1\\s=\frac{t}{2^{k+1}}}} 
 \left(\lceil v, (\cdot)_s\rceil (\partial_2R_1v)_s\right)_{t-2s},
\end{equs}
combined with Lemma \ref{lem:commutator}, yields
\begin{equs}
 \left \lng \left|\left(\lceil v, (\cdot)_{t-\tau}\rceil (\partial_2R_1v)_t\right)_T(x)\right|^{p} \right \rng^{\frac{1}{p}}
 & \lesssim \sum_{\substack{0\leq k\leq n-1\\s=\frac{t}{2^{k+1}}}} 
 \left \lng \left|\left(\lceil v, (\cdot)_{s}\rceil (\partial_2R_1v)_s\right)_{T+t-2s}(x)\right|^{p} \right \rng^{\frac{1}{p}}
 \\
 & \lesssim \sum_{\substack{0\leq k\leq n-1\\s=\frac{t}{2^{k+1}}}} 
 \left(\left((T+t-2s)^{\frac{1}{3}}\right)^{-\frac{1}{4}} \left(s^{\frac{1}{3}}\right)^{\frac{1}{4}}
 + \left((T+t-2s)^{\frac{1}{3}}\right)^{-\frac{3}{4}} \left(s^{\frac{1}{3}}\right)^{\frac{3}{4}} \right)
 \\
 & \lesssim \left(T^{\frac{1}{3}}\right)^{-\frac{1}{4}} 
 \left(t^{\frac{1}{3}}\right)^{\frac{1}{4}} \sum_{0\leq k\leq n-1} 
 \frac{1}{2^{\frac{k+1}{12}}}
 + \left(T^{\frac{1}{3}}\right)^{-\frac{3}{4}} 
 \left(t^{\frac{1}{3}}\right)^{\frac{3}{4}} \sum_{0\leq k\leq n-1} 
 \frac{1}{2^{\frac{k+1}{4}}}
 \\
 & \lesssim 
 \left(T^{\frac{1}{3}}\right)^{-\frac{1}{4}} \left(t^{\frac{1}{3}}\right)^{\frac{1}{4}} + 
 \left(T^{\frac{1}{3}}\right)^{-\frac{3}{4}} \left(t^{\frac{1}{3}}\right)^{\frac{3}{4}},
\end{equs}
which proves the desired claim. 

\smallskip

\noindent\textsc{Step 2}: We now prove that for every $\eps\in(0,\frac{1}{100})$, $1\leq p<\infty$ and $T\in(0,1]$,
\begin{equs}
 \sup_{t\in(0,1]}\sup_{x\in\TT^2} \left\lng |(v\partial_2 R_1 v_t)_T(x)|^p\right\rng^{\frac{1}{p}} 
 \lesssim_{\eps,p} \left(T^{\frac{1}{3}}\right)^{-\frac{3}{4}-\eps}, \label{eq:pre_F_approx_tight}
\end{equs}
which together with Lemma \ref{lem:mod} implies \eqref{eq:F_approx_tight}.

We first assume that $t\in[\frac{T}{2},1]$. Then, by Definition \ref{def:neg_holder} 
of negative H\"older norms and \eqref{eq:R_pos} we know that $[\partial_2 R_1 v]_{-\frac{3}{4}-2\eps} \lesssim [R_1 v]_{\frac{3}{4}-2\eps} \lesssim [v]_{\frac{3}{4}-\eps}$.
Combined with \eqref{eq:neg_hoelder_char} and the fact that $\|v\|_{L^\infty}\lesssim [v]_{\frac{3}{4}-\eps}$, we have that
\begin{equs}
 |(v\partial_2 R_1 v_t)_T(x)| & \leq \|v\partial_2 R_1 v_t\|_{L^\infty} 
 \leq \|v\|_{L^\infty} \|\partial_2 R_1 v_t\|_{L^\infty} 
 \\
 & \lesssim \left(t^{\frac{1}{3}}\right)^{-\frac{3}{4}-2\eps} [v]_{\frac{3}{4}-\eps}^2
 \lesssim \left(T^{\frac{1}{3}}\right)^{-\frac{3}{4}-2\eps} [v]_{\frac{3}{4}-\eps}^2. \label{eq:1st_case_pointwise}
\end{equs}
By \eqref{eq:v_bd}, this implies the estimate
\begin{equs}
 \left\lng |(v\partial_2 R_1 v_t)_T(x)|^p \right\rng_\ell^{\frac{1}{p}} \lesssim
 \left(T^{\frac{1}{3}}\right)^{-\frac{3}{4}-2\eps} \left\lng [v]_{\frac{3}{4}-\eps}^{2p} \right \rng_\ell^{\frac{1}{p}} 
 \lesssim \left(T^{\frac{1}{3}}\right)^{-\frac{3}{4}-2\eps}, \label{eq:1st_case_exp}
\end{equs}
for every $t\in[\frac{T}{2},1]$, uniformly in $x\in \TT^2$ and $\ell\in(0,1]$.

We now assume that $t\in(0,\frac{T}{2}]$. Then, there exists $T_*\in(\frac{T}{4}, \frac{T}{2}]$ such that $t=\frac{T_*}{2^n}$,
for some $n\geq 1$. Using the semigroup property, we write
\begin{equs}
 (v \partial_2 R_1 v_t)_T = \left( (v\partial_2 R_1 v_t)_{T_*-t} - v \partial_2 R_1 v_{T_*} \right)_{T-T_*+t} 
 + \left(v \partial_2 R_1 v_{T_*} \right)_{T-T_*+t}.
\end{equs}
By \eqref{eq:triple_com_est}, the first term can be estimated as
\begin{equs}
 & \left\lng |\left( (v\partial_2 R_1 v_t)_{T_*-t} - v \partial_2 R_1 v_{T_*} \right)_{T-T_*+t}(x)|^p \right\rng^{\frac{1}{p}}
 \\
 & \quad \lesssim 
 \left((T-T_*+t)^{\frac{1}{3}}\right)^{-\frac{3}{4}} \left(T_*^{\frac{1}{3}}\right)^{\frac{3}{4}}
 + 
 \left((T-T_*+t)^{\frac{1}{3}}\right)^{-\frac{1}{4}} \left(T_*^{\frac{1}{3}}\right)^{\frac{1}{4}}
 \lesssim 1,
\end{equs}
where we also used that $T_* \leq \frac{T}{2} \leq T-T_*+t$. For the second term, noting that $T_*>\frac{T}{4}$ and proceeding
as in \eqref{eq:1st_case_pointwise} and \eqref{eq:1st_case_exp}, we obtain the bound
\begin{equs}
 \left\lng |\left(v \partial_2 R_1 v_{T_*} \right)_{T-T_*+t}(x)|^p \right\rng^{\frac{1}{p}} 
 \lesssim \left( T_*^{\frac{1}{3}} \right)^{-\frac{3}{4}-2\eps} 
 \left\lng [v]_{\frac{3}{4}-\eps}^{2p} \right \rng_\ell^{\frac{1}{p}} 
 \lesssim \left( T^{\frac{1}{3}} \right)^{-\frac{3}{4}-2\eps}. 
\end{equs}
Hence, we also proved that
\begin{equs}
 \left\lng |(v \partial_2 R_1 v_t)_T(x)|^p \right\rng^{\frac{1}{p}} \lesssim 
 1+\left( T^{\frac{1}{3}} \right)^{-\frac{3}{4}-2\eps}, \label{eq:2nd_case_exp}
\end{equs}
for $t\in(0,\frac{T}{4}]$, uniformly in $x\in\TT^2$.

Combining \eqref{eq:1st_case_exp} and \eqref{eq:2nd_case_exp} gives \eqref{eq:pre_F_approx_tight} upon relabelling $\eps$.
\end{proof}

As a corollary we obtain,

\begin{corollary} \label{cor:F} Let $\lng\cdot\rng$ satisfy Assumption~\ref{ass}. There exists a unique centered and stationary random variable 
$\xi \mapsto F(\xi)$ such that for every $\eps\in(0,\frac{1}{100})$ and $1\leq p<\infty$,
\begin{equs}
 \lim_{t=2^{-n}\downarrow 0}\left\lng [F - v\partial_2 R_1 v_t]_{-\frac{3}{4}-\eps}^p \right\rng^{\frac{1}{p}} = 0,
 \label{eq:F_uniform_char}  
\end{equs}
and the convergence is uniform in the class of probability measures satisfying Assumption~\ref{ass}. 
\end{corollary}

\begin{proof} We prove that for every $\eps\in(0,1]$, $1\leq p<\infty$ and $0\leq \tau \leq t \leq 1$ dyadic,
\begin{equs}
 \lim_{t,\tau\downarrow0}
 \left\lng [v \partial_2 R_1 v_\tau - v \partial_2 R_1 v_t]_{-\frac{3}{4}-\eps}^p \right\rng^{\frac{1}{p}} =0.
 \label{eq:cauchy}
\end{equs}
By the triangle inequality we have that 
\begin{equs}
 & \left\lng [v \partial_2 R_1 v_\tau - v \partial_2 R_1 v_t]_{-\frac{3}{4}-\eps}^p \right\rng^{\frac{1}{p}}
 \\ 
 & \quad \leq 
 \left\lng [v \partial_2 R_1 v_\tau - (v \partial_2 R_1 v_\tau)_{t-\tau}]_{-\frac{3}{4}-\eps}^p \right\rng^{\frac{1}{p}}
 + \left\lng [(v \partial_2 R_1 v_\tau)_{t-\tau} - v \partial_2 R_1 v_t]_{-\frac{3}{4}-\eps}^p \right\rng^{\frac{1}{p}}. 
\end{equs}
To estimate the first term we use Propositions \ref{prop:sem_est} and \eqref{eq:F_approx_tight} which imply that
\begin{equs}
 \left\lng [v \partial_2 R_1 v_\tau - (v \partial_2 R_1 v_\tau)_{t-\tau}]_{-\frac{3}{4}-\eps}^p \right\rng^{\frac{1}{p}} \lesssim 
 \left((t-\tau)^{\frac{1}{3}}\right)^{\frac{\eps}{2}} \left\lng [v \partial_2 R_1 v_\tau]_{-\frac{3}{4}-\frac{\eps}{2}}^p \right\rng^{\frac{1}{p}} 
 \lesssim \left((t-\tau)^{\frac{1}{3}}\right)^{\frac{\eps}{2}}.
\end{equs}
Using \eqref{eq:triple_com_est}, the second term is estimated as 
\begin{equs}
 \left\lng [(v \partial_2 R_1 v_\tau)_{t-\tau} - v \partial_2 R_1 v_t]_{-\frac{3}{4}-\eps}^p \right\rng^{\frac{1}{p}}
 \lesssim \left(t^{\frac{1}{3}}\right)^{\frac{1}{4}}. 
\end{equs}
Hence, we have proved that
\begin{equs}
 \left\lng [v \partial_2 R_1 v_\tau - v \partial_2 R_1 v_t]_{-\frac{3}{4}-\eps}^p \right\rng^{\frac{1}{p}} 
 \lesssim \left((t-\tau)^{\frac{1}{3}}\right)^{\frac{\eps}{2}} + \left(t^{\frac{1}{3}}\right)^{\frac{1}{4}},
\end{equs}
which implies \eqref{eq:cauchy} after taking $\tau,t\downarrow0$. This implies that the sequence $\{v \partial_2 R_1 v_t\}_{t\in(0,1]}$
is Cauchy in $L^p_{\lng\cdot\rng}\C^{-\frac{3}{4}-\eps}$, hence it converges to a limit $F$ in 
$L^p_{\lng\cdot\rng}\C^{-\frac{3}{4}-\eps}$. The fact that the convergence is uniform in the class of probability 
measures satisfying 
Assumption~\ref{ass} follows since our estimates depend on $\lng\cdot\rng$ only through the spectral gap inequality \eqref{eq:SG}.

Similarly to the proof of Lemma~\ref{lem:commutator}, we can show that $\lng v\partial_2 R_1 v_t(x) \rng = 0$ for every $x\in \TT^2$ and $t\in(0,1]$. 
Hence, as a limit in $L^p_{\lng\cdot\rng}\C^{-\frac{3}{4}-\eps}$, the fact that $F$ is centered and stationary follows from the corresponding 
properties of $v\partial_2 R_1 v_t$. 
\end{proof}

\subsection{Proof of Proposition \ref{Satz1_pavlos}}\label{sec:lifted_measure}

\begin{proof} 
%

Let $\lng\cdot\rng^{\mathrm{lift}}$ be given by $\lng G(\xi,v,F) \rng^{\mathrm{lift}}:= \lng G(\xi,v,F) \rng$ for every bounded and continuous 
$G:\C^{-\frac{5}{4}-\eps} \times \C^{\frac{3}{4}-\eps}\times \C^{-\frac{3}{4}-\eps}\to \RR$, with the convention that under $\lng\cdot\rng$, 
$v=\mathcal{L}^{-1}P\xi$ and $F$ is given by \eqref{eq:F_uniform_char} (under $\lng\cdot\rng^{\mathrm{lift}}$ we think of $(\xi,v,F)$ as a dummy 
variable in $\triple:=\C^{-\frac{5}{4}-\eps} \times \C^{\frac{3}{4}-\eps}\times \C^{-\frac{3}{4}-\eps}$). The fact that $\lng\cdot\rng^{\mathrm{lift}}$ defines a probability measure on $\triple$ is immediate by 
Corollaries~\ref{cor:xi_v} and \ref{cor:F}.

Statements \ref{it:Satz1_pavlos_xi} and \ref{it:Satz1_pavlos_v} are immediate by the construction of $\lng\cdot\rng^{\mathrm{lift}}$. The first
part of statement \ref{it:Satz1_pavlos_F}, that is, \eqref{eq:F_pavlos}, is immediate by Corollary \ref{cor:F}. For the second part, note that in 
the case when $\xi$ is smooth $\lng\cdot\rng$-almost surely, the product $v\partial_2R_1 v$ makes sense $\lng\cdot\rng$-almost surely and we also 
have that $v\partial_2 R_1 v_t \to v\partial_2R_1 v$ $\lng\cdot\rng$-almost surely. By \eqref{eq:F_uniform_char} we know that $v\partial_2 R_1 v_t\to F$
in $\C^{-\frac{3}{4}-\eps}$ $\lng\cdot\rng$-almost surely along a subsequence. Hence, we should have that $F = v\partial_2R_1 v$ $\lng\cdot\rng$-almost
surely.

It remains to prove the continuity statement \ref{it:Satz1_pavlos_cont}. Assume that $\{\lng\cdot\rng_\ell\}_{\ell\downarrow0}$ converges weakly to $\lng \cdot \rng$
and consider $\{\lng\cdot\rng_\ell^{\mathrm{lift}}\}_{\ell\downarrow0}$ and $\lng \cdot \rng^{\mathrm{lift}}$. 

\smallskip

\noindent\textsc{Step 1}: We prove that weak convergence of $\{\lng\cdot\rng_\ell\}_{\ell\downarrow0}$ to $\lng \cdot \rng$ in the Schwartz 
topology implies weak convergence in $\C^{-\frac{5}{4}-\eps}$. Indeed, let $G:\C^{-\frac{5}{4}-\eps} \to \RR$ be bounded and continuous. Then, 
we can write
\begin{equs}
 \lng G(\xi) \rng_\ell - \lng G(\xi) \rng = (\lng G(\xi) \rng_\ell - \lng G(\xi_t) \rng_\ell) + (\lng G(\xi_t) \rng_\ell - \lng G(\xi_t) \rng)
 + (\lng G(\xi_t) \rng - \lng G(\xi) \rng).  
\end{equs}
To treat the first term, we use \eqref{eq:sem_est} and \eqref{eq:xi_bd} (which holds uniform in the class 
of probability measures satisfying Assumption~\ref{ass}) yielding for every $1\leq p<\infty$,
\begin{equs}
 \sup_{\ell\in(0,1]} \lng [\xi -\xi_t]_{-\frac{5}{4}-\eps}^p \rng_\ell^{\frac{1}{p}} 
 & \lesssim t^{\frac{\eps}{6}} \sup_{\ell\in(0,1]} \lng [\xi]_{-\frac{5}{4}-\frac{\eps}{2}}^p \rng_\ell^{\frac{1}{p}} \lesssim t^{\frac{\eps}{6}}.
 \label{eq:xi_t_conv}
\end{equs}
For $\delta>0$ which we fix below, we write
\begin{equs}
 \lng G(\xi) \rng_\ell - \lng G(\xi_t) \rng_\ell = \lng (G(\xi) - G(\xi_t)) \mathbf{1}_{\{[\xi -\xi_t]_{-\frac{5}{4}-\eps} < \delta\}} \rng_\ell
 + \lng (G(\xi) - G(\xi_t)) \mathbf{1}_{\{[\xi -\xi_t]_{-\frac{5}{4}-\eps} \geq \delta\}} \rng_\ell.
\end{equs}
For $\eta>0$, by the continuity of $G$ we can choose $\delta$ sufficiently small such that 
\begin{equs}
 |\lng (G(\xi) - G(\xi_t)) \mathbf{1}_{\{[\xi -\xi_t]_{-\frac{5}{4}-\eps} < \delta\}} \rng_\ell| \leq 
 \frac{\eta}{4} 
\end{equs}
uniformly in $\ell\in(0,1]$. By \eqref{eq:xi_t_conv}, the boundedness of $G$, and Chebyshev's inequality, we can choose $t\in(0,1]$ sufficiently
small such that
\begin{equs}
 |\lng (G(\xi) - G(\xi_t)) \mathbf{1}_{\{[\xi -\xi_t]_{-\frac{5}{4}-\eps} \leq \delta\}} \rng_\ell|\leq 
 \|G\|_{L^\infty\left(\C^{-\frac{5}{4}-\eps}\right)} \sup_{\ell\in(0,1]}\lng \mathbf{1}_{\{[\xi -\xi_t]_{-\frac{5}{4}-\eps} \geq \delta\}} \rng_\ell 
 \leq \frac{\eta}{4}. 
\end{equs}
Hence, we obtain that for $t$ sufficiently small 
\begin{equs}
 \sup_{\ell\in(0,1]}|\lng G(\xi) \rng_\ell - \lng G(\xi_t) \rng_\ell| \leq \frac{\eta}{2}.
\end{equs}
Similarly, we can show that for $t$ sufficiently small
\begin{equs}
 |\lng G(\xi) \rng - \lng G(\xi_t) \rng| \leq \frac{\eta}{2}.
\end{equs}
Since $\{\lng\cdot\rng_\ell\}_{\ell\downarrow0}$ converges to $\lng \cdot \rng$ in the Schwartz topology, for every $t\in(0,1]$ we
know that $\lng G(\xi_t) \rng_\ell \to \lng G(\xi_t) \rng$ as $\ell\downarrow0$. In total, we have that 
\begin{equs}
 \limsup_{\ell\downarrow0} |\lng G(\xi) \rng_\ell - \lng G(\xi) \rng| \leq \eta,
\end{equs}
which proves that $\lng G(\xi) \rng_\ell\to \lng G(\xi) \rng$ as $\ell\downarrow0$ since $\eta$ is arbitrary.

\smallskip

\noindent\textsc{Step 2}: We now prove that $\{\lng\cdot\rng_\ell^{\mathrm{lift}}\}_{\ell\downarrow0}$ converges weakly to 
$\lng \cdot \rng^{\mathrm{lift}}$ in $\triple$. The argument is similar in spirit to \textsc{Step 1}. Let $G: \triple \to \RR$ be a 
bounded continuous function. Then, we write
\begin{equs}
 \lng G(\xi,v,F) \rng_\ell^{\mathrm{lift}} - \lng G(\xi,v,F) \rng^{\mathrm{lift}} & =
 \lng G(\xi, v, F) - G(\xi, v, v \partial_2 R_1 v_t) \rng_\ell 
 \\
 &  \quad +  \lng G(\xi, v, v\partial_2 R_1 v_t) -  G(\xi, v, F)\rng
 \\
 & \quad + \lng G(\xi, v, v \partial_2 R_1 v_t) \rng_\ell - \lng G(\xi, v, v\partial_2 R_1 v_t) \rng. 
\end{equs}
To estimate the first term, for $\delta>0$ to be fixed below, we use the decomposition
\begin{equs}
 \lng G(\xi, v, F) - G(\xi, v, v \partial_2 R_1 v_t) \rng_\ell & = 
 \lng (G(\xi, v, F) - G(\xi, v, v \partial_2 R_1 v_t)) \mathbf{1}_{\{[F- v\partial_2 R_1 v_t]_{-\frac{3}{4}-\eps}<\delta\}} \rng_\ell 
 \\
 & \quad + \lng (G(\xi, v, F) - G(\xi, v, v \partial_2 R_1 v_t)) \mathbf{1}_{\{[F- v\partial_2 R_1 v_t]_{-\frac{3}{4}-\eps}\geq\delta\}} \rng_\ell. 
\end{equs}
For $\eta >0$, by the continuity of $G$ we can choose $\delta>0$ sufficiently small such that
\begin{equs}
 |\lng (G(\xi, v, F) - G(\xi, v, v \partial_2 R_1 v_t)) \mathbf{1}_{\{[F- v\partial_2 R_1 v_t]_{-\frac{3}{4}-\eps}<\delta\}} \rng_\ell| 
 \leq \frac{\eta}{4}
\end{equs}
uniformly in $\ell\in(0,1]$. By \eqref{eq:F_uniform_char}, the boundedness of $G$, and Chebyshev's inequality, we can choose $t$ sufficiently small such that
\begin{equs}
 & \sup_{\ell\in(0,1]} \lng (G(\xi, v, F) - G(\xi, v, v \partial_2 R_1 v_t)) \mathbf{1}_{\{[F- v\partial_2 R_1 v_t]_{-\frac{3}{4}-\eps}\geq\delta\}} \rng_\ell
 \\
 & \quad \leq \|G\|_{L^\infty\left(\C^{-\frac{5}{4}-\eps} \times \C^{\frac{3}{4}-\eps}\times \C^{-\frac{3}{4}-\eps}\right)} 
 \sup_{\ell\in(0,1]} \lng \mathbf{1}_{\{[F- v\partial_2 R_1 v_t]_{-\frac{3}{4}-\eps}\geq\delta\}} \rng_\ell \leq \frac{\eta}{4}.
\end{equs}
Hence, we have proved that for $t$ sufficiently small
\begin{equs}
 \sup_{\ell\in(0,1]} \lng G(\xi, v, F) - G(\xi, v, v \partial_2 R_1 v_t) \rng_\ell \leq \frac{\eta}{2}. 
\end{equs}
In a similar way we can show that for $t$ sufficiently small
\begin{equs}
 \lng G(\xi, v, F) - G(\xi, v, v \partial_2 R_1 v_t) \rng \leq \frac{\eta}{2}. 
\end{equs}
Since $\lng \cdot \rng_\ell \to \lng \cdot \rng$ weakly as $\ell\downarrow0$, we have that $\lng G(\xi, v, v \partial_2 R_1 v_t) \rng_\ell \to \lng G(\xi, v, v\partial_2 R_1 v_t) \rng$ as $\ell\downarrow0$
for every $t\in(0,1]$. Altogether, we get that
\begin{equs}
 \lim_{\ell\downarrow0} (\lng G(\xi,v,F) \rng_\ell^{\mathrm{lift}} - \lng G(\xi,v,F) \rng^{\mathrm{lift}}) \leq \eta,
\end{equs}
which in turn implies that $\lng G(\xi,v,F) \rng_\ell^{\mathrm{lift}}\to \lng G(\xi,v,F) \rng^{\mathrm{lift}}$ as $\ell\downarrow0$ since $\eta$ is arbitrary. Thus $\{\lng\cdot\rng_\ell^{\mathrm{lift}}\}_{\ell\downarrow0}$ 
converges weakly to $\lng \cdot \rng^{\mathrm{lift}}$. 
\end{proof}

%% file: app-hoelder.tex
\section{H\"older spaces} \label{app:hoelder_spaces}

The following equivalent characterization of H\"older norms relies on the ``heat kernel'' of the operator $\mathcal{A}$.

\begin{lemma}[{\cite[Lemma 10, Remark 1]{IO19}}] \label{lem:hoelder_char} 
Let $f$ be a periodic distribution on $\TT^2$.
\begin{enumerate}
\item For $\beta\in(-\frac{3}{2},0)\setminus \{-1,-\frac{1}{2}\}$, we have
\begin{align}\label{eq:neg_hoelder_char}
[f]_\beta\sim\sup_{T\in(0,1]} \left(T^\frac{1}{3}\right)^{-\beta}\|f_T\|_{L^\infty}.
\end{align}
In the critical cases $\beta\in \{-1, -\frac12\}$ we have
\begin{equation} \label{eq:supl_numa}
\sup_{T\in (0,1]} \left(T^\frac{1}{3}\right)^{-\beta} \|f_T\|_{L^\infty} \lesssim [f]_\beta.
\end{equation}
\item For $\beta\in (-\frac{3}{2},0)\setminus \{-1,-\frac{1}{2}\}$ and $f$ of vanishing average, we have
\begin{align}\label{eq:neg_hoelder_A_char}
[f]_\beta \sim\sup_{T\in(0,1]} \left(T^\frac{1}{3}\right)^{-\beta}\|T\mathcal{A}f_T\|_{L^\infty}.
\end{align}
In the critical cases $\beta\in \{-1, -\frac12\}$ we have\footnote{Indeed, using \eqref{eq:supl_numa} and 
\eqref{lp_psi} we have 
\begin{align*}
	\|\mathcal{A} f_T\|_{L^{\infty}}
	=\|f_{\frac{T}{2}}*\mathcal{A}\psi_{\frac{T}{2}}\|_{L^{\infty}}
	\leq \|f_{\frac{T}{2}}\|_{L^{\infty}} \|\mathcal{A}\psi_{\frac{T}{2}}\|_{L^1}
	\lesssim (T^{\frac{1}{3}})^{\beta}[f]_\beta \frac{1}{T}.
\end{align*}
}
\begin{equation} \label{eq:supl_in_A}
\sup_{T\in (0,1]} \left(T^\frac{1}{3}\right)^{-\beta} \|T\mathcal{A} f_T\|_{L^\infty} \lesssim [f]_\beta.
\end{equation}
\item For $\alpha\in (0,\frac{3}{2})\setminus \{1\}$ we have
\begin{align}\label{eq:pos_hoelder_A_char}
[f]_\alpha \sim\sup_{T\in(0,1]} \left(T^{\frac{1}{3}}\right)^{-\alpha}\|T\mathcal{A}f_T\|_{L^\infty}.
\end{align}
In the critical case $\alpha=1$ we have\footnote{Indeed, since $\mathcal{A}\psi_T$ has vanishing average, we write 
$\mathcal{A}f_T(x)=\int_{\RR^2} \mathcal A \psi_T(y)(f(x-y)-f(x)) \, \dd y$ and deduce via Step 2 in the proof of Lemma 10 in 
\cite{IO19} that 
\begin{equs}
	\|\mathcal{A}f_T\|_{L^{\infty}} 
	\leq [f]_\alpha \int_{\RR^2} |\mathcal{A} \psi_T(y)| d(y,0)^\alpha \, \dd y
	\lesssim [f]_\alpha \left(T^{\frac{1}{3}}\right)^{(-3+\alpha)}.
\end{equs}
}
\begin{equation}
\label{alpha=1}
\sup_{T\in(0,1]} \left(T^{\frac{1}{3}}\right)^{-1}\|T\mathcal{A}f_T\|_{L^\infty}\lesssim [f]_1.
\end{equation}
\end{enumerate}
In the case of periodic distributions $f$ \emph{of vanishing average on $\TT^2$}, one can consider the supremum over all $T>0$ in 
\eqref{eq:neg_hoelder_char}, \eqref{eq:supl_numa}, \eqref{eq:neg_hoelder_A_char}, and \eqref{eq:supl_in_A},
while in \eqref{eq:pos_hoelder_A_char} and \eqref{alpha=1} the suprema over $T\in(0,1)$ and $T>0$ are equivalent even for distributions of 
nonvanishing average. 
\end{lemma}

We also have the following interpolation inequality.
\begin{lemma} \label{lem:interpolation}
For every $-\frac{3}{2}<\beta<0 < \gamma < \alpha<\frac{3}{2}$ 
there exists a constant $C>0$ such that the following interpolation inequality holds for every $f:\TT^2 \to \RR$,
\begin{align} \label{eq:interpolation-hoelder}
	[f]_\gamma \leq C [f]_{\beta}^\lambda [f]_{\alpha}^{1-\lambda},
\end{align}
where $\lambda\in(0,1)$ is given by $\gamma = \lambda \beta + (1-\lambda) \alpha$. If $f$ has vanishing average in $\TT^2$, 
\eqref{eq:interpolation-hoelder} also holds for $\gamma =0$ with $[f]_\gamma$ replaced by $\|f\|_{L^\infty}$.   
\end{lemma}

\begin{proof}
	By \eqref{eq:neg_hoelder_char} and \eqref{eq:supl_numa}, we have for every 
	$T\in(0,1]$ and $x,y\in \TT^2$:
	\begin{equation}
	\label{00}
		\begin{cases}
		& |f_{2T} (x) - f_T(x)|  \leq  \|f_{T}\|_{L^{\infty}}+ \|f_{2T}\|_{L^{\infty}} \lesssim [f]_\beta \left(T^{\frac{1}{3}}\right)^\beta, \\
		& |\partial_1 f_{2T} (x) - \partial_1 f_T(x)|  \leq (\|f_{\frac{3T}{2}}\|_{L^{\infty}}+ \|f_{\frac{T}{2}}\|_{L^{\infty}}) 
		\int_{\RR^2} |\partial_1 \psi_{\frac{T}{2}}(z)| \, \dd z \lesssim [f]_\beta \left(T^{\frac{1}{3}}\right)^{\beta-1}, \\
		& |f_T(x) - f_T(y)|  \leq 2 \|f_{T}\|_{L^{\infty}} \lesssim [f]_{\beta} \left(T^{\frac{1}{3}}\right)^\beta,\\
		& |\partial_1 f_T(y)(x_1-y_1)|\leq \|\partial_1 \psi_{\frac{T}{2}}\|_{L^{1}(\RR^2)} 
		\|f_{\frac{T}{2}}\|_{L^{\infty}} d(x,y) \lesssim  [f]_{\beta} \left(T^{\frac{1}{3}}\right)^\beta \frac{d(x,y)}{T^{\frac{1}{3}}},\\
		& |f_T(x) - f_T(y)-\partial_1 f_T(y)(x_1-y_1)|  \lesssim [f]_{\beta} \left(T^{\frac{1}{3}}\right)^\beta 
		\max\left(1, \frac{d(x,y)}{T^{\frac{1}{3}}}\right),
		\end{cases}
	\end{equation}
	where we used \cite[equation (26)]{IO19}. In the case $\alpha\in (0,1]$, we claim that for every $T\in(0,1]$ and $x,y\in \TT^2$:
	\begin{equation}
		\label{11}
		|f_{2T} (x) - f_T(x)|  \lesssim [f]_\alpha \left(T^{\frac{1}{3}}\right)^\alpha \quad \textrm{and} \quad		|f_T(x) - f_T(y)|  \lesssim [f]_{\alpha} d^{\alpha}(x,y). 
	\end{equation}
	Indeed, by \cite[equation (26)]{IO19}, we deduce 
	\begin{equs}
	 |f_T(x) - f_T(y)|
	 &\leq \int_{\RR^2} |\psi_T(z)| |f(x-z)-f(y-z)| \, \dd z
	 \leq [f]_\alpha d(x,y)^\alpha \int_{\RR^2} |\psi_T(z)| \, \dd z \\
	 &\lesssim [f]_{\alpha} d^{\alpha}(x,y),
	 \\
	 |f_{2T} (x) - f_T(x)| 
	 &\leq \int_{\RR^2} |\psi_T(z)| |f_T(x-z)-f_T(x)| \, \dd z
	 \lesssim [f]_{\alpha} \int_{\RR^2} |\psi_T(z)| d(z,0)^\alpha\, \dd z \\
	 &\lesssim [f]_\alpha \left(T^{\frac{1}{3}}\right)^\alpha.
	\end{equs}
	In the case $\alpha\in (1, \frac32)$, arguing as above, we also have for every $T\in(0,1]$ and $x,y\in \TT^2$:
	\begin{equation} \label{22}
			\begin{cases}
			& |f_T(x) - f_T(y)-\partial_1 f_T(y)(x_1-y_1)|  \lesssim [f]_{\alpha} d^{\alpha}(x,y),\\
			& |f_{2T} (x) - f_T(x)|\leq \int_{\RR^2} |\psi_T(z)| |f_T(x-z)-f_T(x)+\partial_1 f_T(x)z_1|\, \dd z
			\lesssim [f]_\alpha \left(T^{\frac{1}{3}}\right)^\alpha,
			\\
			& |\partial_1 f_{2T} (x) - \partial_1 f_T(x)|  \leq 
			\int_{\RR^2} |\partial_1 \psi_{\frac{T}{2}}(z)| |f_{\frac{3T}{2}}(x-z)-f_{\frac{T}{2}}(x-z)| \, \dd z
			\lesssim [f]_\alpha \left(T^{\frac{1}{3}}\right)^{\alpha-1},
			\end{cases}
	\end{equation}
	where we used $\int_{\RR^2} \psi_T \, \dd z=1$ and $\int_{\RR^2} z_1\psi_T(z) \, \dd z=0$. To prove \eqref{eq:interpolation-hoelder} we distinguish three different cases for $\gamma$.
	
	\smallskip
	
	\noindent\textsc{Case} $\gamma\in (0,1)$: First, assume that $\alpha\in (0,1]$. Interpolating \eqref{00} and
	\eqref{11}, we obtain for every $T\in (0,1]$ and $x,y\in \TT^2$:
	\begin{align*}
		|f_{2T} (x) - f_T(x)| & \lesssim [f]_\beta^{\lambda} [f]_\alpha^{1-\lambda} \left(T^{\frac{1}{3}}\right)^{\gamma} \\
		|f_T(x) - f_T(y)| & \lesssim [f]_\beta^{\lambda} [f]_\alpha^{1-\lambda} \left(T^{\frac{1}{3}}\right)^{\lambda\beta} d(x,y)^{(1-\lambda)\alpha}.
	\end{align*}
	If $n_0\in\ZZ$ is the largest integer such that  $2^{-\frac{n_0}{3}} \geq d(x,y)$,
	then for every $n\geq n_0$ 
	\begin{equs}
		&|f_{2^{-n}} (x) - f_{2^{-n}}(y)| 
		\\
		& \quad \leq |f_{2^{-n}} (x) - f_{2^{-n_0}}(x)| + |f_{2^{-n_0}} (x) - f_{2^{-n_0}}(y)| 
		+ |f_{2^{-n}} (y) - f_{2^{-n_0}}(y)|
		\\
		& \quad \leq \sum_{k=n_0}^{n-1} |f_{2^{-(k+1)}} (x) - f_{2^{-k}}(x)| + |f_{2^{-n_0}} (x) - f_{2^{-n_0}}(y)| 
		 + \sum_{k=n_0}^{n-1} |f_{2^{-(k+1)}} (y) - f_{2^{-k}}(y)|
		 \\
		& \quad \lesssim [f]_\beta^{\lambda} [f]_\alpha^{1-\lambda} \left( \sum_{k=n_0}^{n-1} \left(2^{-\frac{k}{3}}\right)^{\gamma} + 
		\left(2^{-\frac{n_0}{3}}\right)^{\lambda\beta} d(x,y)^{(1-\lambda)\alpha}\right) 
		\\
		& \quad \lesssim  [f]_\beta^{\lambda} [f]_\alpha^{1-\lambda} \left( \left(2^{-\frac{n_0}{3}}\right)^{\gamma} 
		+ \left(2^{-\frac{n_0}{3}}\right)^{\lambda\beta} d(x,y)^{(1-\lambda)\alpha}\right) \lesssim  [f]_\beta^{\lambda} [f]_\alpha^{1-\lambda} d(x,y)^\gamma,
		\label{eq:alpha_small}
	\end{equs}
	which in turn implies \eqref{eq:interpolation-hoelder} by letting $n\to\infty$. 
	
	In the case $\alpha\in (1, \frac32)$, first, one needs to choose $n_0\in\ZZ$ such that $2^{-\frac{n_0}{3}} \geq \kappa d(x,y)>2^{-\frac{n_0+1}{3}}$ with a constant 
	$\kappa>0$, depending only on $\gamma$, which we fix below. Second, one needs to estimate
	the intermediate term $|f_{2^{-n_0}} (x) - f_{2^{-n_0}}(y)|$ differently.
	For this, we need to use that for every $T\in(0,1]$, by \eqref{eq:neg_hoelder_char}, \eqref{eq:supl_numa}, Definition \ref{def:neg_holder}, 
	and since $\gamma<1$,
	\begin{equs}
	 \|\partial_1 f_T\|_{L^{\infty}}  \lesssim [\partial_1 f]_{\gamma-1} (T^{\frac{1}{3}})^{\gamma-1}\leq c_0  
	 [f]_{\gamma} (T^{\frac{1}{3}})^{\gamma-1}
	\end{equs}
	where $c_0>0$ depends only on $\gamma$. By interpolating between \eqref{00} and \eqref{22}, we estimate the intermediate term  
	for a constant $C>0$ (depending on $\alpha$, $\beta$, $\gamma$) by, 
	\begin{align*}
	& |f_{2^{-n_0}} (x) - f_{2^{-n_0}}(y)| 
	\\
	& \quad \leq  
	|f_{2^{-n_0}} (x) - f_{2^{-n_0}}(y)-\partial_1 f_{2^{-n_0}}(y)(x_1-y_1)| + |\partial_1 f_{2^{-n_0}}(y)(x_1-y_1)|
	\\
	& \quad \leq C [f]_\beta^{\lambda} [f]_\alpha^{1-\lambda} \left(2^{-\frac{n_0}{3}}\right)^{\lambda\beta} 
	d(x,y)^{(1-\lambda)\alpha} \max\left(1, \frac{d(x,y)}{2^{-\frac{n_0}{3}}}\right)^{\lambda} + c_0  [f]_{\gamma} (2^{-\frac{n_0}{3}})^{\gamma-1} d(x,y)
	\\
	& \quad \leq C_{\kappa} [f]_\beta^{\lambda} [f]_\alpha^{1-\lambda} d(x,y)^\gamma+ c_0  [f]_{\gamma} \kappa^{\gamma-1} d(x,y)^\gamma.
	\end{align*}
	Choosing $\kappa>0$ such that $c_0 \kappa^{\gamma-1}=\frac{1}{2}$ and proceeding as in \eqref{eq:alpha_small} (where now the change of $n_0$ 
	affects the implicit constant by a factor depending on $\kappa$), after passing to the limit $n\to \infty$ we obtain 
	\begin{equs}
	 |f(x)-f(y)|\leq \left (C  [f]_\beta^{\lambda} [f]_\alpha^{1-\lambda} +\frac12 [f]_{\gamma}\right) d(x,y)^\gamma.
	\end{equs}
	Dividing by $d(x,y)^\gamma$ and taking the supremum over $x\neq y$ we finally obtain \eqref{eq:interpolation-hoelder}.
	
	\smallskip
	
	\noindent\textsc{Case} $\gamma\in (1, \frac{3}{2})$: Since $\alpha>\gamma$ we also have $\alpha\in (1, \frac{3}{2})$. 
	Interpolating \eqref{00} and \eqref{22}, we obtain for every $T\in (0,1]$ and $x,y\in \TT^2$,
	\begin{align*}
		|f_{2T} (x) - f_T(x)| & \lesssim [f]_\beta^{\lambda} [f]_\alpha^{1-\lambda} \left(T^{\frac{1}{3}}\right)^{\gamma},
		\\
		|\partial_1f_{2T} (x) - \partial_1 f_T(x)| & \lesssim [f]_\beta^{\lambda} [f]_\alpha^{1-\lambda} \left(T^{\frac{1}{3}}\right)^{\gamma-1},
		\\
		|f_T(x) - f_T(y)-\partial_1 f_T(y)(x_1-y_1)| & \lesssim [f]_\beta^{\lambda} [f]_\alpha^{1-\lambda} 
		\left(T^{\frac{1}{3}}\right)^{\lambda\beta} d(x,y)^{(1-\lambda)\alpha} 
		\max\left(1, \frac{d(x,y)}{T^{\frac{1}{3}}}\right)^{\lambda}.
	\end{align*}
	If $n_0\in\ZZ$ is the largest integer such that  $2^{-\frac{n_0}{3}} \geq d(x,y)$, the same argument as in the previous case
	yields for $n\geq n_0$,
	\begin{equs}
		& |f_{2^{-n}} (x) - f_{2^{-n}}(y)-\partial_1  f_{2^{-n}}(y)(x_1-y_1) |
		\\
		& \quad \leq  |f_{2^{-n_0}} (x) - f_{2^{-n_0}}(y)-\partial_1  f_{2^{-n_0}}(y)(x_1-y_1)|
		+ |f_{2^{-n}} (x) - f_{2^{-n_0}}(x)| 
		\\
		& \quad \quad + |f_{2^{-n}} (y) - f_{2^{-n_0}}(y)| 
		+ |\partial_1 f_{2^{-n}} (y) - \partial_1  f_{2^{-n_0}}(y)| |x_1-y_1| \\
		& \quad \lesssim [f]_\beta^{\lambda} [f]_\alpha^{1-\lambda} 
		\left(\left(2^{-\frac{n_0}{3}}\right)^{\lambda\beta} d(x,y)^{(1-\lambda)\alpha}
		+ \sum_{k=n_0}^{n-1} \left(2^{-\frac{k}{3}}\right)^{\gamma}
		+ d(x,y) \sum_{k=n_0}^{n-1} \left(2^{-\frac{k}{3}}\right)^{\gamma-1}\right) 
		\\
		& \quad \lesssim  [f]_\beta^{\lambda} [f]_\alpha^{1-\lambda} d(x,y)^\gamma,
	\end{equs}
	which yields \eqref{eq:interpolation-hoelder} by letting $n\to\infty$. 

\smallskip
\noindent\textsc{Case} $\gamma=1$: Take a sequence $(\gamma_n)\subset(1, \frac{3}{2})$ such that $\gamma_n \searrow 1$, and consider the corresponding exponents $\lambda_n \to \lambda$. Then by \cite[Remark 2]{IO19} and the previous case, we have 
\begin{align*}
	[f]_1 \lesssim [f]_{\gamma_n} \lesssim [f]_{\beta}^{\lambda_n} [f]_{\alpha}^{1-\lambda_n},
\end{align*}
with implicit constants independent of $n$. We can therefore perform the limit $n\to\infty$ to conclude the estimate for $\gamma=1$.

\smallskip
\noindent\textsc{Case} $\gamma =0$ and $f$ has vanishing average: If $[f]_\alpha=0$, then $f\equiv 0$,\footnote{This is clear for $\alpha\in (0,1)$ since $f$ has vanishing average. For $\alpha\in(1,\frac{3}{2})$ notice that by \cite[Lemma 12]{IO19}, $\|\partial_1 f\|_{L^{\infty}} = 0$, hence by Definition \ref{def:pos_holder} $f$ is constant, and this constant is $0$ since $f$ has vanishing average.} so \eqref{eq:interpolation-hoelder} holds trivially. Assume that $[f]_\alpha\neq 0$. If $\alpha\in (0,1]$,
	we have for every $T>0$, 
	\begin{align*}
		|f(x)| & \leq |f(x) - f_T(x)| + |f_T(x)| 
		\lesssim  \int_{\RR^2} |\psi_T(z)| |f(x-z)-f(x)| \, \dd z 
		+ \left(T^{\frac{1}{3}}\right)^\beta [f]_\beta 
		\\
		& \lesssim \left(T^{\frac{1}{3}}\right)^\alpha [f]_\alpha + \left(T^{\frac{1}{3}}\right)^\beta [f]_\beta, 
	\end{align*}
	while if $\alpha\in (1, \frac{3}{2})$,
	\begin{align*}
		|f(x)| & \leq |f(x) - f_T(x)| + |f_T(x)|
		\lesssim  \int_{\RR^2} |\psi_T(z)| |f(x-z)-f(x)+\partial_1 f(x) z_1| \, \dd z
		+ \left(T^{\frac{1}{3}}\right)^\beta [f]_\beta
		\\
		& \quad \lesssim \left(T^{\frac{1}{3}}\right)^\alpha [f]_\alpha + \left(T^{\frac{1}{3}}\right)^\beta [f]_\beta, 
	\end{align*}
	where we used $\int_{\RR^2} z_1 \psi_T(z)\, \dd z=0$ and Lemma \ref{lem:hoelder_char} in the case of distributions with vanishing average, as $T$ can be larger than $1$.
	Choosing $T^{\frac{1}{3}} = \left(\frac{[f]_{\beta}}{[f]_{\alpha}}\right)^{\frac{1}{\alpha-\beta}}$ leads to the conclusion.
\end{proof}

\begin{remark} \label{rem:partial_1_interpolation} One can also prove that for $-\frac{1}{2} < \beta <1< \alpha < \frac{3}{2}$ 
the interpolation estimate
\begin{equs}
 \|\partial_1 f\|_{L^\infty} & \lesssim [f]_\beta^\lambda [f]_\alpha^{1-\lambda}\label{eq:partial_1_interpolation}
\end{equs}
holds for $\lambda \in(0,1)$ given by $1 = \lambda \beta + (1-\lambda) \alpha$. Indeed, by Lemma \ref{lem:interpolation} (the case $\gamma=0$) we know that
\begin{equs}
 \|\partial_1 f\|_{L^\infty} & \lesssim [\partial_1 f]_{\beta-1}^\lambda [\partial_1 f]_{\alpha-1}^{1-\lambda}. 
\end{equs}
By Definition \ref{def:neg_holder}, we have $[\partial_1 f]_{\beta-1} \lesssim [f]_{\beta}$ and by \cite[Lemma 12]{IO19}, we know $[\partial_1 f]_{\alpha-1} \lesssim [f]_{\alpha}$, so the desired estimate follows. 
\end{remark}

We also need the following lemma for the Hilbert transform acting on H\"older spaces.

\begin{lemma} \label{lem:R} \hfill
\begin{enumerate}
 \item For $\beta\in(-\frac{3}{2},0)$ and $\eps>0$ such that $\beta-\eps\in (-\frac{3}{2},0)$, there
 exists a constant $C>0$ such that for every periodic distribution $f$,
 \begin{equs}
  {[}R_1f{]}_{\beta-\varepsilon} \leq C [f]_\beta. \label{eq:R_neg}
 \end{equs}
 \item For $\alpha \in (0,\frac{3}{2})$ and $\eps>0$ such that $\alpha-\eps\in(0,\frac{3}{2})$ there exists a constant $C>0$ such that for every $f:\TT^2\to \RR$,
 \begin{equs}
  {[}R_1f{]}_{\alpha-\varepsilon} \leq C [f]_\alpha. \label{eq:R_pos}
 \end{equs}
\end{enumerate}
\end{lemma}

\begin{proof} To prove \eqref{eq:R_neg}, we claim that for every $\varepsilon\in (0,\frac{3}{2})$ and $T\in (0,1]$, we have that
\begin{equs}
 \|R_1f_T\|_{L^\infty} \lesssim \left(T^{\frac{1}{3}}\right)^{-\varepsilon} \|f_{\frac{T}{2}}\|_{L^\infty}.
\end{equs}
Indeed, if we write $\Psi_T$ for the periodization of $\psi_T$, by the semigroup property and Young's inequality for convolution
we have for $p=\frac{5}{2\varepsilon}>1$,
\begin{equs}
 \|R_1f_T\|_{L^\infty} \lesssim \|f_{\frac{T}{2}}\|_{L^p} \|R_1\Psi_{\frac{T}{2}}\|_{L^{\frac{p}{p-1}}} 
 \lesssim \|f_{\frac{T}{2}}\|_{L^p} \|\Psi_{\frac{T}{2}}\|_{L^{\frac{p}{p-1}}}
 \lesssim \|f_{\frac{T}{2}}\|_{L^\infty} \left(T^{\frac{1}{3}}\right)^{-\frac{5}{2p}}
\end{equs}
where we used that the Hilbert transform is bounded\footnote{This follows from the fact that the Hilbert transform $R_1$ is
bounded on $L^q(\RR^2)$ for any $q\in (1,\infty)$ and the transference of multipliers method \cite[Theorem 4.3.7]{Gra14}.}
on $L^{\frac{p}{p-1}}(\TT^2)$ for $\frac{p}{p-1}>1$ and the bound $\|\Psi_{\frac{T}{2}}\|_{L^{\frac{p}{p-1}}}\lesssim (T^{\frac{1}{3}})^{-\frac{5}{2p}}$, which follows from Remark~\ref{rem:perio}. Then \eqref{eq:R_neg} follows via the characterization of negative H\"older norms \eqref{eq:neg_hoelder_char} and 
\eqref{eq:supl_numa} if $\beta - \epsilon \neq -1, -\frac{1}{2}$. In those cases, consider $\gamma\in(\beta-\epsilon, \beta)$ and use the previous case together with \cite[Remark 2]{IO19}.

Equation \eqref{eq:R_pos} is essentially \cite[Lemma 7]{IO19}, noting that we can assume that $f$ is of vanishing average,
as $R_1$ annihilates constants and $[f-\int_{\TT^2} f \, \dd x]_\alpha\leq [f]_\alpha$.
\end{proof}

\begin{lemma}\label{lem:frac-deriv-hoelder}
	Let $\alpha\in(-\frac{3}{2}, \frac{3}{2}) \setminus \{0\}$ and $s>0$ such that $\alpha-s \in (-\frac{3}{2}, \frac{3}{2}) \setminus \{-1, -\frac{1}{2}, 0, 1\}$. 
	There exists a constant $C>0$ such that for every periodic $f\in \mathcal{C}^{\alpha}$,
	\begin{align*}
		\left[|\partial_1|^s f\right]_{\alpha-s} \leq C [f]_{\alpha}.
	\end{align*}
\end{lemma}

\begin{proof}
	Without loss of generality, we may assume that $f$ is of vanishing average because $[f-\int_{\TT^2} f \, \dd x]_\alpha\leq [f]_\alpha$ and $\left[|\partial_1|^s f\right]_{\alpha-s}$ is invariant by adding a constant to $f$. Then by the semigroup property,
	$|\partial_1|^{s} f_T = \left( \frac{T}{2} \right)^{-\frac{s}{3}} f_{\frac{T}{2}} * (|\partial_1|^{s}\psi)_{\frac{T}{2}}$ for every $T\in (0,1]$ and since $|\partial_1|^{s}\psi\in L^1(\RR^2)$, we deduce
	$\|T \mathcal{A} |\partial_1|^s f_T\|_{L^{\infty}} \lesssim \left( T^{\frac{1}{3}} \right)^{-s} \|T \mathcal{A} f_{\frac{T}{2}} \|_{L^{\infty}}$.
	Hence, we obtain that
	\begin{align*}
		\left( T^{\frac{1}{3}} \right)^{-(\alpha-s)} \|T\mathcal{A} |\partial_1|^s f_T \|_{L^{\infty}}
		\lesssim \left(T^{\frac{1}{3}}\right)^{-\alpha} \|T \mathcal{A} f_{\frac{T}{2}} \|_{L^{\infty}},
	\end{align*}
	and the conclusion follows by Lemma \ref{lem:hoelder_char}.
\end{proof}

\begin{lemma} \label{lem:comp_emb} Let $\alpha\in(-\frac{3}{2}, \frac{3}{2}) \setminus \{0\}$. For every sequence $\{f_n\}_{n\geq 1}\subset \C^\alpha$
with $\sup_{n\geq 1} [f_n]_\alpha <\infty$, there exists $f\in \C^\alpha$ with 
\begin{equs}
 {} [f]_\alpha \leq \liminf_{n\to\infty} [f_n]_\alpha,
\end{equs}
such that $f_n\to f$ in $\C^{\alpha-\eps}$ for every $\eps>0$ along a subsequence. In particular, the embedding 
$\C^{\alpha}\hookrightarrow \C^{\alpha-\eps}$ is compact.
\end{lemma}

\begin{proof} First, assume that $\alpha-\eps\neq -1, -\frac{1}{2}, 0, 1$.
Let $\{f_n\}_{n\geq 1}\subset \C^\alpha$ such that
\begin{equs}
 K := \sup_{n\geq 1} [f_n]_\alpha <\infty, 
\end{equs}
and, if $\alpha \in(0,\frac{3}{2})$, we assume in addition that $\{\|f_n\|_{L^{\infty}}\}_{n\geq1}$ is uniformly bounded (hence, up to subtraction, we may assume that $f_n$ has zero average for any $n\in\NN$).

By \cite[Lemma 13]{IO19} there exists $f\in \C^\alpha$ such that $f_n\to f$ (along a subsequence) in the sense of distributions and  
\begin{equs}
 {} [f]_\alpha \leq \liminf_{n\to\infty} [f_n]_\alpha \leq K.  
\end{equs}
For $T,t> 0$ we have that
\begin{equs}
 \|(\mathcal{A}f-\mathcal{A}f_n)_T\|_{L^\infty} & \leq \|(\mathcal{A}f-\mathcal{A}f_t)_T\|_{L^\infty} + \|(\mathcal{A}f-\mathcal{A}f_n)_{t+T}\|_{L^\infty} 
 \\
 & \quad + \|((\mathcal{A}f_n)_t - \mathcal{A}f_n)_T\|_{L^\infty}. \label{eq:triangle_f}
\end{equs}
By \eqref{eq:neg_hoelder_A_char}, \eqref{eq:pos_hoelder_A_char}, and \eqref{eq:sem_est} below the following estimates hold 
\begin{equs}
 \|(\mathcal{A}f-\mathcal{A}f_t)_T\|_{L^\infty} & \lesssim \left(T^{\frac{1}{3}}\right)^{\alpha-\eps-3} [f-f_t]_{\alpha-\eps} \lesssim K \left(T^{\frac{1}{3}}\right)^{\alpha-\eps-3} \left(t^{\frac{1}{3}}\right)^{\eps}, 
 \\
 \|((\mathcal{A}f_n)_t - \mathcal{A}f_n)_T\|_{L^\infty} & \lesssim \left(T^{\frac{1}{3}}\right)^{\alpha-\eps-3} [(f_n)_t - f_n]_{\alpha-\eps} \lesssim K \left(T^{\frac{1}{3}}\right)^{\alpha-\eps-3} \left(t^{\frac{1}{3}}\right)^{\eps}.
 \end{equs}
 By Young's inequality for convolution and \eqref{lp_psi} we further have that
 \begin{equs}
 \|(\mathcal{A}f-\mathcal{A}f_n)_{t+T}\|_{L^\infty} & \lesssim \left(T^{\frac{1}{3}}\right)^{\alpha-\eps-3} 
 \|(\mathcal{A}f-\mathcal{A}f_n)_t\|_{L^{\frac{5}{2(3-\alpha+\eps)}}}.
\end{equs}
Since $f_n\to f$ in the sense of distributions, $(\mathcal{A}f-\mathcal{A}f_n)_t\to 0$ as $n\to\infty$  pointwise for every $t\in(0,1]$, and 
by \eqref{eq:neg_hoelder_A_char} and \eqref{eq:pos_hoelder_A_char} we know that
\begin{equs}
 \|(\mathcal{A}f-\mathcal{A}f_n)_t\|_{L^\infty} \lesssim \left(t^{\frac{1}{3}}\right)^{\alpha-3} 
 [f-f_n]_\alpha \lesssim K \left(t^{\frac{1}{3}}\right)^{\alpha-3}.
\end{equs}
Hence, by dominated convergence theorem, $\|(\mathcal{A}f-\mathcal{A}f_n)_t\|_{L^{\frac{5}{2(3-\alpha+\eps)}}}\to0$ as $n\to\infty$ for every $t\in(0,1]$. Taking $n\to\infty$
in \eqref{eq:triangle_f} and using again \eqref{eq:neg_hoelder_A_char} and \eqref{eq:pos_hoelder_A_char} we obtain that
\begin{equs}
 \limsup_{n\to\infty} [f-f_n]_{\alpha-\eps} \lesssim K \left(t^{\frac{1}{3}}\right)^{\eps},
\end{equs}
which completes the proof if we let $t\to 0$.

If $\alpha-\eps \in \{-1, -\frac{1}{2}, 0, 1\}$, consider $\gamma\in (\alpha-\eps, \alpha)$; in view of \cite[Remark 2]{IO19} and the above result we then have $[f-f_n]_{\alpha-\eps} \lesssim [f-f_n]_{\gamma} \to 0$ as $n\to \infty$.
\end{proof}

\begin{proposition} \label{prop:sem_est} For every $\alpha\in(-\frac{3}{2}, \frac{3}{2}) \setminus \{0\}$ and $\eps>0$ such that 
$\alpha-\eps \in (-\frac{3}{2}, \frac{3}{2}) \setminus \{-1, -\frac{1}{2}, 0, 1\}$ the following estimate holds:
\begin{equs}
 {} [f-f_t]_{\alpha-\eps} \lesssim \left(t^{\frac{1}{3}}\right)^{\eps} [f]_\alpha. \label{eq:sem_est}
\end{equs}
%
\end{proposition}

\begin{proof} To prove \eqref{eq:sem_est} we use the definition of $\psi_t$ and the semigroup property to estimate for
$t, T\in (0,1]$
\begin{equs}
 \|\mathcal{A}\left(f - f_t\right)_T \|_{L^\infty} 
 & \leq \int_T^{t+T} \|\mathcal{A}\partial_s f_s\|_{L^\infty}  \, \dd s 
 = \int_T^{t+T} \|\mathcal{A}^2 f_s\|_{L^\infty} \, \dd s \\
 &= \int_T^{t+T} \|\mathcal{A}\psi_{\frac{s}{2}} * \mathcal{A}f_{\frac{s}{2}}\|_{L^\infty} \, \dd s
 \leq \int_T^{t+T} \|\mathcal{A}\psi_{\frac{s}{2}}\|_{L^1} \|\mathcal{A}f_{\frac{s}{2}}\|_{L^\infty} \, \dd s.
 \end{equs}
 Since $\|\mathcal{A}\psi_{\frac{s}{2}}\|_{L^1} \lesssim s^{-1}$, 
 \eqref{eq:neg_hoelder_A_char} and \eqref{eq:pos_hoelder_A_char} imply that 
 \begin{equs}
 \|\mathcal{A}\left(f - f_t\right)_T \|_{L^\infty} 
 & \lesssim [f]_{\alpha} \int_T^{t+T} (s^{\frac{1}{3}})^{-3+\alpha} \,\frac{\dd s }{s}
 \lesssim [f]_{\alpha} \left( T^{\frac{1}{3}} \right)^{-3 + \alpha-\eps} \int_T^{t+T} (s^{\frac{1}{3}})^{\eps} \,\frac{\dd s }{s} \\
 & \lesssim [f]_{\alpha}  \left(T^{\frac{1}{3}}\right)^{-3+\alpha-\eps} \left(t^{\frac{1}{3}}\right)^{\eps},
 \end{equs}
so that \eqref{eq:sem_est} follows from \eqref{eq:neg_hoelder_A_char} and \eqref{eq:pos_hoelder_A_char}.
\end{proof}

%% file: app-besov.tex
\hyphenation{an-iso-tropic}
\section{Besov spaces} \label{s:besov_spaces}

In the next lemma we summarize some useful properties of the Besov seminorms that we often use in this article.

\begin{lemma}\label{lem:classic} Let $0<s<s'\leq 1$, $ 1\leq p\leq q<\infty$, and $j\in\{1,2\}$. The following estimates hold
	\begin{align} 
		\|f\|_{\dotB^{s'}_{p;1}} & \leq [f]_{s'}\label{eq:besov_to_holder_1}\, , \\
		\|f\|_{\dotB^s_{p;2}} & \leq [f]_{\frac{3}{2}s}  \label{eq:besov_to_holder_2}\, , \\
		\|f\|_{\dotB^{s}_{p;j}}  \leq \|f\|_{\dotB^{s'}_{p;j}} \quad \text{and}\quad \|f\|_{\dotB^{s'}_{p;j}} 
		& \leq \|f\|_{\dotB^{s'}_{q;j}} \label{eq:besov_s_p_q}\, , \\
		\||\partial_j|^sf\|_{L^p} & \leq C(s,s') \|f\|_{\dotB^{s'}_{p;j}}\, , \label{eq:sobolev_to_besov} 
	\end{align}
	for every function $f:\TT^2 \to \RR$, where the constant $C(s,s')>0$ depends only on $s$ and $s'$.
\end{lemma}

\begin{proof} 
	Estimates \eqref{eq:besov_to_holder_1} and \eqref{eq:besov_to_holder_2} are immediate from Definitions 
	\ref{def:pos_holder} and \ref{def:besov_s_p}. Both estimates in \eqref{eq:besov_s_p_q} follow from Definition \ref{def:besov_s_p} and Jensen's inequality.
	To prove \eqref{eq:sobolev_to_besov} we first notice that by a simple calculation of the Fourier coefficients 
	we have the identity 
	\begin{align*}
		|\partial_j|^s f(x) = k_s \int_\RR \frac{\partial_j^hf(x)}{|h|^{s+1}} \, \dd h, \quad x \in \TT^2,
	\end{align*}
	for some constant $k_s>0$ which depends only on $s$, where we interpret the integral as a principle value. Then by Minkowski's 
	inequality we get
	\begin{align*}
		\||\partial_j|^s f\|_{L^p} & 
		\leq k_s \int_\RR \left(\int_{\TT^2} \frac{|\partial_j^hf(x)|^p}{|h|^{(s+1)p}} \, \dd x \right)^{\frac{1}{p}} \, \dd h
		\\ 
		& \leq k_s \int_{|h|\leq 1} \left(\int_{\TT^2} \frac{|\partial_j^hf(x)|^p}{|h|^{(s+1)p}} \, \dd x \right)^{\frac{1}{p}} \, \dd h 
		+ k_s \int_{|h|\geq 1} \left(\int_{\TT^2} \frac{|\partial_j^hf(x)|^p}{|h|^{(s+1)p}} \, \dd x \right)^{\frac{1}{p}} \, \dd h.
	\end{align*}
	The first term in the last inequality is estimated by
	\begin{align*}
		\int_{|h|\leq 1} \left(\int_{\TT^2} \frac{|\partial_j^hf(x)|^p}{|h|^{(s+1)p}} \, \dd x \right)^{\frac{1}{p}} \, \dd h
		& =2 \int_0^1 \left(\int_{\TT^2} \frac{|\partial_j^hf(x)|^p}{h^{(s+1)p}} \, \dd x \right)^{\frac{1}{p}} \, \dd h \\
		& \leq 2 \int_0^1 \frac{h^{s'-s}}{h} \, \dd h \, \|f\|_{\dotB^{s'}_{p;j}} = \frac2{s'-s} \|f\|_{\dotB^{s'}_{p;j}},
	\end{align*}
	where we also used translation invariance of the torus and that $s'>s$. The second term is estimated by Minkowski's inequality
	\begin{align*}
		\int_{|h|\geq 1} \left(\int_{\TT^2} \frac{|\partial_j^hf(x)|^p}{|h|^{(s+1)p}} \, \dd x \right)^{\frac{1}{p}} \, \dd h
		& \leq 4 \int_1^\infty \frac{1}{h^{s+1}} \, \dd h \, \|f\|_{L^p}=\frac4s \|f\|_{L^p},
	\end{align*}
	where we used again translation invariance of the torus.  If $f$ has vanishing average
	in $x_j$, then $\|f\|_{L^p}\leq \|f\|_{\dotB^{s'}_{p;j}}$. Otherwise  $\|f- \int_0^1 f \, \dd x_j\|_{L^p} \leq \|f\|_{\dotB^{s'}_{p;j}}$
	and we can replace $f$ by $f - \int_0^1 f \, \dd x_j$. 
	This completes the proof. 
\end{proof}

\begin{lemma} \label{lem:leibniz} \hfill
	\begin{enumerate}[label=(\roman*)]
	\item \label{item:leibniz-1} For every $s\in (0,1)$, $p\geq 1$ and $f,g: \TT^2 \to \RR$ we have
	\begin{align*}
		\|fg\|_{\dotB^s_{p;1}} & \leq \|f\|_{\dotB^s_{p;1}} \|g\|_{L^\infty} + \|f\|_{L^p} [g]_{s}, \\
		\|fg\|_{\dotB^s_{p;2}} & \leq  \|f\|_{\dotB^s_{p;2}} \|g\|_{L^\infty} + \|f\|_{L^p} [g]_{\frac{3}{2}s}.
	\end{align*}
	\item \label{item:leibniz-2} For every $s\in(0,1)$, there exists a constant $C(s)>0$ depending only on $s$ such that
	for every $\varepsilon\in (0,1-s)$ and $f,g: \TT^2 \to \RR$ we have
	\begin{equs}
		\||\partial_1|^s(fg)\|_{L^2} & \leq 2\||\partial_1|^sf\|_{L^2} \|g\|_{L^\infty} +\frac{C(s)}{\sqrt\eps} \|f\|_{L^2} [g]_{s+\varepsilon}, \\
		\||\partial_2|^s(fg)\|_{L^2} & \leq 2 \||\partial_2|^sf\|_{L^2} \|g\|_{L^\infty} +\frac{C(s)}{\sqrt\eps} \|f\|_{L^2} [g]_{\frac{3}{2}(s+\varepsilon)}.
	\end{equs}
	\end{enumerate}
\end{lemma}

\begin{proof}
	\begin{enumerate}[label=(\roman*)]
	\item \label{item:leibniz-proof-1} For $j\in \{1,2\}$ we have that 
	\begin{align*}
		\partial_j^h(fg)(x) = (\partial_j^h f)(x) g(x+h e_j) + f(x) (\partial_j^h g)(x),
	\end{align*}
	hence by Minkowski's inequality, $\|\partial_j^h (fg)\|_{L^p} \leq \|(\partial_j^h f) g(\cdot+h e_j)\|_{L^p} + \|f(\partial_j^h g)\|_{L^p}.$
	It follows that $\|fg\|_{\dotB^s_{p;j}} \leq \|f\|_{\dotB^s_{p;j}} \|g\|_{L^\infty} + \|f\|_{L^p} [g]_{s_j}$ with $s_1=s$ and $s_2=\frac32s$.
	
	\item By \eqref{eq:fourier} we know that 
	\begin{equs}
		\||\partial_j|^2(fg)\|^2_{L^2} = c_s \int_\RR \frac{1}{|h|^{2s}} \int_{\TT^2} |\partial_j^h(fg)(x)|^2 \, \dd x \frac{\dd h}{|h|}.
	\end{equs}
	Similarly to \ref{item:leibniz-proof-1} we can prove that 
	\begin{equs}
		& c_s \int_\RR \frac{1}{|h|^{2s}} \int_{\TT^2} |\partial_j^h(fg)(x)|^2 \, \dd x \frac{\dd h}{|h|} 
		\\
		& \quad \leq 2 c_s \int_\RR \frac{1}{|h|^{2s}} \int_{\TT^2} |\partial_j^hf(x)|^2 \, \dd x \frac{\dd h}{|h|} \|g\|^2_{L^\infty}
		+ 2 c_s \|f\|^2_{L^2} \int_\RR \frac{1}{|h|^{2s}} \sup_{x\in\TT^2}|\partial_j^h g(x)|^2 \frac{\dd h}{|h|}
		\\
		& \quad =2 \||\partial_j|^s f\|_{L^2}^2 \|g\|^2_{L^\infty}
		+  2 c_s\|f\|^2_{L^2} \int_\RR \frac{1}{|h|^{2s}} \sup_{x\in\TT^2}|\partial_j^h g(x)|^2 \frac{\dd h}{|h|},
	\end{equs}
	where in the last step we used again \eqref{eq:fourier}. Using the periodicity of $g$ and the fact that for $h>1$
	we can write $h=h_{\mathrm{fr}}+h_{\mathrm{int}}$ with $h_{\mathrm{fr}}\in (0,1]$ and $h_{\mathrm{int}}\in \ZZ$, we notice that
	$|\partial_j^h g(x)|=|\partial_j^{h_{\mathrm{fr}}} g(x)|\lesssim [g]_{s_j}h_{\mathrm{fr}}^{s_j}\lesssim [g]_{s_j}$ with 
	$s_j\in\{s+\eps, \frac{3}{2}(s+\eps)\}$. Then the result follows from  
	\begin{equs}
		\left(\int_0^1+\int_1^\infty\right) \frac{1}{h^{2s}} \sup_{x\in\TT^2}|\partial_j^h g(x)|^2 \frac{\dd h}{h}
		& \lesssim \bigg(\frac1\eps+\int_1^\infty \frac{\dd h}{h^{1+2s}}\bigg) [g]^2_{s_j}.
	\end{equs}
	\end{enumerate}
\end{proof}

%

\begin{lemma} \label{lem:dual_ineq}
	For every $s\in (0,1)$ and $\gamma \in(s,1]$, there exists a constant $C(s,\gamma)>0$ depending only on $s$ and $\gamma$
	such that the following duality estimate holds for every $f,g:\TT^2 \to \RR$,
	\begin{align*}
		\left | \int_{\TT^2} f g  \, \dd x \right| \leq C(s,\gamma) 
		\left(\||\partial_1|^\gamma f\|_{L^1} + \||\partial_2|^{\frac{2}{3}\gamma} f\|_{L^1} 
		+ \|f\|_{L^1}\right) [g]_{-s}.
	\end{align*}
	\end{lemma}

\begin{proof} 
	By the mean value theorem and the definition of the kernel $\psi_T$ we have that
	\begin{align*}
		\int_{\TT^2} f (g - g_1)  \, \dd x & = \int_0^\frac{1}{2} \int_{\TT^2} f \, \partial_Tg_{2T} \, \dd x \, \dd T
		= \int_0^\frac{1}{2} \int_{\TT^2} f \, \left(|\partial_1|^3 - \partial_2^2\right)g_{2T} \, \dd x \,\dd T.
	\end{align*}
	Let $\psi^{(1)}  = |\partial_1|^{3-\gamma}\psi$ and $\psi^{(2)} = |\partial_2|^{2-\frac{2}{3}\gamma}\psi$. Recalling that 
	$|\partial_1|^\alpha \psi, |\partial_2|^\alpha \psi\in L^1(\RR^2)$ for every $\alpha\geq 0$,\footnote{\label{footnote:tensorization}As explained in 
	\cite[Proof of Lemma 10]{IO19}, the kernel $\psi$ factorizes
	into a Gaussian in $x_2$ and a kernel 
	$\varphi(x_1)$ that is smooth and $|x_1|^3\partial^k_1 \varphi\in L^\infty(\RR)$ for every $k\geq 0$ because in Fourier space $\xi_1^k e^{-|\xi_1|^3}$ has integrable derivatives up to order $3$.}
	integrating by parts and using the semigroup property we obtain by \eqref{eq:neg_hoelder_char} and Remark~\ref{rem:perio},
	\begin{align*}
		\left |\int_0^\frac{1}{2} \int_{\TT^2} f \, \left(|\partial_1|^3 - \partial_2^2\right)g_{2T} \, \dd x \,\dd T \right| & = \left|\int_0^\frac{1}{2} \int_{\TT^2} \left(|\partial_1|^3 - \partial_2^2\right)f_T \, g_T \, \dd x \,\dd T\right| \\
		& \lesssim_s \int_0^{\frac{1}{2}} \int_{\TT^2} T^{-\frac{3-\gamma}{3}} \left||\partial_1|^\gamma f * \psi^{(1)}_T\right| \, [g]_{-s} T^{-\frac{s}{3}} \, \dd x \, \dd T \\
		& \quad + \int_0^{\frac{1}{2}} \int_{\TT^2} T^{-\frac{3-\gamma}{3}} \left||\partial_2|^{\frac{2}{3}\gamma}f*\psi^{(2)}_T\right| \, [g]_{-s} T^{-\frac{s}{3}} \, \dd x \, \dd T \\
		& \lesssim_s  \left(\||\partial_1|^\gamma f\|_{L^1} + \||\partial_2|^{\frac{2}{3}\gamma} f\|_{L^1}\right) [g]_{-s} \int_0^{\frac{1}{2}} T^{-\frac{3-\gamma+s}{3}} \, \dd T.
	\end{align*}
	The proof is complete since the last integral is finite for $\gamma>s$ and by \eqref{eq:neg_hoelder_char} we also have the estimate
	\begin{align*}
		\left|\int_{\TT^2} f g_1 \dd x \right| \leq \|f\|_{L^1} \|g_1\|_{L^\infty} 
		\lesssim_s \|f\|_{L^1} [g]_{-s}.
	\end{align*}
\end{proof}

The following Lemma establishes an optimal Sobolev embedding with respect to our anisotropic metric. Recall that in our context the (scaling) dimension 
of the space is $\mathrm{dim}=\frac{5}{2}$, and  $\partial_2$ costs as much as $\frac{3}{2}$ of $\partial_1$. Therefore, the critical exponent of the 
embedding $H^1_{\mathrm{anisotropic}} \subset L^{2^*}$ is given by $2^* = \frac{2\,\mathrm{dim}}{\mathrm{dim}-2} = 10$. 

\begin{lemma} \label{lem:an_embedding} There exists a constant $C>0$ such that for any function $f:\TT^2 \to \RR$ of vanishing average in $x_1$,
\begin{align}
 \|f\|_{L^{10}} & \leq C\left( \|\partial_1 f\|_{L^2} + \||\partial_2|^{\frac{2}{3}} f \|_{L^2}\right), \label{eq:H^1_an_emb}
 \\
 \|f\|_{L^{\frac{10}{3}}} & \leq C\left( \||\partial_1|^{\frac{1}{2}} f\|_{L^2} + \||\partial_2|^{\frac{1}{3}} f \|_{L^2}\right),
 \label{eq:H^1/2_an_emb}
 \\
 \|f\|_{L^{\frac{5}{1-2\eps}}} & \leq C \left( \||\partial_1|^{\frac{3}{4}+\epsilon} f\|_{L^2} 
 + \||\partial_2|^{\frac{2}{3}(\frac{3}{4}+\epsilon)} f \|_{L^2} \right), \quad \eps\in[0,\tfrac{1}{4}).
 \label{eq:H^3/4_an_emb}
\end{align}
\end{lemma}

\begin{proof} We first prove \eqref{eq:H^1_an_emb}. For that, by the one-dimensional Gagliardo-Nirenberg-Sobolev inequalities, using that
$f(\cdot, x_2)$ has vanishing average, we have that for every $x_1,x_2\in [0,1)$,
\begin{align*}
	\|f(\cdot, x_2)\|_{L^{\infty}_{x_1}} & \lesssim \|\partial_1 f(\cdot, x_2) \|_{L^2_{x_1}}^{\frac{1}{6}} \|f(\cdot, x_2)\|_{L^{10}_{x_1}}^{\frac{5}{6}}, 
	\\
	\|f(x_1,\cdot)\|_{L^{\infty}_{x_2}} & \lesssim \||\partial_2|^{\frac{2}{3}} f(x_1,\cdot)\|_{L^2_{x_2}}^{\frac{3}{8}} \|f(x_1,\cdot)\|_{L^{10}_{x_2}}^{\frac{5}{8}}
	+ \|f(x_1,\cdot)\|_{L^1_{x_2}}.
\end{align*}
Hence, H\"older's inequality implies 
\begin{align*}
	\left\| \|f\|_{L^{\infty}_{x_1}} \right\|_{L^6_{x_2}} & \lesssim \|\partial_1 f \|_{L^2}^{\frac{1}{6}} \|f\|_{L^{10}}^{\frac{5}{6}}, 
	\\ 
	\left\| \|f\|_{L^{\infty}_{x_2}} \right\|_{L^4_{x_1}} & \lesssim \||\partial_2|^{\frac{2}{3}} f \|_{L^2}^{\frac{3}{8}} \|f\|_{L^{10}}^{\frac{5}{8}} + \|f\|_{L^4}\lesssim \||\partial_2|^{\frac{2}{3}} f \|_{L^2}^{\frac{3}{8}} \|f\|_{L^{10}}^{\frac{5}{8}} + \|f\|_{L^{10}}. 
\end{align*}
Therefore, we obtain  
\begin{align*}
	\|f\|_{L^{10}} 
	&\leq \left\| \|f\|_{L^{\infty}_{x_1}} \right\|_{L^6_{x_2}}^{\frac{3}{5}} \left\| \|f\|_{L^{\infty}_{x_2}} \right\|_{L^4_{x_1}}^{\frac{2}{5}} 
	\lesssim \|\partial_1 f \|_{L^2}^{\frac{1}{10}} \| |\partial_2|^{\frac{2}{3}} f\|_{L^2}^{\frac{3}{20}} \|f\|_{L^{10}}^{\frac{3}{4}} 
	+ \|\partial_1 f\|_{L^2}^{\frac{1}{10}} \|f\|_{L^{10}}^{\frac9{10}}.
\end{align*}
It follows by Young's inequality that for any $\varepsilon \in (0,1)$ there exists a constant $C(\varepsilon)>0$ such that 
\begin{align*}
	\|f\|_{L^{10}}
	&\leq \varepsilon \|f\|_{L^{10}} + C(\varepsilon) \left( \|\partial_1 f \|_{L^2} + \| |\partial_2|^{\frac{2}{3}} f\|_{L^2} \right),
\end{align*}
from which \eqref{eq:H^1_an_emb} follows. 

For the second inequality observe that $\frac{3}{10} = \frac{1}{2}\cdot \frac{1}{10} + \frac{1}{2} \cdot \frac{1}{2}$, so that
\eqref{eq:H^1/2_an_emb} follows from complex interpolation (with change of measure). Indeed, by inequality \eqref{eq:H^1_an_emb}
in the form 
\begin{align*}
	\|f\|_{L^{10}}^2 \lesssim  \sum_{k\in (2\pi \ZZ)^2} (|k_1|^2 + |k_2|^{\frac{4}{3}}) |\widehat{f}(k)|^2,
\end{align*}
it 
follows that $\mathrm{Id}: L^2((2\pi\ZZ)^2,w \,\dd\chi) \to L^{10}(\TT^2,\dd x)$ is bounded, where $w(k):=|k_1|^2 + |k_2|^{\frac{4}{3}}$ 
and $\chi$ is the counting measure, while Parseval's identity $\|f\|_{L^2}^2 = \sum_{k\in(2\pi\ZZ)^2} |\widehat{f}(k)|^2$
shows that $\mathrm{Id}: L^2((2\pi\ZZ)^2,\dd\chi) \to L^2(\TT^2, \dd x)$ is bounded. Hence, by interpolation (see \cite[Corollary 5.5.4]{BL76}), it follows that $\mathrm{Id}: L^2((2\pi\ZZ)^2, w^\frac{1}{2}\,\dd\chi) \to L^{\frac{10}{3}}(\TT^2,\,\dd x)$ is bounded,
which implies \eqref{eq:H^1/2_an_emb}. The bound \eqref{eq:H^3/4_an_emb} follows similarly by interpolation.
\end{proof}

\begin{lemma}\label{lem:frac-sobolev}
There exists a constant $C>0$ such that for every $s_1\in[0,1]$ and $s_2\in[0,\frac{2}{3}]$ with $s_1+\frac{3}{2}s_2 \leq 1$ and every
periodic function $f:\TT^2\to \RR$ of vanishing average in $x_1$ the following holds 
\begin{equs}
 \||\partial_1|^{s_1} |\partial_2|^{s_2} f\|_{L^2}^2 \leq C \mathcal{H}(f).
\end{equs}
In particular, any sublevel set of $\mathcal{E}$ (respectively $\HH$) over $\mathcal{W}$ is relatively compact in $L^2$.
\end{lemma}

\begin{proof} 
The desired estimate is immediate by H\"older's inequality in Fourier space and \eqref{eq:bound-2/3}.
To prove that the sublevel sets of $\mathcal{E}$ (respectively $\HH$) over $\mathcal{W}$ are relatively compact, we notice that
$\mathcal{H}(f)$ controls $\||\partial_1|^{\frac{1}{2}}f\|_{L^2}$ and $\||\partial_2|^{\frac{1}{2}}f\|_{L^2}$, hence also $\|f\|_{L^2}$ since $f$ has vanishing average. By the compact embedding 
$H^{\frac{1}{2}}\subset L^2$ and \eqref{eq:bound-harmonic}, we deduce that any sublevel set of $\mathcal{E}$ (respectively $\HH$) over $\mathcal{W}$ is relatively compact in $L^2$.
\end{proof}

Below we use the following notation for $s>0$ and a periodic function $f:\TT^2 \to \RR$ with vanishing average in $x_1$, 
	\begin{equs}
		\Hnorm{s}{f}^2 := \int_{\TT^2} \left( |\partial_1|^s f \right)^2\,\dd x 
		+ \int_{\TT^2} \left( |\partial_1|^{-\frac{s}{2}} |\partial_2|^{s} f \right)^2 \,\dd x. 
		\label{eq:Hnorm}
	\end{equs}

\begin{lemma}\label{lem:fractional-H}
	Let $s>0$. There exists a constant $C>0$ such that for any periodic function $f:\TT^2 \to \RR$ with vanishing average in $x_1$ there holds
	\begin{align}\label{eq:fractional-H}
		\||\partial_1|^s f\|_{L^2} + \||\partial_2|^{\frac{2}{3}s} f \|_{L^2} \leq C \Hnorm{s}{f}.
	\end{align}
\end{lemma}

\begin{proof}
	Note that by the definition of $\Hnorm{s}{f}$ we only have to bound $\||\partial_2|^{\frac{2}{3}s} f \|_{L^2} \lesssim \Hnorm{s}{f}$. This follows easily by an application of Hölder's inequality in Fourier space. Indeed, by \eqref{eq:fourier} we have that\footnote{Recall that $f$ has vanishing average in $x_1$, in particular $\widehat{f}(0, k_2) = 0$ for all $k_2 \in 2\pi\ZZ$.} 
	\begin{align*}
		\int_{\TT^2} | |\partial_2|^{\frac{2}{3}s} f|^2 \mathrm{d}x 
		&= \sum_{k\in (2\pi\ZZ)^2} |k_2|^{\frac{4}{3}s} |\widehat{f}(k)|^2 
		= \sum_{k\in (2\pi\ZZ)^2} |k_1|^{-\frac{2}{3}s} |k_2|^{\frac{4}{3}s} |\widehat{f}(k)|^{\frac{4}{3}} \, |k_1|^{\frac{2}{3}s} |\widehat{f}(k)|^{\frac{2}{3}}\\
		&\leq \left( \sum_{k\in (2\pi\ZZ)^2} |k_1|^{-s} |k_2|^{2s} |\widehat{f}(k)|^2 \right)^{\frac{2}{3}} \left( \sum_{k\in (2\pi\ZZ)^2} |k_1|^{2s} |\widehat{f}(k)|^2 \right)^{\frac{1}{3}} \\
		&= \left( \int_{\TT^2} | |\partial_1|^{-\frac{s}{2}} |\partial_2|^s f |^2 \mathrm{d}x \right)^{\frac{2}{3}} \left( \int_{\TT^2} |\partial_1 f|^{s}\,\mathrm{d}x \right)^{\frac{1}{3}} \\
		&\leq \frac{2}{3} \int_{\TT^2} | |\partial_1|^{-\frac{s}{2}} |\partial_2|^s f |^2 \mathrm{d}x + \frac{1}{3} \int_{\TT^2} ||\partial_1|^s f|^2\,\mathrm{d}x
		\leq \frac{2}{3} \Hnorm{s}{f}^2.
	\end{align*}
\end{proof}

\begin{lemma}\label{lem:1/2-1-3/4+}
	Let $\epsilon\in (0,\frac{1}{4}]$. There exists a constant $C>0$ such that for any $f,g:\TT^2 \to \RR$ with vanishing average in $x_1$ there holds 
	\begin{align*}
		\left\| |\partial_1|^{\frac{1}{2}} (fg) \right\|_{L^2} \leq C \Hnorm{1}{f} \Hnorm{\frac{3}{4}+\epsilon}{g}.
	\end{align*}  
	In particular, we have that $\| |\partial_1|^{\frac{1}{2}}(fg)\|_{L^2} \leq C \HH(f)^{\frac{1}{2}} \HH(g)^{\frac{1}{2}}$.
\end{lemma}

\begin{proof}
	By \eqref{eq:fourier}, writing $g^h := g(\cdot + h)$, we have 
	\begin{align}
		\left\| |\partial_1|^{\frac{1}{2}} (fg) \right\|_{L^2}^2 
		&= \int_{\RR} \int_{\TT^2} |\partial_1^h(fg)|^2\,\dd x \frac{\dd h}{|h|^{2}} \nonumber\\
		&\lesssim \int_{\RR} \int_{\TT^2} |\partial_1^h f|^2 |g^h|^2\,\dd x \frac{\dd h}{|h|^{2}} + \int_{\RR} \int_{\TT^2} |\partial_1^h g|^2 |f|^2\,\dd x \frac{\dd h}{|h|^{2}}. \label{eq:1-3/4-splitting}
	\end{align}
	To estimate the first term on the right-hand side of \eqref{eq:1-3/4-splitting} we use Lemma~\ref{lem:an_embedding}. We fix 
	$\epsilon \in (0,\frac{1}{4}]$ and let $p_{\epsilon}' = \frac{5+ \kappa_{\epsilon}}{2}$ with 
	$\kappa_{\epsilon} = \frac{10\epsilon}{1-2\epsilon}$, and $p_{\epsilon} = \frac{5 + \kappa_{\epsilon}}{3+\kappa_{\epsilon}}>1$. 
	Then $\tfrac{1}{p_{\epsilon}} + \tfrac{1}{p_{\epsilon}'} = 1$, so that by Hölder's inequality,
	\begin{align*}
		\int_{\RR} \int_{\TT^2} |\partial_1^h f|^2 |g^h|^2\,\dd x \frac{\dd h}{|h|^{2}}
		&\leq \int_{\RR} \left(\int_{\TT^2} |\partial_1^h f|^{2p_{\epsilon}}\,\dd x\right)^{\frac{1}{p_{\epsilon}}} \left( \int_{\TT^2} |g^h|^{5+\kappa_{\epsilon}}\,\dd x\right)^{\frac{2}{5+\kappa_{\epsilon}}} \frac{\dd h}{|h|^{2}} \\
		&= \|g\|_{L^{5+\kappa_{\epsilon}}}^2 \int_{\RR} \left(\int_{\TT^2} |\partial_1^h f|^{2p_{\epsilon}}\,\dd x\right)^{\frac{1}{p_{\epsilon}}} \frac{\dd h}{|h|^{2}},
	\end{align*}
	where in the last step we also used translation invariance. Note that the exponent $2 p_{\epsilon} \in (2, 10)$, so that we may interpolate
	\begin{align*}
		\int_{\RR} \left(\int_{\TT^2} |\partial_1^h f|^{2p_{\epsilon}}\,\dd x\right)^{\frac{1}{p_{\epsilon}}} \frac{\dd h}{|h|^{2}} 
		&= \int_{\RR} \left(\int_{\TT^2} |\partial_1^h f|^{2}\,\dd x\right)^{\theta_{\epsilon}} \left(\int_{\TT^2} |\partial_1^h f|^{10}\,\dd x\right)^{\frac{1-\theta_{\epsilon}}{5}} \frac{\dd h}{|h|^{2}}\\
		&\lesssim \|f\|_{L^{10}}^{2(1-\theta_{\epsilon})} \int_{\RR} \left(\int_{\TT^2} |\partial_1^h f|^{2}\,\dd x\right)^{\theta_{\epsilon}}\frac{\dd h}{|h|^{2}},
	\end{align*}
	with $\theta_{\epsilon} = \frac{1}{2}+\epsilon$. By Jensen's inequality, the mean-value theorem, and Poincaré's inequality for functions with zero average in $x_1$, it then follows that 
	\begin{align*}
		\int_{\RR} \left(\int_{\TT^2} |\partial_1^h f|^{2}\,\dd x\right)^{\theta_{\epsilon}}\frac{\dd h}{|h|^{2}} 
		\lesssim \|f\|_{L^2}^{1+2\epsilon} \left(\int_{|h|\geq 1} \frac{\dd h}{|h|^{2}}\right)^{\theta_{\epsilon}} + \|\partial_1 f\|_{L^2}^{1+2\epsilon} \left(\int_{|h|<1} \frac{\dd h}{|h|^{1-2\epsilon}}\right)^{\theta_{\epsilon}} 
		\lesssim \|\partial_1 f\|_{L^2}^{1+2\epsilon}.
	\end{align*}
	By Lemmata \ref{lem:an_embedding} and \ref{lem:fractional-H} we can estimate $\|f\|_{L^{10}} \lesssim \Hnorm{1}{f}$ and $\|g\|_{L^{5+\kappa_{\epsilon}}} \lesssim \Hnorm{\frac{3}{4}+\epsilon}{g}$, hence 
	\begin{align*}
		\int_{\RR} \int_{\TT^2} |\partial_1^h f|^2 |g^h|^2\,\dd x \frac{\dd h}{|h|^{2}} 
		\lesssim \Hnorm{1}{f}^2 \Hnorm{\frac{3}{4}+\epsilon}{g}^2.
	\end{align*}
	It remains to bound the second term on the right-hand side of \eqref{eq:1-3/4-splitting}. Similar to the fist term, we get by Hölder's inequality and interpolation that
	\begin{align*}
		\int_{\RR}\int_{\TT^2} |\partial_1^h g|^2 |f|^2 \, \dd x \frac{\dd h}{|h|^2} 
		&\leq \int_{\RR} \left( \int_{\TT^2} |\partial_1^h g|^{\frac{5}{2}} \,\dd x \right)^{\frac{4}{5}} \left(\int_{\TT^2} |f|^{10} \,\dd x \right)^{\frac{1}{5}} \frac{\dd h}{|h|^2} \\
		&\leq \|f\|_{L^{10}}^2 \int_{\RR} \left( \int_{\TT^2} |\partial_1^h g|^2 \,\dd x \right)^{\frac{2}{3}} \left(\int_{\TT^2} |\partial_1^h g|^5 \,\dd x \right)^{\frac{2}{15}} \frac{\dd h}{|h|^2} \\
		&\lesssim \|f\|_{L^{10}}^2 \|g\|_{L^5}^{\frac{2}{3}} \int_{\RR} \left( \int_{\TT^2} |\partial_1^h g|^2 \,\dd x \right)^{\frac{2}{3}} \frac{\dd h}{|h|^2}.
	\end{align*}
	If $\eps<\frac{1}{4}$, splitting the integral in $h$, it follows with Jensen's inequality that
	\begin{align*}
		\int_{\RR} \left( \int_{\TT^2} |\partial_1^h g|^2 \,\dd x \right)^{\frac{2}{3}} \frac{\dd h}{|h|^2} 
		&= \int_{|h|\geq 1} \left( \int_{\TT^2} |\partial_1^h g|^2 \,\dd x \right)^{\frac{2}{3}} \frac{\dd h}{|h|^2} + \int_{|h|<1} \left( \int_{\TT^2} |\partial_1^h g|^2 \,\dd x \right)^{\frac{2}{3}} \frac{\dd h}{|h|^2} \\
		&\lesssim \|g\|_{L^2}^{\frac{4}{3}} \left( \int_{|h|\geq 1} \frac{\dd h}{|h|^2} \right)^{\frac{2}{3}} + \int_{|h|<1} \left( \int_{\TT^2} \frac{|\partial_1^h g|^2}{|h|^{\frac{3}{2}+6\epsilon}} \,\dd x \right)^{\frac{2}{3}} \frac{\dd h}{|h|^{1-4\epsilon}} \\
		&\lesssim \|g\|_{L^2}^{\frac{4}{3}} + \left( \int_{|h|<1} \int_{\TT^2} \frac{|\partial_1^h g|^2}{|h|^{\frac{3}{2}+2\epsilon}} \,\dd x  \frac{\dd h}{|h|}\right)^{\frac{2}{3}} 
		\lesssim \|g\|_{L^2}^{\frac{4}{3}} + \| |\partial_1|^{\frac{3}{4}+\epsilon} g \|_{L^2}^{\frac{4}{3}}.
	\end{align*}
	Using again Lemmata \ref{lem:an_embedding} and \ref{lem:fractional-H}, we may bound 
	$\|f\|_{L^{10}} \lesssim \Hnorm{1}{f}$ and $\|g\|_{L^{5}} \lesssim \Hnorm{\frac{3}{4}+\epsilon}{g}$,
	the conclusion follows with $\|g\|_{L^2}\leq \Hnorm{\frac{3}{4}+\epsilon}{g}$.
	If $\eps=\frac{1}{4}$, we use instead the estimate
	\begin{equs}
	 \int_{|h|<1} \left( \int_{\TT^2} |\partial_1^h g|^2 \,\dd x \right)^{\frac{2}{3}} \frac{\dd h}{|h|^2} 
	 \leq \|\partial_1 g\|_{L^2}^{\frac{4}{3}} \int_{|h|<1} \frac{\dd h}{|h|^{\frac{2}{3}}} 
	 \lesssim \Hnorm{1}{g}^{\frac{4}{3}}.
	\end{equs}
\end{proof}

As for the Sobolev embedding in Lemma \ref{lem:an_embedding}, we next prove the embedding $H^1_{\mathrm{anistropic}}\subset \mathcal{C}^{\alpha}$, which is optimal and the critical exponent is given by $\alpha = 1-\frac{\mathrm{dim}}{2}= -\frac{1}{4}$. 

\begin{lemma}[Besov embedding into Hölder spaces] \label{lem:hoelder-embedding} 
	There exists a constant $C>0$ such that for any periodic distribution $f:\TT^2 \to \RR$ of vanishing average, and $s\in (0, \frac{11}{4}) \setminus \{\frac{1}{4}, \frac{3}{4}, \frac{5}{4}, \frac{9}{4}\}$, 
	there holds
	\begin{align}
		\left[f\right]_{-\frac{5}{4}+s} \leq C( \||\partial_1|^s f\|_{L^2} + \||\partial_2|^{\frac{2}{3}s} f \|_{L^2}), 
		\label{eq:s_emb}
	\end{align}
	In particular, $[w]_{-\frac{1}{4}}\leq C \mathcal{H}(w)^{\frac{1}{2}}$ for every $w\in \mathcal W$. 
\end{lemma}

\begin{proof}
	Since $f$ is of vanishing average, by \eqref{eq:neg_hoelder_A_char} and \eqref{eq:pos_hoelder_A_char} we know
	that for $s\in (0, \frac{11}{4}) \setminus \{\frac{1}{4}, \frac{3}{4}, \frac{5}{4}, \frac{9}{4}\}$,
	\begin{equs}
	 {[f]}_{-\frac{5}{4}+s} \sim \sup_{T\in(0,1]} \left(T^{\frac{1}{3}} \right)^{\frac{5}{4}-s} \|T \mathcal{A}f_T\|_{L^\infty}.
	\end{equs}
	Writing $|\partial_1|^{3} f_T=|\partial_1|^s f*|\partial_1|^{3-s} \psi_{T}$ and 
	$\partial_2^{2} f_T=|\partial_2|^{\frac{2}{3}s} f*|\partial_2|^{2-\frac{2}{3}s} \psi_{T}$,
	and using that $|\partial_j|^{\alpha} \psi \in L^2(\RR^2)$ for every $\alpha\geq 0$ and $j=1,2$, we deduce with Young's 
	inequality for convolution of functions (as in Remark~\ref{rem:perio}) that
	\begin{equs}
	\|\mathcal{A}f_T\|_{L^\infty} & \leq \||\partial_1|^s f\|_{L^2} \||\partial_1|^{3-s}\psi_{T}\|_{L^2}
	 + \||\partial_2|^{\frac{2}{3}s} f\|_{L^2} \||\partial_2|^{2-\frac{2}{3}s} \psi_{T}\|_{L^2} \\
	 &\lesssim \left(T^{\frac{1}{3}}\right)^{-3 +s - \frac{5}{4}} \left(  \||\partial_1|^s f\|_{L^2} + \||\partial_2|^{\frac{2}{3}s} f\|_{L^2} \right).
	\end{equs}
	This implies that 
	\begin{align*}
	 [f]_{-\frac{5}{4}+s}
	 \sim \sup_{T\in(0,1]} \left(T^{\frac{1}{3}} \right)^{\frac{5}{4}-s} \|T \mathcal{A}f_T\|_{L^\infty}  
	 \lesssim \||\partial_1|^s f\|_{L^2} + \||\partial_2|^{\frac{2}{3}s} f\|_{L^2}.
	\end{align*}
	Combining \eqref{eq:s_emb} for $s=1$ and \eqref{eq:bound-2/3}, we conclude that $[w]_{-\frac{1}{4}}\lesssim \mathcal{H}(w)^{\frac{1}{2}}$ for
	every $w\in \mathcal W$. 
\end{proof}

The next proposition is the classical $1$-dimensional embedding of Besov spaces into $L^p$ spaces in the periodic setting.
\footnote{The nonperiodic version of the statement is essentially a combination of 
\cite[Proposition 2.20 and Theorem 2.36]{BCD11}.} 

\begin{proposition} \label{prop:1d_besov_emb} 
	For every $p\in(1,\infty]$ and $q \in[1,p]$ with $(p,q)\neq (\infty,1)$, there exists a constant $C(p,q)>0$ such that 
	for every periodic $f:[0,1) \to \RR$ with vanishing average
	\begin{align}
		\left(\int_0^1 |f(z)|^p \, \dd z\right)^{\frac{1}{p}}
		\leq C(p,q) \int_0^1 \frac{1}{h^{\frac{1}{q}-\frac{1}{p}}} \left(\int_0^1 |\partial_1^hf(z)|^q \, \dd z\right)^{\frac{1}{q}} \frac{\dd h}{h},
		\label{eq:1d_besov_emb}
	\end{align}
	with the usual interpretation for $p=\infty$ or $q=\infty$. 
\end{proposition}

\begin{proof}
	We prove the statement for $p\in(1,\infty)$ (the case $p=\infty$ follows similarly).  
	For $T\in(0,1]$ let $\varphi_T(h)=\frac{1}{\sqrt{4\pi T}} \mathrm{e}^{-\frac{h^2}{4T}}$, $h\in \RR$, be the heat semigroup,
	and denote its periodization by
	\begin{equs}
	 \Phi_T(h)=\frac1{\sqrt{4\pi T}} \sum_{k\in \ZZ} \mathrm{e}^{-\frac{(h-k)^2}{4T}}=\sum_{k \in \ZZ} \mathrm{e}^{-4\pi^2k^2 T}
	 \mathrm{e}^{2\pi i k h}, \quad h\in [0,1).
	\end{equs}
	Note that $\Phi_T$ is smooth, $\Phi_T\geq 0$, $\|\Phi_T\|_{L^1}=\int_0^1 \Phi_T\, \dd h=1$, and
	for every $T\in(0,1]$
	\begin{equs}
	 \|\Phi_T\|_{L^\infty}\leq 1+\sum_{k \in \ZZ\setminus \{0\}} \mathrm{e}^{-4\pi^2k^2 T}\lesssim 1+\int_{\RR}
	 \widehat \varphi_T(\xi)\, \dd\xi
	 \lesssim 1+ \varphi_T(0)
	 \lesssim \frac{1}{\sqrt T}.
	\end{equs}
	Therefore, by interpolation, for every $r\in [1, \infty]$ and $T\in (0,1]$ we have 
	\begin{equation}
	\label{lp_period}
	\|\Phi_T\|_{L^r}\lesssim \sqrt{T}^{-(1-\frac{1}{r})}. 
	\end{equation}
	We also claim that for every $T\in (0,1]$
	\begin{equation}
	\label{pavlos}
	\sup_{|h|\leq \frac{1}{2}} \frac{h^2}{\sqrt T} \Phi_T(h)\lesssim 1.
	\end{equation}
	Indeed, using that for every $k\in \ZZ\setminus \{0\}$, $|h|\leq \frac12$, and $T\in (0,1]$, 
	$\frac{(h-k)^2}{4T}\geq \frac{\left(|k|-\frac{1}{2}\right)^2}{4T}\geq \frac{1}{16}$,
	we obtain\footnote{Recall that $\mathrm{e}^{-x}\lesssim \frac1x$ for $x\geq \frac{1}{16}$.}
	$\mathrm{e}^{-\frac{(h-k)^2}{4T}}\lesssim \frac{4T}{(h-k)^2}\lesssim \frac{T}{k^2}$.
	It then follows that
	\begin{align*}
	\frac{\sqrt{4\pi}h^2}{\sqrt T} \Phi_T(h) & =\frac{h^2}{T} \mathrm{e}^{-\frac{h^2}{4T}}
	+\sum_{k\in \ZZ\setminus \{0\}} \frac{h^2}T \mathrm{e}^{-\frac{(h-k)^2}{4T}}
	\lesssim \sup_{z\in \RR} z^2 \mathrm{e}^{-\frac{z^2}{4}}+h^2 \sum_{k\in \ZZ\setminus \{0\}} \frac{1}{k^2}\lesssim 1.
	\end{align*}
	
	By the semigroup property and the periodicity of $f$ we know that for $z\in (-\frac{1}{2}, \frac{1}{2})$,
		\begin{align*}
		f*\Phi_{2T}(z) - f*\Phi_T(z) = \int_{-\frac{1}{2}}^{\frac{1}{2}} \left(\partial_1^{-h}f * \Phi_{T}(z)\right) \Phi_{T}(h) \, \dd h.  
	\end{align*}
	Using Minkowski's inequality and Young's inequality for convolution with exponents 
	$1+\frac{1}{p} = \frac{1}{r}+\frac{1}{q}$ with $r\in [1, \infty)$ we deduce by \eqref{lp_period},
	\begin{equs}
		&\left(\int_{-\frac{1}{2}}^{\frac{1}{2}} |f*\Phi_{2T}(z) - f*\Phi_T(z)|^p \,\mathrm{d} z \right)^{\frac{1}{p}} 
		\lesssim \int_{-\frac{1}{2}}^{\frac{1}{2}} \Phi_T(h) \left( \int_{-\frac{1}{2}}^{\frac{1}{2}} |\partial_1^{-h}f * \Phi_{T}(z)|^p \, \mathrm{d} z \right)^{\frac{1}{p}} \, \dd h
		\\  
		&\quad \lesssim \left(\int_{-\frac{1}{2}}^{\frac{1}{2}} |\Phi_T(z)|^r \,\dd z\right)^{\frac{1}{r}}  \int_{-\frac{1}{2}}^{\frac{1}{2}} \Phi_T(h) \left(\int_{-\frac{1}{2}}^{\frac{1}{2}} |\partial_1^hf(z)|^q 
		\, \mathrm{d} z \right)^{\frac{1}{q}} \, \mathrm{d} h 
		\\
		&\quad \lesssim \frac{1}{\sqrt{T}^{\left(\frac{1}{q} - \frac{1}{p}\right)}} \int_{-\frac{1}{2}}^{\frac{1}{2}} \Phi_T(h)
		\left(\int_{-\frac{1}{2}}^{\frac{1}{2}} |\partial_1^hf(z)|^q \, \mathrm{d} z \right)^{\frac{1}{q}} \, \mathrm{d} h,
		\label{eq:conv_incr_est}
	\end{equs}
	where we also used a change of variables and periodicity to replace $-h$ by $h$. 
	
	To prove \eqref{eq:1d_besov_emb} we now write
	\begin{align*}
		f(z) = \sum_{k=1}^\infty \left(f * \Phi_{2\cdot 2^{-k}}(z) - f *\Phi_{2^{-k}}(z)\right) - f*\Phi_1(z),
	\end{align*}
	and obtain by \eqref{eq:conv_incr_est} that
	\begin{align}\label{eq:1d-Lp-besov-decomp}
		\left( \int_0^1 |f(z)|^p \, \mathrm{d} z \right)^{\frac{1}{p}} =
		\left( \int_{-\frac{1}{2}}^{\frac{1}{2}} |f(z)|^p \, \mathrm{d} z \right)^{\frac{1}{p}} 
		& \lesssim \Sigma^{(1)}+\Sigma^{(2)}  + \left(\int_{-\frac{1}{2}}^{\frac{1}{2}} |f*\Phi_1(z)|^p \, \dd z \right)^{\frac{1}{p}},
	\end{align}
	where 
	\begin{align*}
		\Sigma^{(1)} &=  \sum_{k=1}^\infty  {\sqrt{2}}^{k\left(\frac{1}{q} - \frac{1}{p}\right)} \int_{\{|h|\leq \frac{1}{2}: \, |h|\sqrt{2}^k\leq 1\}} \Phi_{2^{-k}}(h) \left(\int_{-\frac{1}{2}}^{\frac{1}{2}} |\partial_1^h f(z)|^q \, \mathrm{d} z \right)^{\frac{1}{q}} \, \mathrm{d} h, \\
		\Sigma^{(2)} &=  \sum_{k=1}^\infty  {\sqrt{2}}^{k\left(\frac{1}{q} - \frac{1}{p}\right)} 
		\int_{\{|h|\leq \frac{1}{2}:\, |h|\sqrt{2}^k> 1\}} \Phi_{2^{-k}}(h)
		\left(\int_{-\frac{1}{2}}^{\frac{1}{2}} |\partial_1^hf(z)|^q \, \mathrm{d} z \right)^{\frac{1}{q}} \, \mathrm{d} h.
	\end{align*}
	Using that by \eqref{lp_period}, $\sup_{|h|\leq \frac{1}{2}}|\Phi_{2^{-k}}(h)| \lesssim \sqrt{2}^k$, we can estimate $\Sigma^{(1)}$ as follows:
	\begin{equs}
		\Sigma^{(1)} 
		& \lesssim \int_{-\frac{1}{2}}^{\frac{1}{2}} \sum_{\{k\geq 1: \, {|h|\sqrt{2}}^k\leq 1\}} 
		{\sqrt{2}}^{k\left(\frac{1}{q} - \frac{1}{p}+1\right)} 
		\left(\int_{-\frac{1}{2}}^{\frac{1}{2}} |\partial_1^hf(z)|^q \, \mathrm{d} z \right)^{\frac{1}{q}} \, \mathrm{d} h
		\\
		& \lesssim \int_{-\frac{1}{2}}^{\frac{1}{2}} \frac{1}{|h|^{\frac{1}{q}-\frac{1}{p}}} \left(\int_{-\frac{1}{2}}^{\frac{1}{2}}
		|\partial_1^hf(z)|^q \, \mathrm{d} z \right)^{\frac{1}{q}} \, \frac{\mathrm{d} h}{|h|}.  
	\end{equs}
	For $\Sigma^{(2)}$, by \eqref{pavlos} and the fact that $\frac{1}{q} - \frac{1}{p}-1 <0$ (since $(p,q)\neq (\infty, 1)$),
	we have
	\begin{equs}
		\Sigma^{(2)} 
		& \lesssim  \sum_{k=1}^\infty {\sqrt{2}}^{k\left(\frac{1}{q} - \frac{1}{p}-1\right)} 
		\int_{\{|h|\leq \frac{1}{2}:\, |h|\sqrt{2}^k> 1\}} \frac{1}{|h|} \left(\int_{-\frac{1}{2}}^{\frac{1}{2}} 
		|\partial_1^hf(z)|^q \, \mathrm{d} z \right)^{\frac{1}{q}} \, \frac{\mathrm{d} h}{|h|}
		\\
		& \lesssim_{\frac{1}{q}-\frac{1}{p}} \int_{-\frac{1}{2}}^{\frac{1}{2}} \frac{1}{|h|^{\frac{1}{q}-\frac{1}{p}}} 
		\left(\int_{-\frac{1}{2}}^{\frac{1}{2}} |\partial_1^hf(z)|^q \, \mathrm{d} z \right)^{\frac{1}{q}}
		\, \frac{\mathrm{d} h}{|h|}.
	\end{equs}
	For the last term in \eqref{eq:1d-Lp-besov-decomp} we use that $f$ has vanishing average and periodicity, which implies that 
	$\int_{-\frac{1}{2}}^{\frac{1}{2}} f*\Phi_1(z+h) \, \dd h =0$ for every $z\in (-\frac{1}{2},\frac{1}{2})$, to obtain the bound 
	\begin{equs}
		&\left(\int_{-\frac{1}{2}}^{\frac{1}{2}} |f*\Phi_1(z)|^p \, \dd z \right)^{\frac{1}{p}} 
		= \left(\int_{-\frac{1}{2}}^{\frac{1}{2}} \left|f*\Phi_1(z)- \int_{-\frac{1}{2}}^{\frac{1}{2}} f*\Phi_1(z+h) \, \dd h\right|^p \, \dd z \right)^{\frac{1}{p}}
		\\
		&\quad\leq \int_{-\frac{1}{2}}^{\frac{1}{2}} \left(\int_{-\frac{1}{2}}^{\frac{1}{2}} |\partial_1^h f* \Phi_1(z)|^p \, \dd z \right)^{\frac{1}{p}} \, \dd h
		 \lesssim \int_{-\frac{1}{2}}^{\frac{1}{2}} \left(\int_{-\frac{1}{2}}^{\frac{1}{2}} |\partial_1^h f(z)|^q \, \dd z\right)^{\frac{1}{q}} \, \dd h
		\\
		&\quad \lesssim \int_{-\frac{1}{2}}^{\frac{1}{2}} \frac{1}{|h|^{\frac{1}{q}-\frac{1}{p}}} \left(\int_{-\frac{1}{2}}^{\frac{1}{2}} |\partial_1^h f(z)|^q \, \mathrm{d} z \right)^{\frac{1}{q}} \frac{\mathrm{d} h}{|h|},
		\label{eq:vanishing_aver_contr}
	\end{equs}
	where we also used Minkowski's inequality, Young's inequality for convolution and the fact that $\frac{1}{q} - \frac{1}{p}+1 >0$. 
	The right hand side of \eqref{eq:vanishing_aver_contr} is estimated by twice the right hand side of \eqref{eq:1d_besov_emb}, 
	so the conclusion follows.
\end{proof}

The next lemma allows us to connect the estimate \eqref{eq:GJO_clas} with regularity in Besov spaces.

\begin{lemma} \label{lem:besov-infty-p}
	There exists a constant $C>0$ such that for every $s\in(0,1)$, $p\in [1, \infty)$ and every periodic function $f:[0,1) \to \RR$
	the following estimate holds:
	\begin{equs}
		\sup_{h\in(0,1]} \frac{1}{h^s} \left(\int_0^1 |\partial_1^hf(z)|^p \, \dd z\right)^{\frac{1}{p}} 
		\leq C \sup_{h\in(0,1]} \frac{1}{h^s} 
		\left(\frac{1}{h}\int_0^h \int_0^1 |\partial_1^{h'}f(z)|^p \, \dd z \, \dd h'\right)^{\frac{1}{p}}. \label{eq:h_aver_contr}
	\end{equs}
	\end{lemma}

\begin{proof}
	Let $h\in(0,1]$. Then for $h'\in(0,h]$ we have that
	\begin{equs}
		\int_0^1 |\partial_1^hf(z)|^p \, \dd z \leq 2^{p-1} \int_0^1 |\partial_1^{h'}f(z)|^p \, \dd z + 
		2^{p-1} \int_0^1 |\partial_1^{h-h'}f(z+h')|^p \, \dd z. 
	\end{equs}
	Integrating over $h'\in[\frac{h}{2},h]$ we obtain that
	\begin{equs}
		h \int_0^1 |\partial_1^hf(z)|^p \, \dd z \leq 2^p \int_{\frac{h}{2}}^h \int_0^1 |\partial_1^{h'}f(z)|^p \, \dd z \, \dd h' + 
		2^p\int_{\frac{h}{2}}^h \int_0^1 |\partial_1^{h-h'}f(z+h')|^p \, \dd z \, \dd h'.
	\end{equs}
	By the change of variables $h'' = h-h'$ and $z' = z + h -h''$, upon relabelling, we see that
	\begin{equs}
		\int_{\frac{h}{2}}^h \int_0^1 |\partial_1^{h-h'}f(z+h')|^p \, \dd z \, \dd h' = 
		\int_0^{\frac{h}{2}} \int_0^1 |\partial_1^{h'}f(z)|^p \, \dd z \, \dd h',
	\end{equs}
	where we also used periodicity in $z$. Hence we have proved that
	\begin{equs}
		\int_0^1 |\partial_1^hf(z)|^p \, \dd z \leq \frac{2^p}{h} 
		\int_0^h \int_0^1 |\partial_1^{h'}f(z)|^p \, \dd z \, \dd h',
	\end{equs}
	which in turn implies \eqref{eq:h_aver_contr}.
\end{proof}

%% file: app-stochastic-2.tex
\section{Stochastic estimates} \label{app:stochastic}	

We show that solutions of the linearized equation 
\begin{align}\label{eq:linearized}
	\mathcal{L}v = P\xi
\end{align}
almost surely have (negative) infinite total energy 
(see \eqref{eq:total-energy} for the definition) under the law of white noise. 

\begin{proposition}\label{prop:inf-energy}
	Assume that $\langle\cdot\rangle$ is the law of white noise. 
	If $v$ is the solution of vanishing average in $x_1$-direction to $\mathcal{L}v = P\xi$, then 
	$E_{tot}(v) = -\infty$ $\langle\cdot\rangle$-almost surely.
\end{proposition}

\begin{proof}
Recall that in Fourier space we have an explicit representation of the solution to \eqref{eq:linearized} as
\begin{align*}
	\widehat{v}(k) &= \frac{\widehat{\xi}(k)}{k_1^2 + |k_1|^{-1}k_2^2} \quad \text{for } k_1 \neq 0 \quad \text{and} \quad 
	\widehat{v}(0, k_2) = 0 \quad \text{for all } k_2 \in 2\pi\ZZ.
\end{align*}
A short calculation shows that the harmonic part of the energy is
\begin{align*}
	\mathcal{H}(v) = \int_{\TT^2} \xi v \,\mathrm{d}x,
\end{align*}
so that 
\begin{align*}
	E_{tot}(v) = -\mathcal{H}(v) + \int_{\TT^2} \left( |\partial_1|^{-\frac{1}{2}} \partial_1 \frac{v^2}{2} \right)^2\,\mathrm{d}x - 2 \int_{\TT^2} \left( |\partial_1|^{-\frac{1}{2}} \partial_1 \frac{v^2}{2} \right) \left( |\partial_1|^{-\frac{1}{2}} \partial_2 v \right)\,\mathrm{d}x.
\end{align*}
By Young's inequality, we have
\begin{align}\label{eq:Etot-bound}
	E_{tot}(v) \leq -\frac{1}{2} \mathcal{H}(v) + 3 \int_{\TT^2} \left( |\partial_1|^{-\frac{1}{2}} \partial_1 \frac{v^2}{2} \right)^2\,\mathrm{d}x.
\end{align}
By \eqref{eq:besov_to_holder_1}, \eqref{eq:sobolev_to_besov} and \cite[Lemma 12]{IO19}, we may estimate
\begin{align*}
	 \int_{\TT^2} \left(|\partial_1|^{-\frac{1}{2}} \partial_1\frac{v^2}{2}\right)^2 \, \dd x &
	 = \int_{\TT^2} \left(R_1|\partial_1|^{\frac{1}{2}} \frac{v^2}{2}\right)^2 \, \dd x
	 = \int_{\TT^2} \left(|\partial_1|^{\frac{1}{2}}\frac{v^2}{2}\right)^2 \, \dd x 
	  \lesssim [v^2]_{\frac{1}{2} + \varepsilon}^2 \lesssim [v]_{\frac{3}{4}-\varepsilon}^4,
\end{align*}
where $[v]_{\frac{3}{4}-\varepsilon}$ is finite $\lng\cdot\rng$-almost surely by \eqref{eq:v_bd}

We next show that 
\begin{equs}
	\mathcal{H}(v) = \sum_{k_1\neq 0} \frac{|\widehat{\xi}(k)|^2}{k_1^2 + |k_1|^{-1} k_2^2} 
\end{equs}
diverges $\lng\cdot\rng$-almost surely. Since $\hat \xi(-k) = -\overline{\hat \xi(k)}$, we have that
\begin{equs}
 \mathcal{H}(v) = 2 \sum_{k\in 2\pi \ZZ^2\setminus\{k_1\leq 0\}} \frac{|\widehat{\xi}(k)|^2}{k_1^2 + |k_1|^{-1} k_2^2}
 = : 2 \sum_{k\in 2\pi \ZZ^2\setminus\{k_1\leq 0\}} a_k |\widehat{\xi}(k)|^2.
\end{equs}
By the independence of $\{|\hat\xi(k)|\}_{k\in 2\pi \ZZ^2\setminus\{k_1\leq 0\}}$ and Kolmogorov's $0$-$1$ law, we know that the probability of the event 
$\{\mathcal{H}(v)= +\infty\}$ is either $1$ or $0$. Hence, it is enough to show that $\lng \{\mathcal{H}(v)= +\infty\}\rng >0$. 
We first notice that 
\begin{equs}
 \sum_{k\in 2\pi \ZZ^2\setminus\{k_1\leq 0\}} a_k |\widehat{\xi}(k)|^2 \geq 
 \sum_{k\in 2\pi \ZZ^2\setminus\{k_1\leq 0\}} a_k \mathbf{1}_{\{|\widehat{\xi}(k)|^2\geq 1\}}
 =: \sum_{k\in 2\pi \ZZ^2\setminus\{k_1\leq 0\}} a_k X_k
\end{equs}
and since the random variables $\xi(k)$ are identically distributed, there exists $p\in(0,1)$ such that 
$\lng X_k\rng = p$ for every $k\in2\pi\ZZ^2\setminus\{k_1\leq 0\}$. Given $M\geq 1$, there exists 
a finite subset $J\subset2\pi\ZZ^2\setminus\{k_1\leq 0\}$ such that $\frac{p}{2} \sum_{k\in J} a_k \geq M$ and the following estimate
holds
\begin{equs}
 p \sum_{k\in J} a_k & = \left\lng \sum_{k\in J} a_k X_k \right\rng 
 \\
 & \leq \frac{p}{2} \sum_{k\in J} a_k \left(1- \bigg\lng \bigg\{\sum_{k\in J} a_k X_k \geq M\bigg\}\bigg\rng\right)
 + \sum_{k\in J} a_k \bigg\lng \bigg\{\sum_{k\in J} a_k X_k \geq M\bigg\}\bigg\rng.
\end{equs}
Then, it is easy to see that
\begin{equs}
 \bigg\lng \bigg\{\sum_{k\in J} a_k X_k \geq M\bigg\}\bigg\rng \geq \frac{p}{2-p} >0,
\end{equs}
which in turn implies 
\begin{equs} 
 \bigg\lng \bigg\{\sum_{k\in 2\pi\ZZ^2\setminus\{k_1\leq 0\}} a_k X_k =+\infty \bigg\}\bigg\rng
 = \lim_{M\uparrow +\infty} 
 \bigg\lng \bigg\{\sum_{k\in 2\pi\ZZ^2\setminus\{k_1\leq 0\}} a_k X_k \geq M\bigg\}\bigg\rng \geq \frac{p}{2-p} >0.
\end{equs}
Thus, we obtain that
\begin{equs}
 \lng \{\mathcal{H}(v)= +\infty\}\rng \geq \bigg\lng \bigg\{\sum_{k\in 2\pi\ZZ^2\setminus\{k_1\leq 0\}} a_k X_k =+\infty \bigg\}\bigg\rng>0,
\end{equs}
which proves the desired claim. 
\end{proof}

The next lemma is a Kolmogorov-type criterion for periodic random fields.

\begin{lemma} \label{lem:mod} Let $\{g(x)\}_{x\in \TT^2}$ be a random field and assume that for some 
$\alpha\in (0, \frac{3}{2})$ and every $1\leq p<\infty$ 
\begin{equs}
 \sup_{T\in (0,1]} \left(T^{\frac{1}{3}}\right)^\alpha \sup_{x\in \TT^2} \lng |g_T(x)|^p \rng^{\frac{1}{p}} <\infty.
\end{equs}
Then, for every $\varepsilon\in(0,\frac{3}{2}-\alpha)$, $g\in\C^{-\alpha-\varepsilon}$ 
almost surely and for every $1\leq p<\infty$ there exists $C(\eps, p)>0$ such that 
\begin{equs}
 \left\lng [g]_{-\alpha-\varepsilon}^p \right\rng^{\frac{1}{p}} \leq C(\eps,p).
\end{equs}
\end{lemma}

\begin{proof} First assume that $\alpha+\eps \neq 1,\frac{1}{2}$.
Let $p_\varepsilon>\frac{5}{2\varepsilon}$. We first claim that
\begin{equs}
 {} [g]_{-\alpha-\eps}
 \lesssim \sum_{n\geq 1}  
 \left(2^\frac{1}{3}\right)^{-n\left(\eps-\frac{5}{2p_\varepsilon}\right)}
 \left(2^{\frac{1}{3}}\right)^{-n\alpha} \|g_{2^{-n}}\|_{L^{p_\varepsilon}}.
\end{equs}
Indeed, for every $T\in(0,1]$ we can find $n\geq 1$ such that $2^{-n} < T \leq 2^{-n+1}$ and by the semigroup property and Remark 
\ref{rem:perio}, we obtain 
\begin{align*}
 \left(T^{\frac{1}{3}}\right)^{\alpha+\varepsilon} \|g_T\|_{L^\infty} 
 &\lesssim \left(2^{\frac{1}{3}}\right)^{-n(\alpha+\varepsilon)} \|(g_{2^{-n}})_{T-2^{-n}}\|_{L^\infty} 
 \lesssim \left(2^{\frac{1}{3}}\right)^{-n(\alpha+\varepsilon)}  \|g_{2^{-n}}\|_{L^\infty}
 \\
 & \lesssim \left(2^{\frac{1}{3}}\right)^{-n(\alpha+\varepsilon)}
 \left(2^\frac{1}{3}\right)^{(n+1)\frac{5}{2p_\varepsilon}} \|g_{2^{-n-1}}\|_{L^{p_\varepsilon}}
 \\
 & \lesssim \sum_{n\geq 1}  
 \left(2^\frac{1}{3}\right)^{-n\left(\eps-\frac{5}{2p_\varepsilon}\right)}
 \left(2^{\frac{1}{3}}\right)^{-n\alpha} \|g_{2^{-n}}\|_{L^{p_\varepsilon}}
\end{align*}
and, taking the supremum over all $T\in (0,1]$, we conclude the above claim via \eqref{eq:neg_hoelder_char}.

Then, for every $p\geq p_\varepsilon$, Minkowski's and Jensen's inequality imply
\begin{equs}
 \left\lng [g]_{-\alpha-\varepsilon}^p \right\rng^{\frac{1}{p}} 
 & \lesssim \left\lng\left(\sum_{n\geq 1}  
 \left(2^\frac{1}{3}\right)^{-n\left(\eps-\frac{5}{2p_\varepsilon}\right)}
 \left(2^{\frac{1}{3}}\right)^{-n\alpha} \|g_{2^{-n}}\|_{L^{p_\varepsilon}}\right)^p\right\rng^{\frac{1}{p}}
 \\
 & \leq \sum_{n\geq 1} \left(2^{\frac{1}{3}}\right)^{-n \left(\varepsilon - \frac{5}{2p_\varepsilon}\right)}   
 \left(2^{\frac{1}{3}}\right)^{-n\alpha} \left\lng \|g_{2^{-n}}\|_{L^{p_\varepsilon}}^p \right\rng^{\frac{1}{p}} 
 \\
 & \leq \sum_{n\geq 1} \left(2^{\frac{1}{3}}\right)^{-n \left(\varepsilon - \frac{5}{2p_\varepsilon}\right)}   
 \left(2^{\frac{1}{3}}\right)^{-n\alpha} \left\lng\|g_{2^{-n}}\|_{L^p}^p\right\rng^{\frac{1}{p}}
 \\
 & \lesssim \sum_{n\geq 1} \left(2^{\frac{1}{3}}\right)^{-n \left(\varepsilon - \frac{5}{2p_\varepsilon}\right)}
 \sup_{T\in(0,1]} \left(T^{\frac{1}{3}}\right)^\alpha \sup_{x\in \TT^2} \lng|g_T(x)|^p\rng^{\frac{1}{p}}, 
\end{equs}
where in the final step we used the estimate $\left\lng\|g_T\|_{L^p}^p\right\rng^{\frac{1}{p}}\leq \sup_{x\in \TT^2} \lng|g_T(x)|^p\rng^{\frac{1}{p}}$. As $\eps>\frac{5}{2 p_{\eps}}$, our hypothesis yields the conclusion. 
For $p\in [1, p_\eps)$, one concludes via Jensen's inequality.

In the critical case $\alpha+\eps=1,\frac{1}{2}$, one considers $\gamma<\alpha+\eps$ and applies the above to conclude $g\in \C^{-\gamma}$, which together with \cite[Remark 2]{IO19} gives $g\in \C^{-\alpha-\eps}$.
\end{proof}


%% file: app-linear.tex
\section{Some estimates for the linear equation}

\begin{lemma}\label{lem:regularised}
There exists $C>0$ such that for every $\eps\in (0,\frac1{8})$ and every $\xi\in\mathcal{C}^{-\frac{5}{4}-\eps}$,
the solution $v$ of vanishing average in $x_1$ to the equation $\mathcal{L} v = P\xi$ satisfies 
$|\partial_1|^{-1} \partial_2 v \in \mathcal{C}^{\frac{1}{4}-\eps}$ with 
	\begin{align*}
		\left[|\partial_1|^{-1} \partial_2 v\right]_{\frac{1}{4}-\eps} \leq C [\xi]_{-\frac{5}{4}-\eps}.
	\end{align*}
\end{lemma}

\begin{proof} 
	Recalling our notation $\mathcal A=|\partial_1|{\mathcal L}$, we have that 
	\begin{align*}
		P \xi_T =\mathcal{L} v_T = |\partial_1|^{-1} \mathcal{A} v_T.
	\end{align*}
	For $g=|\partial_1|^{-1} \partial_2 v$, this yields 
	\begin{align*}
		\mathcal{A} g_T = \mathcal{A} |\partial_1|^{-1} \partial_2 v_T = \partial_2 P \xi_T.
	\end{align*}
	Hence, for $T\in (0,1]$ we have that
	\begin{align*}
		\|T\mathcal{A}g_{T}\|_{L^{\infty}} 
		&= T \|\partial_2 P \xi_{T} \|_{L^{\infty}}
		= T \|P \xi_{\frac{T}{2}} * \partial_2 \psi_{\frac{T}{2}} \|_{L^{\infty}}
		= 2^{\frac{1}{2}} \left( T^{\frac{1}{3}} \right)^{3-\frac{3}{2}} \| P \xi_{\frac{T}{2}} * (\partial_2 \psi)_{\frac{T}{2}} \|_{L^{\infty}} \\
		&\lesssim \left( T^{\frac{1}{3}} \right)^{\frac{3}{2}} \| P \xi_{\frac{T}{2}}\|_{L^{\infty}} 
		\lesssim \left( T^{\frac{1}{3}} \right)^{\frac{1}{4}-\eps} [P\xi]_{-\frac{5}{4}-\eps},
	\end{align*}
	where we used the characterisation of negative Hölder spaces from Lemma \ref{lem:hoelder_char}.
	Note that the implicit constant is universal for $\eps$ small, in particular for $\eps\in (0, \frac18)$. Hence, we obtain that
	\begin{align*}
		\left[|\partial_1|^{-1} \partial_2 v\right]_{\frac{1}{4}-\eps} = [g]_{\frac{1}{4}-\eps}
		\lesssim \sup_{T\in (0,1]} \left( T^{\frac{1}{3}} \right)^{-\frac{1}{4}+\eps} \|T\mathcal{A}g_{T}\|_{L^{\infty}}
		\lesssim [\xi]_{-\frac{5}{4}-\eps}.
	\end{align*}
\end{proof}

%
%
%
%
%
%
%
%

\begin{proposition} \label{prop:sobolev_reg} There exists a constant $C>0$ such that for every 
$\delta \xi\in L^2(\TT^2)$, the solution $\delta v$ of vanishing average in $x_1$ to 
$\mathcal{L}\delta v = P \delta \xi$ satisfies the following estimates:
 \begin{equs}
  {[\delta v]}_{\frac{3}{4}} & \leq C \|\delta \xi\|_{L^2}, \label{eq:hoelder_reg}
  \\
  \|\partial_1 \delta v\|_{L^{10}} & \leq C \|\delta \xi\|_{L^2}, \label{eq:sobolev_reg_1} 
  \\
  \|\partial_2 \delta v\|_{L^{\frac{10}{3}}} & \leq C \|\delta \xi\|_{L^2}. \label{eq:sobolev_reg_3}
 \end{equs}
 \end{proposition}

\begin{proof} Writing $\partial_1^2 \delta v$ and $|\partial_2|^{\frac{4}{3}} \delta v$ as Fourier series 
and using Young's inequality in the form
%
 %
$ |k_2|^{\frac{8}{3}}\lesssim k_1^4+\frac{k_2^4}{k_1^2}$
 %
%
for every $k_1, k_2\in 2\pi \ZZ$, we obtain that 
\begin{equs}
 \|\partial^2_1 \delta v\|_{L^2}^2 + \||\partial_2|^{\frac{4}{3}} \delta v\|_{L^2}^2 
 & \lesssim \|{\mathcal L}\delta v\|_{L^2}^2 \lesssim \|P\delta \xi\|_{L^2}^2  \lesssim \|\delta \xi\|_{L^2}^2,
\end{equs}
which leads to \eqref{eq:hoelder_reg} via \eqref{eq:s_emb} for $s=2$. Using the same argument, we also obtain that 
\begin{equs}
 \|\partial_1^2\delta v\|_{L^2}^2 + \||\partial_2|^{\frac{2}{3}}\partial_1\delta v\|_{L^2}^2 & \lesssim \|\delta \xi\|_{L^2}^2, 
 \label{eq:H^2_an_est_1} 
\\
\|\partial_1|\partial_2|^{\frac{2}{3}}\delta v\|_{L^2}^2 + \||\partial_2|^{\frac{4}{3}}\delta v\|_{L^2}^2 & \lesssim \|\delta \xi\|_{L^2}^2,
\label{eq:H^2_an_est_2}
 \\
 \||\partial_1|^{\frac{1}{2}}\partial_2\delta v\|_{L^2}^2 + \||\partial_2|^{\frac{1}{3}}\partial_2\delta v\|_{L^2}^2
 & \lesssim \|\delta \xi\|_{L^2}^2 \label{eq:H^2_an_est_3}.
\end{equs}
By the embedding result in Lemma \ref{lem:an_embedding} we know that 
\begin{equs}
 \|\partial_1\delta v\|_{L^{10}} & \lesssim 
 \|\partial_1^2\delta v\|_{L^2} + \||\partial_2|^{\frac{2}{3}}\partial_1\delta v\|_{L^2}, \label{eq:emb_1}
 \\
\|\partial_2^h\delta v\|_{L^{10}} & \lesssim 
\|\partial_1 \partial_2^h\delta v\|_{L^2} + \||\partial_2|^{\frac{2}{3}}\partial_2^h\delta v\|_{L^2}, \label{eq:emb_2}
\\
 \|\partial_2 \delta v\|_{L^{\frac{10}{3}}} 
 & \lesssim \||\partial_1|^{\frac{1}{2}}\partial_2\delta v\|_{L^2} 
 + \||\partial_2|^{\frac{1}{3}}\partial_2\delta v\|_{L^2}. \label{eq:emb_3}
\end{equs}
Combining \eqref{eq:H^2_an_est_1} with \eqref{eq:emb_1} and \eqref{eq:H^2_an_est_3} with 
\eqref{eq:emb_3} we obtain \eqref{eq:sobolev_reg_1} and \eqref{eq:sobolev_reg_3}. 
\end{proof}

\begin{lemma}\label{lem:qualitative-regularity-linear} Let $\xi$ be smooth. Then the solution $v$ of $\mathcal{L}v = P\xi$ with vanishing average in $x_1$ is smooth. 
\end{lemma}

\begin{proof}
	If $\xi$ is smooth, then $\partial_1^{m} \partial_2^{n} \xi \in L^2(\TT^2)$ for all $m,n \in \NN_0$. It follows by a simple calculation in Fourier space that 
		\begin{align*}
			\|\partial_1^{m} \partial_2^{n} v\|_{L^2}^2 
			&= \sum_k |k_1|^{2m} |k_2|^{2n} |\widehat{v}(k)|^2 
			= \sum_{k_1\neq 0} |k_1|^{2m} |k_2|^{2n} \frac{|k_1|^2 |\widehat{\xi}(k)|^2 }{|k_1|^3 + |k_2|^2} \\
			&\leq \sum_{k_1\neq 0} |k_1|^{2m+2} |k_2|^{2n}|\widehat{\xi}(k)|^2 
			= \|\partial_1^{m+1} \partial_2^{n} \xi\|_{L^2}^2<\infty
		\end{align*}
		for all $m,n\in\NN_0$, in particular, $v$ is smooth by Sobolev embedding. 
\end{proof}

%% file: app-regularity.tex
\section{Regularity of finite-energy solutions for smooth data}\label{app:regularity}
In this section we develop an $L^2$-based regularity theory for weak solutions $u$ with finite energy $\HH(u)<\infty$ of the Euler--Lagrange equation
\begin{align}\label{eq:EL-appendix}
	\mathcal{L} u = 
	- P\left(u R_1\partial_2 u - \frac{1}{2} u R_1 \partial_1 u^2\right) 
	- \frac{1}{2} R_1 \partial_2  u^2  
	+  P\xi.
\end{align}


For $h\in(0,1]$, define the difference quotients
\begin{align}\label{eq:diffquot}
	D_i^h u = |h|^{-\alpha_i} \partial_i^h u, \quad \alpha_1 = 1, \alpha_2 = \frac{2}{3}.
\end{align}

\begin{proposition}[$H^{2-}$ estimate]\label{prop:H2}
	Assume that $\xi \in L^2$. There exists a constant $C>0$ such that for any solution $u$ of the Euler--Lagrange equation \eqref{eq:EL-appendix} with $\HH(u)<\infty$ we have 
	\begin{align*}
		\sup_{h\in(0,1]} \HH(D_i^h u) \leq C \left( 1 + \|\xi\|_{L^2}^2 + \HH(u)^{12} \right).
	\end{align*}
	In particular, for any $s\in[0,2)$ we have that $\||\partial_1|^s u\|_{L^2} + \||\partial_2|^{\frac{2}{3}s}u\|_{L^2}<\infty$.
\end{proposition}

\begin{proposition}[$H^{3-}$ estimate]\label{prop:H3}
	Assume that $\xi$ satisfies $\HH(\xi) <\infty$. There exists a constant $C>0$ such that for any solution $u$ of the Euler--Lagrange equation \eqref{eq:EL-appendix} with $\HH(u)<\infty$ we have 
	\begin{align*}
		\sup_{h,h'\in(0,1]} \HH(D_i^h D_i^{h'} u) \leq C \left( 1+ \HH(\xi)^{22} + \HH(u)^{132} \right).
	\end{align*}
	In particular, for any $s\in[0,3)$ we have that $\||\partial_1|^s u\|_{L^2} + \||\partial_2|^{\frac{2}{3}s}u\|_{L^2}<\infty$.
\end{proposition}

For the proof of Propositions \ref{prop:H2} and \ref{prop:H3} it is convenient to work with the scale of norms
	\begin{align*}
		\Hnorm{s}{f}^2 = \int_{\TT^2} \left( |\partial_1|^s f \right)^2\,\dd x 
		+ \int_{\TT^2} \left( |\partial_1|^{-\frac{s}{2}} |\partial_2|^{s} f \right)^2 \,\dd x,
	\end{align*}
for $s>0$ and periodic functions $f:\TT^2 \to \RR$ with vanishing average in $x_1$, as defined in \eqref{eq:Hnorm}. These norms are adapted to the harmonic 
energy $\HH$, in particular, we have that $\Hnorm{1}{f}=\mathcal{H}(f)^{\frac{1}{2}}$, and one may think of the norms
$\Hnorm{s}{\cdot}$ defining a scale of (anisotropic) Sobolev spaces. Indeed, as shown in Lemma~\ref{lem:fractional-H} the norms $\Hnorm{s}{\cdot}$ control an anisotropic 
fractional gradient in $L^2$. 

\begin{lemma}\label{lem:diffquot}
	There exists a constant $C>0$ such that for any $u\in \mathcal{W}$ there holds 
	\begin{align*}
		\sup_{h\in(0,1]} \|D_i^h u\|_{L^2} \leq C \Hnorm{1}{u}, \quad i = 1,2.
	\end{align*} 
\end{lemma}
\begin{proof}
	We treat the two directions $i=1,2$ separately. For $i=1$, the claim follows easily by the mean-value theorem, which implies that 
	\begin{align*}
		\sup_{h\in(0,1]}\|D_1^h u\|_{L^2} \leq \|\partial_1 u\|_{L^2}.
	\end{align*}
	For $i=2$ we appeal to Lemma \ref{lem:besov-infty-p} (or rather its analogue for functions on $\TT^2$) to estimate
	\begin{align*}
		\sup_{h\in(0,1]} \|D_2^h u\|_{L^2}^2 
		&= \sup_{h\in(0,1]} \frac{1}{h^{\frac{4}{3}}} \int_{\TT^2} |\partial_2^h u|^2\,\dd x 
		\lesssim \sup_{h\in(0,1]} \frac{1}{h^{\frac{4}{3}}} \frac{1}{h} \int_0^h \int_{\TT^2} |\partial_2^{h'} u|^2\,\dd x\,\dd h' \\
		&\lesssim  \sup_{h\in(0,1]} \int_0^h \int_{\TT^2} \frac{|\partial_2^{h'} u|^2}{h'^{\frac{4}{3}}} \,\dd x \frac{\dd h'}{h'} 
		 \lesssim \int_{\RR} \int_{\TT^2} \frac{|\partial_2^{h} u|^2}{|h|^{\frac{4}{3}}} \,\dd x \frac{\dd h}{|h|} 
		\stackrel{\eqref{eq:fourier}}{\lesssim} \||\partial_2|^{\frac{2}{3}} u \|_{L^2}^2.
	\end{align*}
	The conclusion then follows from Lemma \ref{lem:fractional-H}.
\end{proof}

The proof of Propositions \ref{prop:H2} and \ref{prop:H3} mainly relies on the following two lemmata: 
\begin{lemma}\label{lem:box1}
	Let $\epsilon\in (0,\frac{1}{4})$. There exists a constant $C_{\epsilon}>0$ such that for all periodic functions $f,g,\varphi:\TT^2 \to \RR$ of vanishing average in $x_1$ we have 
	\begin{align*}
		\left| \int_{\TT^2} f (\partial_2 g) \varphi \,\dd x \right| \leq C_{\epsilon} \Hnorm{1}{f} \Hnorm{1}{g} \Hnorm{\frac{3}{4}+\epsilon}{\varphi}.
	\end{align*}
\end{lemma}

\begin{proof}
	This follows easily from the definition of $\Hnorm{1}{\cdot}$ and Lemma \ref{lem:1/2-1-3/4+}. Indeed, we have 
	\begin{align*}
		\left| \int_{\TT^2} f (\partial_2 g) \varphi \,\dd x \right|
		&= \left| \int_{\TT^2} |\partial_1|^{-\frac{1}{2}} \partial_2 g \, |\partial_1|^{\frac{1}{2}}(f \varphi) \,\dd x \right| 
		\lesssim \Hnorm{1}{g} \| |\partial_1|^{\frac{1}{2}}(f \varphi) \|_{L^2} 
		\lesssim \Hnorm{1}{g} \Hnorm{1}{f} \Hnorm{\frac{3}{4}+\epsilon}{\varphi}.
	\end{align*}
\end{proof}

\begin{lemma}\label{lem:box2}
	There exists a constant $C>0$ such that for all periodic functions $f,g,h,\varphi:\TT^2 \to \RR$ of vanishing average in $x_1$ we have 
	\begin{align*}
		\left| \int_{\TT^2} f R_1(g \partial_1 h) \varphi \,\dd x \right| \leq C \Hnorm{1}{f} \Hnorm{1}{g} \Hnorm{1}{h} \Hnorm{\frac{1}{2}}{\varphi}.
	\end{align*}
\end{lemma}

\begin{proof}
	We first use Cauchy--Schwarz to bound
	\begin{align*}
		\left| \int_{\TT^2} f R_1(g \partial_1 h) \varphi \,\dd x \right| 
		=\left| \int_{\TT^2} (\partial_1 h) g R_1(f\varphi) \,\dd x \right| 
		\leq \|\partial_1 h \|_{L^2} \| g R_1(f\varphi) \|_{L^2}
		\leq \Hnorm{1}{h} \| g R_1(f\varphi) \|_{L^2} .
	\end{align*}
	By Hölder's inequality and the boundedness of $R_1$ on $L^{\frac{5}{2}}(\TT^2)$ (\ref{lem:R}) , we may further estimate
	\begin{align*}
		\| g R_1(f\varphi) \|_{L^2} 
		\leq \|g\|_{L^{10}} \|R_1(f \varphi)\|_{L^{\frac{5}{2}}}
		\lesssim \|g\|_{L^{10}} \|f \varphi\|_{L^{\frac{5}{2}}}
		\lesssim \|g\|_{L^{10}} \|f\|_{L^{10}} \|\varphi\|_{L^{\frac{10}{3}}},
	\end{align*}
	from which the claim follows with the Sobolev-type embeddings
	\begin{align*}
		\|g\|_{L^{10}} &\stackrel{\eqref{eq:H^1_an_emb}}{\lesssim} \|\partial_1 g\|_{L^2} + \||\partial_2|^{\frac{2}{3}} g\|_{L^2} \stackrel{\eqref{eq:fractional-H}}{\lesssim} \Hnorm{1}{g},\\
		\|\varphi\|_{L^{\frac{10}{3}}} &\stackrel{\eqref{eq:H^1/2_an_emb}}{\lesssim} \||\partial_1|^{\frac{1}{2}} \varphi\|_{L^2} + \||\partial_2|^{\frac{1}{3}} \varphi\|_{L^2} \stackrel{\eqref{eq:fractional-H}}{\lesssim} \Hnorm{\frac{1}{2}}{\varphi}.
	\end{align*}
\end{proof}

\begin{proof}[Proof of Proposition \ref{prop:H2}]
	Since $\Hnorm{1}{u}^2 = \HH(u) <\infty$, we can test the Euler--Lagrange equation \eqref{eq:EL-appendix} with $\psi = D_i^{-h} D_i^h u$, where $D_i^h$ denotes the difference quotient introduced in \eqref{eq:diffquot} for $h\in (0,1]$. Then the left-hand side of \eqref{eq:EL-appendix} turns into
	\begin{align*}
		\int_{\TT^2} \mathcal{L} u D_i^{-h} D_i^h u \,\dd x = \HH(D_i^h u) = \Hnorm{1}{D_i^h u}^2.
	\end{align*}
	We continue with estimating the terms on the right-hand side of the Euler--Lagrange equation:
	\begin{enumerate}[label=\textsc{Term \arabic*},leftmargin=0pt,labelsep=*,itemindent=*]
		\item ($u R_1\partial_2 u$). By a discrete integration by parts, we can write
		\begin{align*}
			\int_{\TT^2} u R_1 \partial_2 u D_i^{-h} D_i^h u \,\dd x 
			= \int_{\TT^2} D_i^h u (\partial_2 R_1 u) D_i^h u \,\dd x + \int_{\TT^2} u^h \partial_2(D_i^h R_1 u) D_i^h u \,\dd x,
		\end{align*}
		where $u^h = u( \cdot + h)$.
		Applying Lemma \ref{lem:box1} to each of these two terms, together with 
		\begin{align}\label{eq:Ruh}
			\Hnorm{1}{R_1u} = \Hnorm{1}{u} \quad  \text{and}  \quad \Hnorm{1}{u^h} = \Hnorm{1}{u}, 
		\end{align}
		then yields
		\begin{align*}
			\left| \int_{\TT^2} D_i^h u (\partial_2 R_1 u) D_i^h u \,\dd x \right|
			&\lesssim \Hnorm{1}{D_i^h u} \Hnorm{1}{R_1u} \Hnorm{\frac{4}{5}}{D_i^h u}
			\lesssim \Hnorm{1}{D_i^h u} \Hnorm{1}{u} \Hnorm{\frac{4}{5}}{D_i^h u}, \\
			\left| \int_{\TT^2} u^h \partial_2(D_i^h R_1 u) D_i^h u \,\dd x \right|
			&\lesssim \Hnorm{1}{u^h} \Hnorm{1}{R_1 D_i^h u} \Hnorm{\frac{4}{5}}{D_i^h u}
			\lesssim \Hnorm{1}{u} \Hnorm{1}{D_i^h u} \Hnorm{\frac{4}{5}}{D_i^h u}.
		\end{align*}
		Note that by interpolation (which is easily seen in Fourier space) and Lemma \ref{lem:diffquot}, there holds
		\begin{align}\label{eq:H3/4-interpolation}
			\Hnorm{\frac{4}{5}}{D_i^h u} \leq \|D_i^h u\|_{L^2}^{\frac{1}{5}} \Hnorm{1}{D_i^h u}^{\frac{4}{5}}
			\lesssim \Hnorm{1}{u}^{\frac{1}{5}} \Hnorm{1}{D_i^h u}^{\frac{4}{5}},
		\end{align}
		hence 
		\begin{align*}
			\left|\int_{\TT^2} u R_1 \partial_2 u D_i^{-h} D_i^h u \,\dd x \right| 
			\lesssim \Hnorm{1}{u}^{\frac{6}{5}} \Hnorm{1}{D_i^h u}^{\frac{9}{5}}.
		\end{align*}
		
		\item ($R_1\partial_2 u^2$). As for the first term, after a discrete integration by parts
		\begin{align*}
			\int_{\TT^2} R_1 \partial_2 u^2 D_i^{-h} D_i^h u \,\dd x 
			&= \int_{\TT^2} D_i^{h}(u\partial_2 u) R_1D_i^h u \,\dd x \\
			&= \int_{\TT^2} (D_i^{h}u) \partial_2 u R_1D_i^h u \,\dd x + \int_{\TT^2} u^h D_i^{h}\partial_2 u R_1D_i^h u \,\dd x,
		\end{align*}
		an application of Lemma \ref{lem:box1}, the interpolation \eqref{eq:H3/4-interpolation}, and Lemma \ref{lem:diffquot} implies
		\begin{align*}
			\left| \int_{\TT^2} R_1 \partial_2 u^2 D_i^{-h} D_i^h u \,\dd x \right|
			&\lesssim \Hnorm{1}{D_i^h u} \Hnorm{1}{u} \Hnorm{\frac{4}{5}}{D_i^h u}
			\lesssim \Hnorm{1}{u}^{\frac{6}{5}} \Hnorm{1}{D_i^h u}^{\frac{9}{5}}.
		\end{align*}
		
		\item ($u\partial_1 R_1 u^2$). For the cubic term in $u$ we have 
		\begin{align*}
			\int_{\TT^2}u\partial_1 R_1 u^2 D_i^{-h} D_i^h u \,\dd x
			&= \int_{\TT^2} (D_i^h u) R_1( u\partial_1 u) D_i^h u \,\dd x
			+\int_{\TT^2} u^h R_1( D_i^h u\partial_1 u) D_i^h u \,\dd x \\
			&\quad +\int_{\TT^2} u^h R_1( u^h \partial_1 D_i^h u) D_i^h u \,\dd x.
		\end{align*}
		By Lemma \ref{lem:box2} we may estimate the first term by 
		\begin{align*}
			\int_{\TT^2} (D_i^h u) R_1( u\partial_1 u) D_i^h u \,\dd x 
			\lesssim \Hnorm{1}{D_i^h u} \Hnorm{1}{u} \Hnorm{1}{u} \Hnorm{\frac{1}{2}}{D_i^h u},
		\end{align*}
		and similarly for the other two terms, so that, together with \eqref{eq:Ruh},
		\begin{align*}
			\left|\int_{\TT^2}u\partial_1 R_1 u^2 D_i^{-h} D_i^h u \,\dd x \right|
			\lesssim \Hnorm{1}{u}^2 \Hnorm{1}{D_i^h u} \Hnorm{\frac{1}{2}}{D_i^h u}.
		\end{align*}
		Interpolation and Lemma \ref{lem:diffquot} then yield
		\begin{align*}
			\Hnorm{\frac{1}{2}}{D_i^h u} 
			\leq \|D_i^h u\|_{L^2}^{\frac{1}{2}} \Hnorm{1}{D_i^h u}^{\frac{1}{2}}
			\lesssim \Hnorm{1}{u}^{\frac{1}{2}} \Hnorm{1}{D_i^h u}^{\frac{1}{2}},
		\end{align*}
		hence
		\begin{align*}
			\left|\int_{\TT^2}u\partial_1 R_1 u^2 D_i^{-h} D_i^h u \,\dd x \right|
			\lesssim \Hnorm{1}{u}^{\frac{5}{2}} \Hnorm{1}{D_i^h u}^{\frac{3}{2}}.
		\end{align*}
		
		\item ($P\xi$). Note that by assumption $\xi \in L^2$, so that Cauchy--Schwarz and Lemma \ref{lem:diffquot} imply
		\begin{align*}
			\left| \int_{\TT^2} P\xi D_i^{-h} D_i^h u \,\dd x \right| 
			\leq \|P\xi\|_{L^2} \|D_i^{-h} D_i^h u \|_{L^2} 
			\lesssim \|P\xi\|_{L^2} \Hnorm{1}{D_i^h u}.
		\end{align*}
	\end{enumerate}
	
	With the above estimates on the three superlinear terms on the right-hand side of the Euler--Lagrange equation, we have therefore shown that
	\begin{align*}
		\Hnorm{1}{D_i^h u}^2 \lesssim \Hnorm{1}{u}^{\frac{6}{5}} \Hnorm{1}{D_i^h u}^{\frac{9}{5}} + \Hnorm{1}{u}^{\frac{5}{2}} \Hnorm{1}{D_i^h u}^{\frac{3}{2}} + \|\xi\|_{L^2} \Hnorm{1}{D_i^h u}, 
	\end{align*}
	which by Young's inequality implies that
	\begin{align*}
		\Hnorm{1}{D_i^h u}^2 \lesssim \Hnorm{1}{u}^{12} + \Hnorm{1}{u}^{10} + \|P\xi\|_{L^2}^2 
		\lesssim 1 + \Hnorm{1}{u}^{12} + \|P\xi\|_{L^2}^2,
	\end{align*}
	with an implicit constant that does not depend on $h$. Finally, we notice that a bound on the quantity $\sup_{h\in(0,1]} \Hnorm{1}{D_i^h u}$ implies a bound on 
	$\||\partial_1|^s u\|_{L^2} + \||\partial_2|^{\frac{2}{3}s}u\|_{L^2}$ for $s\in[0,1)$ by 
	Lemma~\ref{lem:fractional-H} and \eqref{eq:fourier}.
\end{proof}

\begin{proof}[Proof of Proposition \ref{prop:H3}]
	The proof is very similar to the proof of Proposition \ref{prop:H2}. Under the assumptions of Proposition \ref{prop:H3}, Proposition \ref{prop:H2} implies that 
	\begin{align}\label{eq:propH2-consequence}
		\sup_{h\in(0,1]}\Hnorm{1}{D_i^h u} \lesssim 1 + \Hnorm{1}{u}^6 + \|P\xi\|_{L^2} 
		\lesssim 1 + \Hnorm{1}{u}^6 + \Hnorm{1}{\xi}, 
	\end{align}
	so that we may test the Euler--Lagrange equation \eqref{eq:EL-appendix} with $\psi=D_i^{-h'} D_i^{h'} D_i^{-h} D_i^h u$. Then the left-hand side of \eqref{eq:EL-appendix} turns into
	\begin{align*}
		\int_{\TT^2} \mathcal{L} u \, D_i^{-h'} D_i^{h'} D_i^{-h} D_i^h u \,\dd x = \Hnorm{1}{D_i^{h'}D_i^h u}^2.
	\end{align*}
	We proceed by estimating each term on the right-hand side of \eqref{eq:EL-appendix}. 
	\begin{enumerate}[label=\textsc{Term \arabic*},leftmargin=0pt,labelsep=*,itemindent=*]
		\item ($u R_1\partial_2 u$). Integrating by parts twice (with respect to $D_i^{-h'}$ and $D_i^{-h}$, we obtain four terms, all of which can be estimated by either
		\begin{align}\label{eq:term1-bounds}
			\Hnorm{1}{u} \Hnorm{1}{D_i^{h'}D_i^h u} \Hnorm{\frac{4}{5}}{D_i^{h'} D_i^h u}, \quad {\text{or}} \quad 
			\Hnorm{1}{D_i^{h'} u} \Hnorm{1}{D_i^h u} \Hnorm{\frac{4}{5}}{D_i^{h'} D_i^h u}.
		\end{align}
		By interpolation, see \eqref{eq:H3/4-interpolation}, and Lemma \ref{lem:diffquot}, there holds
		\begin{align*}
			\Hnorm{\frac{4}{5}}{D_i^{h'} D_i^h u} \leq \| D_i^{h'} D_i^h u \|_{L^2}^{\frac{1}{5}} \Hnorm{1}{D_i^{h'} D_i^h u}^{\frac{4}{5}} 
			\lesssim \Hnorm{1}{D_i^h u}^{\frac{1}{5}} \Hnorm{1}{D_i^{h'} D_i^h u}^{\frac{4}{5}} 
		\end{align*}
		\item ($R_1\partial_2 u^2$). This term is treated like the previous one, with the same bounds \eqref{eq:term1-bounds}.
		\item ($u\partial_1 R_1 u^2$). For this term we again integrate by parts twice to obtain nine terms, all of which can be bounded by one of the expressions
		\begin{align*}
			\Hnorm{1}{u} \Hnorm{1}{D_i^{h'} u} \Hnorm{1}{D_i^h u} \Hnorm{\frac{1}{2}}{D_i^{h'} D_i^h u}, \quad \text{or} \quad 
			\Hnorm{1}{u}^2 \Hnorm{1}{D_i^{h'} D_i^h u} \Hnorm{\frac{1}{2}}{D_i^{h'} D_i^h u},
		\end{align*}
		where by interpolation and Lemma \ref{lem:diffquot}
		\begin{align*}
			\Hnorm{\frac{1}{2}}{D_i^{h'} D_i^h u} \leq \|D_i^{h'} D_i^h u\|_{L^2}^{\frac{1}{2}} \Hnorm{1}{D_i^{h'} D_i^h u}^{\frac{1}{2}} 
			\lesssim \Hnorm{1}{D_i^h u}^{\frac{1}{2}} \Hnorm{1}{D_i^{h'} D_i^h u}^{\frac{1}{2}}.
		\end{align*}
		\item ($P\xi$). Under the regularity assumption $\Hnorm{1}{\xi}<\infty$ on $\xi$ we can also estimate with Cauchy--Schwarz and Lemma \ref{lem:diffquot}
		\begin{align*}
			\left| \int_{\TT^2} P\xi D_i^{-h'} D_i^{h'} D_i^{-h} D_i^h u \,\dd x \right|
			&= \left| \int_{\TT^2} D_i^{h'}  P\xi D_i^{h'} D_i^{-h} D_i^h u \,\dd x \right|
			\leq \|D_i^{h'} \xi\|_{L^2} \|D_i^{-h} D_i^{h'} D_i^h u\|_{L^2} \\
			&\leq \Hnorm{1}{\xi} \Hnorm{1}{D_i^{h'} D_i^h u}.
		\end{align*}
	\end{enumerate}
	In total, we can therefore bound
	\begin{align*}
		\Hnorm{1}{D_i^{h'}D_i^h u}^2 &\lesssim \Hnorm{1}{u} \Hnorm{1}{D_i^{h} u}^{\frac{1}{5}} \Hnorm{1}{D_i^{h'} D_i^h u}^{\frac{9}{5}} + \Hnorm{1}{D_i^{h'} u} \Hnorm{1}{D_i^h u}^{\frac{6}{5}} \Hnorm{1}{D_i^{h'} D_i^h u}^{\frac{9}{5}} \\
		&\quad+ \Hnorm{1}{u} \Hnorm{1}{D_i^{h'} u} \Hnorm{1}{D_i^{h} u}^{\frac{3}{2}} \Hnorm{1}{D_i^{h'} D_i^h u}^{\frac{1}{2}} + \Hnorm{1}{u}^2 \Hnorm{1}{D_i^{h} u}^{\frac{1}{2}}\Hnorm{1}{D_i^{h'} D_i^h u}^{\frac{3}{2}} \\
		&\quad +\Hnorm{1}{\xi} \Hnorm{1}{D_i^{h'} D_i^h u},
	\end{align*}
	from which it follows by Young's inequality and \eqref{eq:propH2-consequence} that 
	\begin{align*}
		\Hnorm{1}{D_i^{h'}D_i^h u}^2 \lesssim 1 + \Hnorm{1}{u}^{132} + \Hnorm{1}{\xi}^{22},
	\end{align*}
	with an implicit constant that does not depend on $h,h'$. As in the proof of Proposition \ref{prop:H2}, we notice that a bound on 
	$\sup_{h,h'\in(0,1]} \Hnorm{1}{D_i^hD_i^{h'} u}^2$ implies a bound on $\||\partial_1|^s u\|_{L^2} + \||\partial_2|^{\frac{2}{3}s}u\|_{L^2}$ 
	for $s\in[2,3)$ by Lemma \ref{lem:fractional-H} and \eqref{eq:fourier}.
\end{proof}

%% file: app-cylinder.tex
\section{Approximation of quadratic functionals of the noise by cylinder functionals} \label{app:cylinder}

In this section we show the following approximation result, which seems classical but since we could not find a proof in the literature we include it here:

\begin{lemma} \label{lem:cylinder}
	Let $K$ be a linear operator on the space of Schwartz distributions $\mathcal{S}'(\TT^2)$ such that $K(\mathcal{S}'(\TT^2))\subseteq\C^{\infty}(\TT^2)$ and $K: L^2(\TT^2) \to L^2(\TT^2)$ is Hilbert--Schmidt. 
	Assume further that the probability measure $\langle\cdot\rangle$ satisfies Assumption~\ref{ass}~\ref{item:Def-SG}. 
	Consider the quadratic functional $G: \mathcal{S}'(\TT^2) \to \RR$ given by 
	\begin{align*}
		G(\xi) := \xi(K\xi).
	\end{align*}
	Then under the assumption that $\langle |G(\xi)| \rangle < \infty$, $G$ is well-approximated by cylindrical functionals in with respect to
	the norm $\lng |G(\xi)|^{2p} \rng^{\frac{1}{2p}} + \lng \|\frac{\partial G}{\partial\xi}(\xi)\|_{L^2}^{2p} \rng^{\frac{1}{2p}}$ for every 
	$1\leq p<\infty$.
\end{lemma}

\begin{proof}
Without loss of generality, we may assume that $K$ is symmetric. Since $K$ is Hilbert--Schmidt, there exists an orthonormal system 
$\{\varphi_n\}_{n\in\NN}$ of $L^2(\TT^2)$ and a sequence $\{\lambda_n\}_{n\in\NN}\subset \RR$ such that 
\begin{align}\label{eq:kernel-decomposition}
	K = \sum_{n\in\NN} \lambda_n (\varphi_n, \cdot) \varphi_n, \quad \text{and} \quad \|K\|_{HS}^2 = \sum_n \lambda_n^2 < \infty.
\end{align}
Note that by the assumption $K(\mathcal{S}'(\TT^2)) \subseteq \C^{\infty}(\TT^2)$ there holds $\{\varphi_n\}_{n\in\NN} \subset \C^{\infty}(\TT^2)$, in particular $\xi(\varphi_n)$ is well-defined for any $n\in\NN$, and given $N\in\NN$ we may define $K_N \xi := \sum_{n\leq N} \lambda_n \xi(\varphi_n) \varphi_n$. 

\begin{enumerate}[label=\textsc{Step \arabic*},leftmargin=0pt,labelsep=*,itemindent=*]
\item We first show that for any $\xi\in\mathcal{S}'(\TT^2)$, $K_N \xi \to K\xi$ in $\mathcal{S}'(\TT^2)$ as $N\to \infty$. 

Indeed, with $\xi_t := \xi*\psi_t \in \C^{\infty}(\TT^2)$, for any $\phi \in \C^{\infty}(\TT^2)$ we have 
\begin{align*}
	|(K_N \xi - K\xi, \phi)| 
	\leq |(K_N \xi - K_N \xi_t, \phi)| + |(K_N \xi_t - K \xi_t, \phi)| + |(K\xi_t - K\xi, \phi)|.
\end{align*}
Note that $\xi_t \to \xi$ in $\mathcal{S}'$ as $t\downarrow 0$, hence for any $\epsilon > 0$ we may choose $t>0$ small enough so that by symmetry of $K_N$ and $K$,
\begin{align*}
	|(K_N \xi - K_N \xi_t, \phi)| &= |(\xi-\xi_t)(K_N \phi)| < \frac{\epsilon}{3},\\
	|(K\xi_t - K\xi, \phi)| &= |(\xi-\xi_t)(K \phi)| < \frac{\epsilon}{3},
\end{align*}
for any $N\in\NN$. The remaining term can be bounded using Cauchy-Schwarz and Bessel's inequality for orthonormal systems to obtain
\begin{align*}
	|(K_N \xi_t - K \xi_t, \phi)| 
	&= \left| \left( \sum_{n>N} \lambda_n (\varphi_n,\xi_t) (\varphi_n, \phi) \right)\right| 
	\leq  \left( \sum_{n>N} \lambda_n^2 |(\varphi_n,\xi_t)|^2 \right)^{\frac{1}{2}} \left( \sum_{n\in\NN} |(\varphi_n, \phi)|^2 \right)^{\frac{1}{2}} \\
	&\leq \|\xi_t\|_{L^2} \left( \sum_{n>N} \lambda_n^2 \right)^{\frac{1}{2}} \|\phi\|_{L^2}
	< \frac{\epsilon}{3}
\end{align*}
if $N$ is chosen suitably large, since $\{\lambda_n\}_n$ is square-summable. 
\medskip
\item We next claim that $G_N(\xi) := \xi(K_N\xi)$ satisfies $G_N(\xi) \to G(\xi)$ as $N\to\infty$ for all $\xi\in\mathcal{S}'(\TT^2)$.

For the proof of this claim we again appeal to the convergence $\xi_t \to \xi$ in $\mathcal{S'}$ by splitting 
\begin{align*}
	|G_N(\xi) - G(\xi)| 
	&= | \xi(K_N \xi) - \xi(K\xi) | \\
	&\leq | \xi(K_N \xi) - \xi_t(K_N \xi) | + |\xi_t(K_N \xi) - \xi_t(K \xi)| + |\xi_t(K\xi) - \xi(K\xi)|.
\end{align*}
As in \textsc{Step 1}, there holds
\begin{align*}
	| \xi(K_N \xi) - \xi_t(K_N \xi) | \to 0, \quad 
	|\xi_t(K\xi) - \xi(K\xi)| \to 0 
\end{align*}
as $t\downarrow 0$ for any $N\in\NN$, since $K_N \xi, K\xi \in \C^{\infty}(\TT^2)$. Moreover, by \textsc{Step 1}, we have
\begin{align*}
	|\xi_t(K_N \xi) - \xi_t(K \xi)| 
	= |(K_N \xi - K\xi, \xi_t) | \to 0 
\end{align*}
as $N\to \infty$.

\medskip
\item We show that $\frac{\partial G_N(\xi)}{\partial\xi} \to \frac{\partial G(\xi)}{\partial\xi}$ as $N\to\infty$ in $L^{p}_{\langle\cdot\rangle}L^2_x$ for any $1\leq p<\infty$. 

Indeed, by symmetry of $K$ and $K_N$ we have that 
\begin{align*}
	\frac{\partial G_N(\xi)}{\partial\xi} = 2 K_N\xi = 2 \sum_{n\leq N} \lambda_n \xi(\varphi_n) \varphi_n, \quad \text{and} \quad 
	\frac{\partial G(\xi)}{\partial\xi} = 2 K \xi = 2 \sum_{n\in\NN} \lambda_n \xi(\varphi_n) \varphi_n.
\end{align*}
With this we obtain by orthonormality of the $\{\varphi_n\}_n$ that
\begin{align*}
	\left\| \frac{\partial (G-G_N)(\xi)}{\partial\xi} \right\|_{L^2}^{2} 
	= 4 \sum_{n> N} \lambda_n^2 \xi(\varphi_n)^2,
\end{align*}
hence, applying \eqref{eq:SG_mult} in the form of 
\begin{align*}
	\left\langle \xi(\varphi_n)^{2p} \right\rangle^{\frac{1}{2p}} 
	\lesssim_p \|\varphi_n\|_{L^2}
	\lesssim_p 1,
\end{align*}
for all $n\in\NN$, it follows that 
\begin{align*}
	\left\langle \left\| \frac{\partial G(\xi)}{\partial\xi} -  \frac{\partial G_N(\xi)}{\partial\xi} \right\|_{L^2}^{2p}\right\rangle^{\frac{1}{2p}} 
	&= \left\langle \left( 4 \sum_{n>N} \lambda_n^2 \xi(\varphi_n)^2 \right)^p \right\rangle^{\frac{1}{2p}} 
	\leq 2 \left(\sum_{n>N} \lambda_n^2 \left\langle \xi(\varphi_n)^{2p} \right\rangle^{\frac{1}{p}} \right)^{\frac{1}{2}} \\
	&\lesssim_p \left(\sum_{n>N} \lambda_n^2 \right)^{\frac{1}{2}}
	\stackrel{N\to\infty}{\longrightarrow} 0,
\end{align*}
by the finiteness of $\|K\|_{HS}$. 
\medskip
\item We show that the sequence $\{G_N(\xi) - \langle G_N(\xi) \rangle\}_{N\in\NN}$ is Cauchy 
in $L^{2p}_{\langle\cdot\rangle}$ for any $1\leq p<\infty$, in particular, there exists a centred random
variable $G_*(\xi)$ such that $G_N(\xi) - \langle G_N(\xi) \rangle \to G_*(\xi)$ for
$\langle\cdot\rangle$-almost every $\xi$ as $N\to\infty$ along a subsequence. 

Indeed,
since $G_N(\xi)$ is a cylinder functional and $\langle\cdot\rangle$ satisfies Assumption~\ref{ass}~\ref{item:Def-SG}, in particular Proposition~\ref{prop:SG_mult}, we can apply \eqref{eq:SG_mult} to obtain the bound 
\begin{align*}
	\left\langle |G_N(\xi) - \langle G_N(\xi)\rangle - (G_M(\xi) - \langle G_M(\xi)\rangle)|^{2p} \right\rangle^{\frac{1}{2p}} 
	&= \left\langle |G_N(\xi) - G_M(\xi) - \langle G_N(\xi) - G_M(\xi) \rangle|^{2p}\right\rangle^{\frac{1}{2p}} \\
	&\lesssim_p \left\langle \left\| \frac{\partial (G_N-G_M)(\xi)}{\partial\xi} \right\|_{L^2}^{2p}\right\rangle^{\frac{1}{2p}},
\end{align*}
for $N\geq M$. Hence, by \textsc{Step 3},
\begin{align*}
	\left\langle |G_N(\xi) - \langle G_N(\xi)\rangle - (G_M(\xi) - \langle G_M(\xi)\rangle)|^{2p} \right\rangle^{\frac{1}{2p}}
	\stackrel{M,N\to\infty}{\longrightarrow} 0,
\end{align*}
for any $1\leq p<\infty$.

\medskip
\item We claim that $G_N(\xi) \to G(\xi)$ in $L^{2p}_{\lng\cdot\rng}$ as $N\to\infty$. 

Indeed, by \textsc{Step 2} we know that $G_N(\xi) \to G(\xi)$ almost surely, so that with the result from \textsc{Step 4} we may conclude that  
\begin{align*}
	\langle G_N(\xi) \rangle = G_N(\xi) - (G_N(\xi) - \langle G_N(\xi) \rangle) \to G(\xi) - G_*(\xi)
\end{align*}
almost surely along a subsequence as $N\to\infty$. But since $\langle G_N (\xi) \rangle$ is constant in $\xi$ 
(recall that $\xi$ denotes the dummy variable over which $\langle\cdot\rangle$ integrates), the random variable $G(\xi) - G_*(\xi)$ 
must be almost surely constant, i.e.
\begin{align*}
	G(\xi) - G_*(\xi) = \langle G(\xi) - G_*(\xi) \rangle = \langle G(\xi) \rangle ,
\end{align*}
since $\langle G\rangle < \infty$ by assumption and $G_*$ is centred. Hence $\langle G_N (\xi)\rangle \to \langle G(\xi) \rangle$ as $N\to \infty$ along a subsequence.
Since the above argument can be repeated for any subsequence, we obtain that $\langle G_N (\xi)\rangle \to \langle G(\xi) \rangle$ as $N\to \infty$. Since 
$G_N(\xi) \to G(\xi)$ almost surely, we conclude that $G_*(\xi)= G(\xi) - \langle G (\xi)\rangle$ almost surely, which together with \textsc{Step 4}
implies that $G_N(\xi) \to G(\xi)$ in $L^{2p}_{\lng\cdot\rng}$ as $N\to\infty$. 
\end{enumerate}

By \textsc{Step 3--5}, we conclude that $G_N\to G$ with respect to the norm $\lng |G(\xi)|^{2p} \rng^{\frac{1}{2p}} + \lng \|\frac{\partial G}{\partial\xi}(\xi)\|_{L^2}^{2p} \rng^{\frac{1}{2p}}$.
\end{proof}

%% file: bibliography-2.tex

\bigskip